\documentclass[final,reqno]{amsart}

\usepackage[bookmarks=false,draft=false,breaklinks,colorlinks]{hyperref}
\usepackage{mathtools,amssymb,mathrsfs}
\usepackage{fullpage,enumitem,here,braket}
\usepackage{showkeys,comment}
\usepackage[only llbracket,rrbracket,longmapsfrom]{stmaryrd}
\usepackage{tikz}
\usetikzlibrary{cd}

\numberwithin{equation}{subsection}
\numberwithin{figure}{subsection}
\allowdisplaybreaks

\usepackage{enumitem}
\setenumerate{label=(\arabic*),nosep}
\setitemize{nosep}
\newlist{clist}{enumerate}{1}
\setlist*[clist]{label=(\roman*), nosep}

\usepackage{cleveref}
\crefname{thm}{Theorem}{Theorems}
\crefname{dfn}{Definition}{Definitions}
\crefname{prp}{Proposition}{Propositions}
\crefname{lem}{Lemma}{Lemmas}
\crefname{cor}{Corollary}{Corollaries}
\crefname{clm}{Claim}{Claims}
\crefname{fct}{Fact}{Facts}
\crefname{rmk}{Remark}{Remarks}
\crefname{eg}{Example}{Examples}
\crefname{figure}{Figure}{Figures}
\crefname{table}{Table}{Tables}
\crefname{section}{\S\!}{\S\S\!}
\crefname{subsection}{\S\!}{\S\S\!}
\crefname{subsubsection}{\S\!}{\S\S\!}
\crefname{appendix}{Appendix}{Appendices}
\crefname{equation}{}{}

\usepackage{amsthm}
\theoremstyle{definition}
\newtheorem{thm}{Theorem}[subsection]
\newtheorem{dfn}[thm]{Definition}
\newtheorem{prp}[thm]{Proposition}
\newtheorem{lem}[thm]{Lemma}

\newtheorem{clm}[thm]{Claim}
\newtheorem{fct}[thm]{Fact}
\newtheorem{rmk}[thm]{Remark}
\newtheorem{eg}[thm]{Example}
\newtheorem*{rmk*}{Remark}

\newcommand{\ol}{\overline}
\newcommand{\ul}{\underline}
\newcommand{\wh}{\widehat}
\newcommand{\wt}{\widetilde}
\newcommand{\sm}{\setminus}
\newcommand{\bl}{\bullet}

\newcommand{\ve}{\varepsilon}
\newcommand{\pdd}{\partial}
\newcommand{\ceq}{\coloneqq} 
\newcommand{\bs}{\mathbin{\setminus}}
\newcommand{\dbl}{\bl,\bl}

\newcommand{\xrr}[1]{\xrightarrow{\, #1 \, }}
\newcommand{\inj}{\hookrightarrow}

\newcommand{\lto}{\longrightarrow}

\newcommand{\lsto}{\xrr{\sim}}
\newcommand{\linj}{\lhook\joinrel\longrightarrow}
\newcommand{\lsrj}{\relbar\joinrel\twoheadrightarrow}
\newcommand{\mto}{\mapsto}
\newcommand{\lmto}{\longmapsto}

\newcommand{\tboplus}{{\textstyle \bigoplus}}

\newcommand{\bbN}{\mathbb{N}}
\newcommand{\bbQ}{\mathbb{Q}}
\newcommand{\bbZ}{\mathbb{Z}}
\newcommand{\bbK}{\mathbb{K}}
\newcommand{\clD}{\mathcal{D}}
\newcommand{\clE}{\mathcal{E}}
\newcommand{\clH}{\mathcal{H}}
\newcommand{\clO}{\mathcal{O}}
\newcommand{\frg}{\mathfrak{g}}
\newcommand{\frS}{\mathfrak{S}}
\newcommand{\rmT}{\mathrm{T}}

\newcommand{\ac}{\textup{ac}}
\newcommand{\ch}{\textup{ch}}
\newcommand{\cl}{\textup{cl}}
\newcommand{\tfn}{\textup{fn}}
\newcommand{\tCE}{\textup{CE}}

\newcommand{\tuPoi}{\textup{Pois}}
\newcommand{\tAss}{\mathcal{A}\mathit{ssoc}}
\newcommand{\tCom}{\mathcal{C}\!\mathit{om}}
\newcommand{\tLie}{\mathcal{L}{\kern-0.05em}\mathit{ie}}
\newcommand{\tPoi}{\mathcal{P}\!\mathit{ois}}

\newcommand{\cOp}{\mathsf{Op}}
\newcommand{\cMod}{\mathsf{Mod}}

\newcommand{\od}{\ol{1}}
\newcommand{\vac}{\ket{0}} 

\newcommand{\tw}{\bullet \hspace{5pt} \bullet}
\newcommand{\twab}{\bullet \to \bullet}
\newcommand{\twba}{\bullet \leftarrow \bullet}
\newcommand{\thr}{\bullet \hspace{6pt} \bullet \hspace{6pt} \bullet}
\newcommand{\thbc}{\bullet \hspace{7pt} \bullet \to \bullet}
\newcommand{\thabc}{\bullet \to \bullet \to \bullet}

\newcommand{\abs}[1]{\left| #1 \right|}
\newcommand{\rst}[2]{\left. #1 \right|_{#2}}

\newcommand{\oPcl}[1]{\mathcal{P}^{\textup{cl}}_{#1}}
\newcommand{\oPchB}[1]{\mathop{\mathcal{P}^{\textup{ch}N_{\bullet}=N}_{#1}}}
\newcommand{\oPchK}[1]{\mathop{\mathcal{P}^{\textup{ch}N_K=N}_{#1}}}
\newcommand{\oPchW}[1]{\mathop{\mathcal{P}^{\textup{ch}N_W=N}_{#1}}}
\newcommand{\oPclB}[1]{\mathop{\mathcal{P}^{\textup{cl}N_{\bullet}=N}_{#1}}}
\newcommand{\oPclK}[1]{\mathop{\mathcal{P}^{\textup{cl}N_K=N}_{#1}}}
\newcommand{\oPclW}[1]{\mathop{\mathcal{P}^{\textup{cl}N_W=N}_{#1}}}
\newcommand{\oPfn}[1]{\mathop{\mathcal{P}^{\textup{fn}}_{#1}}}

\DeclareMathOperator{\gr}{gr}
\DeclareMathOperator{\id}{id}
\DeclareMathOperator{\pc}{\Pi}

\DeclareMathOperator{\MC}{MC}
\DeclareMathOperator{\Sh}{Sh}
\DeclareMathOperator{\sgn}{sgn}
\DeclareMathOperator{\frP}{\mathfrak{P}}

\DeclareMathOperator{\Der}{Der}
\DeclareMathOperator{\End}{End}
\DeclareMathOperator{\Hom}{Hom}

\DeclareMathOperator{\Res}{Res}

\DeclareMathOperator{\Span}{Span}

\DeclareMathOperator{\clL}{\mathcal{L}}
\DeclareMathOperator{\clS}{\mathcal{S}}
\DeclareMathOperator{\clT}{\mathcal{T}}

\DeclareMathOperator{\oP}{\mathcal{P}}
\DeclareMathOperator{\oQ}{\mathcal{Q}}
\DeclareMathOperator{\oAss}{\mathcal{A}\mathit{ssoc}}
\DeclareMathOperator{\oCom}{\mathcal{C}{\kern-0.1em}\mathit{om}}
\DeclareMathOperator{\oHom}{\mathcal{H}{\kern-0.1em}\mathit{om}}
\DeclareMathOperator{\oLie}{\mathcal{L}{\kern-0.5pt}\mathit{ie}}
\DeclareMathOperator{\oPoi}{\mathcal{P}\!\mathit{ois}}

\begin{document}

\title{Algebraic operad of SUSY Poisson vertex algebra}
\author{Yusuke Nishinaka, Shintarou Yanagida}
\date{01 May 2023} 
\address{Graduate School of Mathematics, Nagoya University.
 Furocho, Chikusaku, Nagoya, Japan, 464-8602.}
\email{m21035a@math.nagoya-u.ac.jp, yanagida@math.nagoya-u.ac.jp}
\thanks{Y.N. is supported by JSPS Research Fellowship for Young Scientists (No.\ 23KJ1120). 
 S.Y.\ is supported by JSPS KAKENHI Grant Number 19K03399.}
\keywords{}

\begin{abstract}
As a continuation of our study (Y.N., S.Y., arXiv:2209.14617) on the algebraic operad of SUSY vertex algebras, we introduce the SUSY coisson operad, which encodes the structures of SUSY Poisson vertex algebras. Our operad is a natural SUSY analogue of the operad encoding the structures of Poisson vertex algebras introduced by Bakalov, De Sole, Heluani and Kac (2019). We also give an embedding of the associated graded of the SUSY chiral operad into the SUSY coisson operad in the filtered case.
\end{abstract}

\maketitle
{\small \tableofcontents}

\setcounter{section}{-1}
\section{Introduction}\label{s:0}

This is a continuation of our study \cite{NY} of algebraic operads concerning SUSY vertex algebras.
There we introduced the superoperads $\oPchW{}$ and $\oPchK{}$ called the SUSY chiral operad, which encodes the structures of $N_W=N$ and $N_K=N$ SUSY vertex algebras in the sense of Heluani and Kac \cite{HK}. The main statement in \cite{NY} is the bijection
\begin{align*}
 \Hom_{\cOp}(\oLie,\oPchB{\Pi^{N+1}V})_{\od} \lsto 
 \{\text{$N_{\bl}=N$ SUSY vertex algebra structures on $(V,\nabla)$}\}, \quad
 \bl=W \text{ or } K. 
\end{align*}
for each supermodule $(V,\nabla)$ over a certain superalgebra $\clH_W$ or $\clH_K$, where $\oLie$ denotes the Lie operad and $\Pi$ denotes the parity change functor. These superoperads are natural SUSY extensions of the chiral operad $P^{\ch}$ introduced by Bakalov, De Sole, Heluani and Kac in \cite{BDHK}, which encodes the structure of vertex algebras.

In this note, we introduce the superoperads $\oPclW{V}$ and $\oPclK{V}$ (\cref{dfn:2P:Pcl}), called \emph{the SUSY coisson operads}. They encode the structures of SUSY Poisson vertex algebras. More precisely, there are bijections (\cref{thm:2P:NWPVA,thm:2P:NKPVA})
\begin{align}\label{eq:0P:oPcl}
 \Hom_{\cOp}(\oLie,\oPclB{\Pi^{N+1}V})_{\od} \lsto
 \{\text{$N=N_{\bl}$ SUSY Poisson vertex algebra structures on $(V,\nabla)$}\}, \quad
 \bl=W \text{ or } K.
\end{align} 
The word ``coisson'' is a synonym of ``Poisson vertex'', borrowed from \cite[2.7]{BD}. 

These superoperads are natural SUSY extensions of the operad $P^{\cl}$ in \cite[\S10]{BDHK} which encodes the structures of Poisson vertex algebras. Also, similarly as in the non-SUSY case \cite[\S10.4]{BDHK}, we have a filtration on the SUSY chiral operads $\oPchB{V}$ for a filtered supermodule $V=(V,\nabla)$, and for the associated graded $\gr \oPchB{V}$ we have a natural embedding (\cref{thm:3P:inj}) 
\begin{align}\label{eq:oP:emb}
 \gr \oPchB{V} \linj \oPclB{V} \quad \bl=W \text{ or } K.
\end{align}

Our construction of SUSY coisson operads basically follow the non-SUSY case in \cite[\S10]{BDHK}, and the bosonic part of the formulas are essentially the same with those in \cite{BDHK}. However, to treat the fermionic part nicely, we give non-trivial extensions in several points. See \cref{rmk:2P:Pcl}, \cref{eg:2P:comp} and the beginning of \cref{s:3P}, for example.

The approach to Poisson vertex algebra structures taken in loc.\ cit.\ is unique and non-standard in the following sense: It is known (see \cite{Fr} and \cite[13.3]{LV} for example) that a Poisson algebra structure on a linear space $V$ corresponds bijectively to an operad morphism $\oPoi \to \oHom_{V}$ from the Poisson operad $\oPoi$ to the endomorphism operad $\oHom_V$. Thus, one may guess that Poisson vertex algebra structures correspond to operad morphism from $\oPoi$ to some operad. The approach of \cite{BDHK} is to consider operad morphisms from $\oLie$ instead, as \eqref{eq:0P:oPcl} indicates.

In \cref{s:AP}, we give a remark on the above approach to Poisson-like algebra structures.
In \cite[\S10.5]{BDHK}, it is argued that there is a finite analogue $P^{\tfn}_V$ of the operad $P^{\cl}_V$ encoding the structures of Poisson algebra, and there is a bijection between two sets of operad morphisms for an even linear space $V$:
\[
 \Hom_{\cOp}(\oPoi,\oHom_V) \lsto \Hom_{\cOp}(\oLie,P^{\tfn}_V).
\]
Then, given a Poisson algebra $A=(V,\cdot,\{\,,\,\})$, we have two cohomology complexes $\frg^{\tfn}_A$ and $C_{\tPoi}(A,A)$, which arise from $P^{\tfn}$ and $\oPoi$, respectively. According to the work of Fresse \cite{Fr}, the latter complex has a bigrading $C_{\tPoi}^{\bl,\bl}(A,A)$. Now it is natural to ask whether the two complexes can be identified in a natural way. In \cref{thm:AP:fcc=pcb}, we show that the complex $\frg^{\tfn}_A$ arising from the finite operad $P^{\tfn}_V$ has a bigrading $(\frg^{\tfn}_A)^{\dbl}$, and the two bicomplexes coincide up to shift:
\[
 (\frg^{\tfn}_A)^{\dbl} \cong C_{\tPoi}^{\bl,\bl+1}.
\]
This statement can be seen as a strengthened form of \cite[Theorem 10.16]{BDHK}.

The operad $P^{\cl}$ in \cite{BDHK} was introduced to reduce the computation of the cohomology of vertex algebras to that of Poisson vertex algebras, as demonstrated in the series of works \cite{BDHK2,BDK20,BDK21,BDKV21}. We expect that a similar calculation can be made using our SUSY coisson operads for SUSY vertex algebras, but leave it for future work.

\subsubsection*{Organization}

\cref{s:1P} is a preliminary section. In \cref{ss:1P:st}, we briefly explain the notation and terminology of superobjects and superoperads. In \cref{ss:1P:WLCAVA} we recall the basics of SUSY vertex algebras from \cite{HK}. In \cref{ss:1P:SUSYPVA}, we recall the basic definitions on SUSY Poisson vertex algebras from \cite{HK}. In \cref{ss:1P:Pch}, we recall the SUSY chiral operad, the superoperad encoding the structure of a SUSY vertex algebra from \cite{NY}. 

\cref{s:2P} and \cref{s:3P} are the main body. In \cref{ss:2P:graph}, we cite from \cite{BDHK} the language of  graphs which will be used to construct our $N_W=N$ SUSY coisson operads $\oPclW{}$ in the first half of \cref{ss:2P:NW} (\cref{dfn:2P:Pcl}). In the latter half of \cref{ss:2P:NW}, we show that the superoperad $\oPclW{}$ does encode the structures of $N_W=N$ SUSY Poisson vertex algebras (\cref{thm:2P:NWPVA}). The arguments in \cref{ss:2P:NW} basically follow the line of the non-SUSY case \cite{BDHK}, but we give careful treatments on the fermionic part. The $N_K=N$ case can be treated similarly as the $N_W=N$ case, and we briefly state the result in \cref{ss:2P:NK}, omitting the proof.

In \cref{s:3P}, we establish a SUSY analogue \eqref{eq:oP:emb} of the embedding of superoperads \cite[\S10.4]{BDHK}. In the beginning, we introduce the notion of a filtration on an $N_W=N$ SUSY vertex algebras (\cref{dfn:3P:filter}), whose associated graded has a natural structure of a $N_W=N$ SUSY Poisson vertex algebra (\cref{prp:3P:grPVA}). In \cref{ss:3P:gr}, we cite the basics of gradings and filtrations on superoperads from \cite{BDHK}. In \cref{ss:3P:filt}, we introduce an operad filtration on our SUSY chiral operad $\oPchW{V}$ for a filtered $\clH_W$-supermodule $V$. In \cref{ss:3P:emb}, \cref{thm:3P:inj}, the associated graded superoperad $\gr\oPchW{V}$ is embedded in the SUSY coisson operad $\oPclW{\gr V}$ of the associated graded of $V$.

\cref{s:AP} gives an observation of the finite operad $P^{\tfn}_V$ in \cite{BDHK}. After the recollection on the Poisson cohomology bicomplex \cite{Fr} and the finite operad in \cref{ss:AP:pcb} and \cref{ss:AP:Pfn}, respectively, we show that the cohomology complex arising from $P^{\tfn}_V$ has a bicomplex structure, and it coincides with the Poisson cohomology bicomplex (\cref{thm:AP:fcc=pcb}).

\subsubsection*{Global notation}

\begin{itemize}
\item
The symbol $\bbN$ denotes the set $\{0,1,2,\dotsc\}$ of non-negative integers. 

\item 
For a positive integer $m$, the symbol $[m]$ denotes the set $\{1, 2,\dotsc, m\} \subset \bbZ$.
We also set $[0] \ceq \emptyset$.

\item 
The word `ring' or `algebra' means a unital associative one unless otherwise specified. 

\item 
Throughout the text, we work over a field $\bbK$ of characteristic $0$, and linear spaces, linear maps, algebras and algebra homomorphisms are defined over $\bbK$ unless otherwise stated.
\end{itemize}

\section{Preliminaries}\label{s:1P}

\subsection{Super terminology}\label{ss:1P:st}

We use the same super terminology as in \cite[\S 1.1]{NY}. 
Here we only recall some symbols for symmetric groups and their modules. 

\begin{itemize}
\item
For $n \in \bbN$, we denote by $\frS_n$ the $n$-th symmetric group with the convention $\frS_0 \ceq \{e\}$. 
We consider the group algebra $\bbK[\frS_n]$ as a purely even superalgebra. 
A linear superspace equipped with a left (resp.\ right) $\bbK[\frS_n]$-supermodule structure is just called a left (resp.\ right) $\frS_n$-supermodule. 

\item
Let $V$ be a linear superspace. 
For $n \in \bbN$, the linear superspace $V^{\otimes n}$ is a left $\frS_n$-supermodule by letting $\sigma \in \frS_n$ act on $v_1 \otimes \dotsb \otimes v_n \in V^{\otimes n}$ by
\begin{align*}
 \sigma(v_1\otimes \cdots\otimes v_n) \ceq 
 \prod_{\substack{1 \le i < j \le n \\ \sigma(i)>\sigma(j)}} (-1)^{p(v_i)p(v_j)} 
 \cdot v_{\sigma^{-1}(1)}\otimes \cdots\otimes v_{\sigma^{-1}(n)}.
\end{align*}
Here $p(v) \in \{0,1\}$ denotes the parity of $v \in V$. In what follows, we always regard $V^{\otimes n}$ as this left $\frS_n$-supermodule unless otherwise specified. 

\item
An $\frS$-supermodule $M=\bigl(M(n)\bigr)_{n \in \bbN}$ is a collection of right $\frS_n$-supermodules $M(n)$.
\end{itemize}

The notion of operads (see \cite{LV} for example) is naturally extended to the super setting. 
We call it a superoperad, following the terminology in \cite{BDHK}. 
More precisely, a superoperad is an $\frS$-module $\oP=\bigl(\oP(n)\bigr)_{n \in \bbN}$ equipped with an even element $1 \in \oP(1)_{\ol{0}}$ and a family of even linear maps $\gamma_\nu\colon \oP(m) \otimes \oP(n_1) \otimes \dotsb \otimes \oP(n_m) \to \oP(n)$ for each $m,n \in \bbN$ and $\nu = (n_1,\dotsc,n_m) \in \bbN^m$ with $n_1+\dotsb+n_m=n$ called the composition maps, satisfying some axiom. We express the composition of the operations $f \in \oP(m)$ and $g_1 \in \oP(n_1),\dotsc,g_m \in \oP(n_m)$ by $\gamma_\nu$ as
\[
 f \circ (g_1 \odot \dotsb \odot g_m) \ceq \gamma_\nu(f \otimes g_1 \otimes \dotsb \otimes g_m) \in \oP(n).
\]
We refer to \cite[Definition 1.1.1]{NY} for the detail.

\subsection{SUSY Lie conformal algebras and SUSY vertex algebras}\label{ss:1P:WLCAVA}

The notion of SUSY vertex algebras and the related algebraic structures are introduced by Heluani and Kac in \cite{HK}, and some of their aspects are reviewed in \cite[\S\S2.1--2.3, \S\S3.1--3.3]{NY}. 
We have two types of SUSY vertex algebras, the $N_W=N$ and the $N_K=N$ SUSY vertex algebras.
Here we only give the definitions of SUSY Lie conformal algebras and SUSY vertex algebras in the $N_W=N$ case, and briefly explain the definitions in the $N_K=N$ case. 

Throughout this \cref{ss:1P:WLCAVA}, we fix a non-negative integer $N$.

\begin{dfn}\label{dfn:1P:poly}
Let $A$ be an index set, and $\Lambda_\alpha=(\lambda_\alpha,\theta_\alpha^1,\ldots,\theta_\alpha^N)$ be a sequence of letters for each $\alpha\in A$. We denote by 
$\bbK[\Lambda_\alpha]_{\alpha \in A}$ the free commutative $\bbK$-superalgebra generated by even $\lambda_\alpha$ $(\alpha\in A)$ and odd $\theta_\alpha^i$ $(\alpha\in A, i\in [N])$, i.e, the $\bbK$-superalgebra generated by these elements with relations 
\begin{align}\label{eq:1P:polyW}
 \lambda_\alpha\lambda_\beta-\lambda_\beta\lambda_\alpha=0, \quad 
 \lambda_\alpha\theta_\beta^i-\theta_\beta^i\lambda_\alpha=0, \quad
 \theta_\alpha^i\theta_\beta^j+\theta_\beta^j\theta_\alpha^i=0 \quad
 (\alpha, \beta\in A, \ i, j\in[N]). 
\end{align}
Each $\Lambda_\alpha$ for $\alpha\in A$ is called a $(1|N)_W$-supervariable, and the $\bbK$-superalgebra $\bbK[\Lambda_\alpha]_{\alpha \in A}$ is called the $N_W=N$ polynomial superalgebra of the supervariables $(\Lambda_\alpha)_{\alpha\in A}$. 
\end{dfn}

In the case $A=[n]=\{1,\dotsc,n\}$ for $n \in \bbZ_{>0}$, we often denote the polynomial superalgebra by $\bbK[\Lambda_k]_{k=1}^n$ instead of $\bbK[\Lambda_k]_{k \in [n]}$. If $A$ consists of one element, then we suppress the subscript $\alpha \in A$ and denote the polynomial superalgebra by $\bbK[\Lambda]$.

For a $(1|N)_W$-supervariable $\Lambda=(\lambda,\theta^1,\ldots,\theta^N)$ and a subset $I=\{i_1<\cdots< i_r\}\subset [N]$, we denote $\theta^I \ceq \theta^{i_1}\cdots\theta^{i_r} \in \bbK[\Lambda]$. Also, we define $\sigma(I, J)\in\{0, \pm 1\}$ for $I,J \subset [N]$ by the relation $\theta^I\theta^J = \sigma(I,J) \theta^{I \cup J}$, and set $\sigma(I) \ceq \sigma(I, [N] \bs I)$. For $m \in \bbN$ and $I \subset [N]$, we denote $\Lambda^{m|I} \ceq \lambda^m \theta^I \in \bbK[\Lambda]$.

For a linear superspace $V$ and $(1|N)_W$-supervariables $\Lambda_\alpha$ ($\alpha \in A$), 
we denote $V[\Lambda_\alpha]_{\alpha \in A} \ceq \bbK[\Lambda_\alpha]_{\alpha\in A} \otimes_{\bbK} V$,
which regarded as a left $\bbK[\Lambda_\alpha]_{\alpha\in A}$-supermodule. 
We will often use it in the particular case $A=[n]$:
\[
 V[\Lambda_k]_{k=1}^n  = \bbK[\Lambda_k]_{k=1}^n \otimes_{\bbK} V.
\]


\begin{dfn}\label{dfn:1P:clHW}
Let $\clH_W$ be the free commutative $\bbK$-superalgebra generated by even $T$ and odd $S^i$ ($i\in[N]$), i.e., the $\bbK$-superalgebra generated by these elements with relations
\begin{align}\label{eq:1P:WTSi}
 TS^i-S^iT=0, \quad S^iS^j+S^jS^i=0\quad (i, j\in[N]).
\end{align}
For simplicity, we set
\begin{align}\label{eq:1P:Wnabla}
 \nabla \ceq (T, S^1, \ldots, S^N).
\end{align}
We also denote a linear superspace $V$ equipped with a left $\clH_W$-supermodule structure as 
\[
 (V,\nabla)=(V,T, S^1, \ldots, S^N),
\]
where $T$ is regarded as an even linear transformation on $V$ and $S^i$ as an odd linear transformation, satisfying the relations \eqref{eq:1P:WTSi}.
\end{dfn}

In the remaining of this \cref{ss:1P:WLCAVA}, let us fix a $(1|N)_W$-supervariable $\Lambda=(\lambda, \theta^1, \ldots, \theta^N)$.

Note that $\clH_W$ is isomorphic to $\bbK[\Lambda]$ as a superalgebra by the homomorphism defined by 
\begin{align*}
 T \lmto -\lambda, \quad S^i \lmto -\theta^i \quad (i \in [N]). 
\end{align*}
Since $\clH_W$ is a commutative superalgebra, we suppress the word `left' of an $\clH_W$-supermodule hereafter.

\begin{dfn}[{\cite[Definition 3.2.2]{HK}}]\label{dfn:1P:WLCA}
Let $(V, \nabla)=(V,T,S^1,\ldots,S^N)$ be an $\clH_W$-supermodule and $[\cdot_\Lambda\cdot]\colon V \otimes V \to V[\Lambda]$ be a linear map of parity $\ol{N}$. A triple $(V,\nabla,[\cdot_\Lambda\cdot])$ is called an $N_W=N$ SUSY Lie conformal algebra if it satisfies the following conditions: 
\begin{clist}
\item (sesquilinearity) For any $a, b\in V$, 
\begin{align*}
\begin{split}
&[Ta_\Lambda b]=-\lambda[a_\Lambda b], \hspace{43pt}
 [a_\Lambda Tb]=(\lambda+T)[a_\Lambda b], \\
&[S^ia_\Lambda b]=-(-1)^{N}\theta^i[a_\Lambda b], \quad 
 [a_\Lambda S^ib]=(-1)^{p(a)+\ol{N}}(\theta^i+S^i)[a_\Lambda b]\quad (i\in[N]). 
\end{split}
\end{align*} 

\item (skew-symmetry) For any $a,b\in V$, 
\begin{align}\label{eq:1P:LCA:ssym}
 [b_\Lambda a]=-(-1)^{p(a)p(b)+\ol{N}}[a_{-\Lambda-\nabla}b],
\end{align}
where we used $\nabla \ceq (T,S^1,\ldots,S^N)$ in \eqref{eq:1P:Wnabla}. 

\item (Jacobi identity) For any $a,b,c\in V$, 
\begin{align}\label{eq:1P:LCA:Jac}
 [a_{\Lambda_1}[b_{\Lambda_2}c]]
=(-1)^{(p(a)+\ol{N})\ol{N}}[[a_{\Lambda_1}b]_{\Lambda_1+\Lambda_2}c]
+(-1)^{(p(a)+\ol{N})(p(b)+\ol{N})}[b_{\Lambda_2}[a_{\Lambda_1}c]], 
\end{align}
where $\Lambda_1, \Lambda_2$ are $(1|N)_W$-supervariables. 
\end{clist}

For simplicity, we say $(V,\nabla)$, or more simply $V$, is an $N_W=N$ SUSY Lie conformal algebra.
The linear map $[\cdot_\Lambda\cdot]$ is called the $\Lambda$-bracket of the $\clH_W$-supermodule $(V, \nabla)$.  
\end{dfn}


For even linear transformations $F$ and $G$ on a linear superspace $V$, we define a linear map $\int_F^G d\Lambda\colon V[\Lambda] \to V$ by 
\begin{align}\label{eq:1P:intLv}
 \int_F^G d\Lambda\, \Lambda^{m|I}v\ceq\frac{\delta_{I, [N]}}{m+1}(G^{m+1}v-F^{m+1}v)
 \quad (m\in\bbN, \, I\subset [N], \, v\in V).
\end{align}
The linear map $\int_F^G d\Lambda$ has the parity $\ol{N}$. 
Also, if $V$ is a superalgebra (not necessarily unital nor associative), we define a linear map 
$\int_F^G d\Lambda \, a\colon V[\Lambda] \to V$ for $a\in V$ by
\begin{align*}
 \Bigl(\int_F^G d\Lambda \, a\Bigr)\Lambda^{m|I}v \ceq 
 \Bigl(\int_F^G d\Lambda\ \Lambda^{m|I}a\Bigr)v,
\end{align*}
where the term $\int_F^G d\Lambda\ \Lambda^{m|I}a$ in the right hand side is given by \eqref{eq:1P:intLv}. Then the linear map $\int_F^G d\Lambda \,a$ has the parity $p(a)$. 
Using this integral, we introduce:

\begin{dfn}[{\cite[Definition 3.3.15]{HK}}]
Let $(V,\nabla,[\cdot_\Lambda\cdot])=(V,T,S^1,\dotsc,S^N,[\cdot_\Lambda\cdot])$ be an $N_W=N$ SUSY Lie conformal algebra (\cref{dfn:1P:WLCA}) and $\mu\colon V \otimes V \to V$ be an even linear map. We denote $a b\ceq\mu(a\otimes b)$ for $a, b\in V$. A tuple $(V,\nabla,[\cdot_\Lambda\cdot],\mu)$ is called a non-unital $N_W=N$ SUSY vertex algebra (non-unital $N_W=N$ SUSY VA for short) if it satisfies the following conditions: 
\begin{clist}
\item For any $a, b\in V$, 
\begin{align*}
 T(ab)=(Ta)b+a(Tb), \quad S^i(ab)=(S^ia)b+(-1)^{p(a)}a(S^ib)\quad(i\in[N]). 
\end{align*}

\item (quasi-commutativity) For any $a, b\in V$, 
\begin{align*}
 ab-(-1)^{p(a)p(b)}ba=\int_{-T}^0d\Lambda[a_\Lambda b]. 
\end{align*}

\item (quasi-associativity) For any $a, b, c\in V$, 
\begin{align*}
  (ab)c-a(bc) = \Bigl(\int_0^Td\Lambda a\Bigr)[b_\Lambda c]
+(-1)^{p(a)p(b)}\Bigl(\int_0^Td\Lambda b\Bigr)[a_\Lambda c]. 
\end{align*}

\item (Wick formula) For any $a, b, c\in V$, 
\begin{align*}
 [a_\Lambda bc] = [a_\Lambda b]c + (-1)^{(p(a)+\ol{N})p(b)}b[a_\Lambda c] + 
                  \int_0^\lambda d\Gamma [[a_\Lambda b]_\Gamma c], 
\end{align*}
where $\Gamma$ is an additional $(1|N)_W$-supervariable. 
\end{clist}

For simplicity, we say $(V, \nabla)$, or more simply $V$, is a non-unital $N_W=N$ SUSY vertex algebra. The map $\mu$ is called the multiplication of $V$.
\end{dfn}

\begin{dfn}
A non-unital $N_W=N$ SUSY vertex algebra $V$ is called an $N_W=N$ SUSY vertex algebra if there exists an even element $\vac \in V$ such that $a\vac=\vac a=a$ for all $a\in V$. 
\end{dfn}

We close this subsection by a brief comment on $N_K=N$ SUSY vertex algebras.
All the definitions and arguments given for the $N_W=N$ case are valid for the $N_K=N$ case by replacing the relation \eqref{eq:1P:polyW} of the $(1|N)_W$-supervariables $\Lambda_\alpha=(\lambda_\alpha,\theta_\alpha^1,\dotsc,\theta_\alpha^N)$ with those of the $(1|N)_K$-supervariables:
\begin{align}\label{eq:1P:polyK} 
 \lambda_\alpha  \lambda_\beta -\lambda_\beta  \lambda_\alpha  = 0, \quad 
 \lambda_\alpha  \theta_\beta^i-\theta_\beta^i \lambda_\alpha  = 0, \quad 
 \theta_\alpha^i \theta_\beta^j+\theta_\beta^j \theta_\alpha^i =
 -2\delta_{\alpha,\beta}\delta_{i,j} \lambda_\alpha,
\end{align}
and replacing the commutative superalgebras $\clH_W$ by the following superalgebra $\clH_K$.

\begin{dfn}\label{dfn:1P:clHK}
Let $\clH_K$ be the $\bbK$-superalgebra generated by an even generator $T$ and odd generators $S^i$ $(i\in[N])$ with relations
\begin{align*}
 TS^i-S^iT = 0, \quad S^iS^j+S^jS^i = 2\delta_{i,j}T \quad (i,j \in [N]). 
\end{align*}
\end{dfn}

Then we have the notion of an $N_K=N$ SUSY Lie conformal algebra and that of an $N_K=N$ SUSY vertex algebra.

\subsection{SUSY Poisson vertex algebras}\label{ss:1P:SUSYPVA}

We cite from \cite{HK} the notion of \emph{SUSY Poisson vertex algebras}, the main object of this note.
Throughout this \cref{ss:1P:SUSYPVA}, we fix a non-negative integer $N$.

\begin{dfn}[{\cite[Definition 3.3.16]{HK}}]\label{dfn:1P:PVA}
Let $(V, \nabla, \{\cdot_\Lambda \cdot\})$ be an $N_W=N$ SUSY Lie conformal algebra and $\mu\colon V\otimes V\to V$ be an even linear map. We denote $ab\ceq \mu(a\otimes b)$ for $a, b\in V$. A tuple $(V, \nabla, \{\cdot_\Lambda \cdot\}, \mu)$ is called a \emph{non-unital $N_W=N$ SUSY Poisson vertex algebra} (\emph{non-unital $N_W=N$ SUSY PVA} for short) if it satisfies the following conditions: 
\begin{clist}
\item For any $a, b\in V$, 
\begin{align*}
 T(ab)=(Ta)b+a(Tb), \quad S^i(ab)=(S^ia)b+(-1)^{p(a)}a(S^ib) \quad (i\in[N]). 
\end{align*}

\item The linear map $\mu$ is commutative, i.e, $ba=(-1)^{p(a)p(b)}ab$ for any $a, b\in V$. 

\item The linear map $\mu$ is associative, i.e, $(ab)c=a(bc)$ for any $a, b, c\in V$. 

\item (Leibniz rule) For any $a, b, c\in V$,
\begin{align}\label{eq:1P:Leib}
\{a_\Lambda bc\}=\{a_\Lambda b\}c+(-1)^{(p(a)+\ol{N})p(b)}b\{a_\Lambda c\}. 
\end{align}
\end{clist}
\end{dfn}

\begin{dfn}
A non-unital $N_W=N$ SUSY PVA $V$ is called an \emph{$N_W=N$ SUSY Poisson vertex algebra}  if there exists $\vac\in V$ such that $a\vac=\vac a=a$ for any $a\in V$. 
\end{dfn}

In the case $N=0$, an $N_W=0$ SUSY PVA is nothing but a PVA in the ordinary sense (see \cite{K} and \cite[Chap.\ 16]{FBZ} for example).

\begin{eg}[{\cite[Proposition 3.1.9]{Y}}]
Let $P$ be an Poisson superalgebra of parity $q \in \bbZ_2$, i.e., a commutative superalgebra $P$ endowed with a Lie bracket $\{\cdot,\cdot\}\colon P \otimes P \to P$ of parity $q$ (see \cite[Definition 3.2.5]{HK}) satisfying
\begin{align*}
 \{a, bc\}=\{a, b\}c+(-1)^{(p(a)+q)p(b)}b\{a, c\}
\end{align*}
for $a, b, c\in P$. Then the $1|N$-superjet algebra $P^\clO=P^{\bbK[Z]}$ (see \cite[Proposition 1.3.5]{Y}) has the strucure of a SUSY PVA defined by 
\begin{align*}
\{a_\Lambda b\}\ceq \theta^{[N]}\{a, b\}
\end{align*}
for each $a, b\in P$. 
\end{eg}

The case $N_K=N$ is similarly introduced:

\begin{dfn}
Let $(V, \nabla, \{\cdot_\Lambda\cdot\})$ be an $N_K=N$ SUSY Lie conformal algebra and $\mu\colon V\otimes V\to V$ be an even linear map. 
We denote $ab\ceq \mu(a\otimes b)$ for $a, b\in V$. 
A tuple $(V, \nabla, \{\cdot_\Lambda\cdot\}, \mu)$ is called a non-unital $N_K=N$ SUSY Poisson vertex algebra if it satisfies the conditions (i)--(iv) in \cref{dfn:1P:PVA} replacing $(1|N)_W$-supervariable $\Lambda$ by $(1|N)_K$-supervariable. 
\end{dfn}

\subsection{The superoperad of SUSY vertex algebras}\label{ss:1P:Pch}

Let us briefly recall the $N_W=N$ SUSY chiral operad $\oPchW{V}$ defined in \cite{NY}. 
Fix a non-negative integer $N$.

\begin{dfn}\label{dfn:1P:O}
Let $n\in\bbZ_{>0}$, $[n] \ceq \{1,\dotsc,n\}$ and $Z_k=(z_k, \zeta_k^1, \ldots, \zeta_k^N)$ be a $(1|N)_W$-supervariable for each $k\in[n]$. We set
\begin{align*}
 z_{k,l} \ceq z_k-z_l, \quad \zeta_{k,l}^i \ceq \zeta_k^i-\zeta_l^i
\end{align*}
for $i\in[N]$ and $k, l\in[n]$. Also, we set $Z_{k,l} \ceq (z_{k,l},\zeta_{k,l}^1,\dotsc,\zeta_{k,l}^N)$ and $\pdd_{Z_k}\ceq (\pdd_{z_k}, \pdd_{\zeta_k^1}, \ldots, \pdd_{\zeta_k^N})$ for simplicity.
\begin{enumerate}
\item 
Let $\bbK[Z_k]_{k=1}^n$ be the $N_W=N$ polynomial superalgebra of the supervariables $Z_1,\dotsc,Z_n$ (\cref{dfn:1P:poly}).
Next, let $\clO_n^{\star} = \bbK[Z_k]_{k=1}^n[z_{k, l}^{-1}]_{1\le k<l\le n}$ be the localization of the $\bbK[Z_k]_{k=1}^n$ by the multiplicatively closed set generated by $\{z_{k,l} \mid 1 \le k<l \le n\}$.
Then we denote by $\clO_n^{\star\rmT}$ the subalgebra of $\clO_n^{\star}$ generated by 
$\{z_{k,l}^{\pm 1}\mid 1\le k<l\le n\} \cup \{\zeta_{k,l}^i \mid i \in [N], \, 1 \le k<l \le n\}$, i.e., 
\begin{align*}
 \clO_n^{\star\rmT} \ceq \bbK[z_{k,l}^{\pm1},\zeta_{k,l}^i \mid i \in [N], \ 1 \le k<l \le n].
\end{align*}
The superscript $\rmT$ indicates the translation-invariant part. 
See \cite[\S2.3, (2.3.11)]{NY} for the explanation.

\item 
Let $\clD_n^{\rmT}$ denote the subalgebra of $\End_\bbK(\clO_n^{\star})$ generated by
$\{z_{k,l}, \zeta_{k, l}^i\mid i\in[N], 1\le k<l\le n\} \cup 
 \{\pdd_{z_k}, \pdd_{\zeta_k^i}\mid i\in[N], k\in[n]\}$.  
\end{enumerate}

Also, by convention, we set $\clO_0^{\star\rmT} = \clD_0^{\rmT} \ceq \bbK$
\end{dfn}

Hereafter until the end of this \cref{ss:1P:Pch}, we fix an $\clH_W$-supermodule $(V, \nabla)$.
See \cref{dfn:1P:clHW} for the definition of $\clH_W$. 

Let $n\in\bbZ_{>0}$. 
The space $V^{\otimes n} \otimes \clO_n^{\star\rmT}$ carries the structure of a right $\clD_n^{\star\rmT}$-supermodule by letting $Z_{k, l}=(z_{k, l}, \zeta_{k, l}^1, \ldots, \zeta_{k, l}^N)$ act as
\begin{align*}
 (v \otimes f) \cdot     z_{k,l}   \ceq v \otimes f     z_{k,l}, \quad 
 (v \otimes f) \cdot \zeta_{k,l}^i \ceq v \otimes f \zeta_{k,l}^i
\end{align*}
and $\pdd_{Z_k}=(\pdd_{z_k}, \pdd_{\zeta_k^1}, \ldots, \pdd_{\zeta_k^N})$ act as
\begin{align*}
 (v \otimes f) \cdot \pdd_{z_k} &\ceq T^{(k)}v \otimes f - v \otimes \pdd_{z_k}f, \\ 
 (v \otimes f) \cdot \pdd_{\zeta_k^i} &\ceq 
 (-1)^{p(v)+p(f)}S^{(k)} v \otimes f - (-1)^{p(f)}v \otimes \pdd_{\zeta_k^i}f. 
\end{align*}
for each $v\in V^{\otimes n}$ and $f\in\clO_n^{\star\rmT}$. Here, for a linear transformation $\varphi$ on $V$, the symbol $\varphi^{(k)}$ denotes the linear transformation on $V^{\otimes n}$ defined by $\varphi^{(k)} \ceq \id_V \otimes \cdots \otimes \overset{k}{\varphi} \otimes \cdots \otimes \id_V$. 

Next, recall that $\bbK[\Lambda_k]_{k=1}^n$ has a right $\clH_W$-supermodule structure 
(see \cite[the paragraph before (2.2.4)]{NY}): 
\[
 a(\Lambda_1,\dotsc,\Lambda_n) \cdot T = 
 a(\Lambda_1,\dotsc,\Lambda_n)\Bigl(-\sum_{k=1}^n \Lambda_k\Bigr), \quad 
 a(\Lambda_1,\dotsc,\Lambda_n) \cdot S^i = 
 a(\Lambda_1,\dotsc,\Lambda_n)\Bigl(-\sum_{k=1}^n \theta_k^i\Bigr)
\]
for $a(\Lambda_1,\dotsc,\Lambda_n) \in \bbK[\Lambda_k]_{k=1}^n$.
Then we can form a linear superspace
\begin{align}\label{eq:1P:WVnL}
 V_{\nabla}[\Lambda_k]_{k=1}^n \ceq \bbK[\Lambda_k]_{k=1}^n \otimes_{\clH_W} V.
\end{align}
It is a right $\clD_n^{\rmT}$-supermodule by 
\begin{align*}
&(a(\Lambda_1,\ldots,\Lambda_n) \otimes v) \cdot z_{k,l}
 \ceq (\pdd_{\lambda_l}-\pdd_{\lambda_k})a(\Lambda_1,\ldots,\Lambda_n) \otimes v, \\
&(a(\Lambda_1, \ldots, \Lambda_n)\otimes v) \cdot \zeta_{k,l}^i
 \ceq (-1)^{p(a)+p(v)}(\pdd_{\zeta_l^i}-\pdd_{\zeta_k^i})a(\Lambda_1,\ldots,\Lambda_n)\otimes v, 
\end{align*}
and 
\begin{align*}
&(a(\Lambda_1,\ldots,\Lambda_n) \otimes v) \cdot \pdd_{z_k}
 \ceq -\lambda_ka(\Lambda_1,\ldots,\Lambda_n) \otimes v, \\
&(a(\Lambda_1,\ldots,\Lambda_n) \otimes v) \cdot \pdd_{\zeta_k^i}
 \ceq -(-1)^{p(a)+p(v)}\theta_k^i a(\Lambda_1,\ldots,\Lambda_n) \otimes v
\end{align*}
for each $a(\Lambda_1,\ldots,\Lambda_n) \in \bbK[\Lambda_k]_{k=1}^n$ and $v \in V$. 

\begin{dfn}
For an $\clH_W$-supermodule $(V, \nabla)$, we define an $\frS$-supermodule 
\begin{align*}
 \oPchW{V} \ceq \bigl(\oPchW{V}(n)\bigr)_{n \in \bbN}
\end{align*} 
as follows. First, for each $n \in \bbN$, we define a linear superspace $\oPchW{V}(n)$ by
\begin{align*}
 \oPchW{V}(n) \ceq 
 \Hom_{\clD_n^{\rmT}}(V^{\otimes n} \otimes \clO_n^{\star\rmT}, V_\nabla[\Lambda_k]_{k=1}^n).
\end{align*}
We denote its element $X \in \oPchW{V}(n)$ as 
\[
 X\colon V^{\otimes n} \otimes \clO_n^{\star\rmT} \lto V_{\nabla}[\Lambda_k]_{k=1}^n, \quad 
 v_1 \otimes \dotsb \otimes v_n \otimes f \lmto 
 X_{\Lambda_1,\dotsc,\Lambda_n}(v_1 \otimes \dotsb \otimes v_n \otimes f),
\]
emphasizing the supervariables $\Lambda_k$'s.
Second, for $\sigma \in \frS_n$ and $X \in \oPchW{V}(n)$, we define a linear map 
$X^\sigma\colon V^{\otimes n}\otimes \clO_n^{\star\rmT}\to V_\nabla[\Lambda_k]_{k=1}^n$ by
\begin{align*}
 X^{\sigma}(v_1 \otimes \dotsb \otimes v_n \otimes f) \ceq 
 X_{\sigma(\Lambda_1, \ldots, \Lambda_n)}(\sigma(v_1\otimes \cdots \otimes v_n)\otimes \sigma f), 
\end{align*}
where 
\begin{align*}
&\sigma(\Lambda_1, \ldots, \Lambda_n) \ceq 
 (\Lambda_{\sigma^{-1}(1)}, \ldots, \Lambda_{\sigma^{-1}(n)}), \\
&\sigma(v_1\otimes \cdots \otimes v_n) \ceq
 \prod_{\substack{1\le k<l\le n\\ \sigma(k)>\sigma(l)}}(-1)^{p(v_k)p(v_l)} \cdot 
 v_{\sigma^{-1}(1)}\otimes \cdots \otimes v_{\sigma^{-1}(n)}, \\
&(\sigma f)(Z_1, \ldots, Z_n) \ceq f(Z_{\sigma(1)}, \ldots, Z_{\sigma(n)}),
\end{align*}
Then, it follows that $X^{\sigma}\in \oPchW{V}(n)$, and as a consequence we have the $\frS$-supermodule $\oPchW{V}$.
\end{dfn}

Hereafter, for $m \in \bbZ_{>0}$ and $n \in \bbN$, we use the notation
\begin{align}\label{eq:1P:Nmn}
 \bbN^m_n \ceq \{(n_1,\dotsc,n_m) \in \bbN^m \mid n_1+\dotsb+n_m=n \}.
\end{align}
Also, let us recall from \cite[Lemma 2.2.3]{NY} that we have the linear isomorphism
\begin{align}\label{eq:2P:VNL=VL}
\begin{split}
 V_{\nabla}[\Lambda_k]_{k=1}^n &\lsto V[\Lambda_k]_{k=1}^{n-1} \\
 a(\Lambda_1,\dotsc,\Lambda_n) \otimes v &\lmto 
 a(\Lambda_1,\dotsc,\Lambda_{n-1},-\Lambda_1-\cdots-\Lambda_{n-1}-\nabla)v
\end{split}
\end{align}
for $a \in \bbK[\Lambda_k]_{k=1}^n$ and $v \in V$.

\begin{dfn}
Let $m,n \in \bbN$, $(n_1,\ldots,n_m) \in \bbN^m_n$, and set $N_j \ceq n_1+\cdots+n_j$ for $j \in [m]$ and $N_0 \ceq 0$. 
\begin{enumerate}
\item
For $Y_j \in \oPchW{V}(n_j)$, $j \in [m]$, we define a linear map
\begin{align*}
 Y_1 \odot \dotsb \odot Y_m \colon V^{\otimes n} \otimes \clO^{\star\rmT}_n \lto 
 \bigotimes_{j=1}^m V[\Lambda_k]_{k=N_{j-1}+1}^{N_j-1} \otimes \clO^{\star\rmT}_m
\end{align*}
by 
\begin{align*}
&(Y_1 \odot \cdots \odot Y_m)(v_1 \otimes \cdots \otimes v_n \otimes f) \\
&\ceq \pm (Y_1)_{\Gamma_1,\ldots,\Gamma_{N_1-1}}(w_1 \otimes f_1) \otimes \cdots 
 \otimes  (Y_m)_{\Gamma_{N_{m-1}+1},\ldots,\Gamma_{N_m-1}}(w_m \otimes f_m)
 \otimes \rst{f_0}{Z_k=Z_{N_j} (N_{j-1}< k \le N_j)}. 
\end{align*}
for $v_1, \ldots, v_n\in V$ and $f\in\clO^{\star\rmT}_n$. Here we denote by 
\begin{align*}
 (Y_j)_{\Lambda_1,\ldots,\Lambda_{n_j-1}}(w) \ceq 
 (Y_r)_{\Lambda_1,\ldots,\Lambda_{n_j-1},-\Lambda_1,\ldots,-\Lambda_{n_j-1}-\nabla}(w) 
 \quad (w \in V^ {\otimes n_j})
\end{align*} 
the element of $V[\Lambda_k]_{k=1}^{n_j-1}$ corresponding to $Y_j(w)$ by the isomorphism \eqref{eq:2P:VNL=VL}. 
Also, we set
\begin{align*}
&\pm  \ceq \prod_{1\le i<j\le m}(-1)^{p(w_i)p(Y_j)}, \quad
  w_j \ceq v_{N_{j-1}+1} \otimes \cdots \otimes v_{N_j}, \quad (j \in [m]), \\
&\Gamma_k \ceq \Lambda_k-\pdd_{Z_k} \quad (k \in [n] \bs \{N_1,\ldots,N_m\}),
\end{align*}
and
\begin{align*}
 f = f_0 f_1 \cdots f_m, \quad f_0 \in \clO_n^{\star\rmT}, \quad
 f_j = \prod_{N_{j-1}<k<l\le N_j}z_{k,l}^{-m_{k,l}^j} \quad (m_{k,l}^j \in \bbN)
\end{align*}
is a decomposition such that $f_0$ has no poles at $z_k=z_l$ ($N_{j-1}<k<l\le N_j$, $j \in [m]$). 

\item 
For $X \in \oPchW{V}(m)$ and $Y_j \in \oPchW{V}(n_j)$ with $j \in [m]$, let
\begin{align*}
 X \circ (Y_1 \odot \cdots \odot Y_m)\colon 
 V^{\otimes n} \otimes \clO^{\star\rmT}_n \lto V_\nabla[\Lambda_k]_{k=1}^n
\end{align*}
denote the linear map defined by the composition 
\[
 V^{\otimes n} \otimes \clO^{\star\rmT}_n 
 \xrr{Y_1 \odot \cdots \odot Y_m} 
 \bigotimes_{j=1}^m V[\Lambda_k]_{k=N_{j-1}+1}^{N_j-1} \otimes \clO^{\star\rmT}_m 
 \xrr{X_{\Lambda'_1,\ldots,\Lambda'_m}}
 V_\nabla[\Lambda_k]_{k=1}^n. 
\]
Here we set $\Lambda'_j \ceq \Lambda_{N_{j-1}+1}+\dotsb+\Lambda_{N_j}$ for each $j \in [m]$, 
and the symbol $X_{\Lambda'_1,\ldots,\Lambda'_m}$ stands for the linear map defined by
\begin{align*}
 X_{\Lambda_1',\ldots,\Lambda_m'}\colon a_1 v_1 \otimes \cdots \otimes a_m v_m \otimes f \lmto 
 \pm(a_1 \cdots a_m)X_{\Lambda'_1,\ldots,\Lambda'_m}(v_1 \otimes \cdots \otimes v_m \otimes f)
\end{align*}
for each $a_j \in \bbK[\Lambda_k]_{k=N_{j-1}+1}^{N_j-1}$, $v_j \in V$ ($j \in [m]$) 
and $f \in \clO^{\star\rmT}_m$ with the sign
\begin{align*}
 \pm \ceq \prod_{1\le i<j\le m} (-1)^{p(v_i)p(a_j)} \cdot \prod_{j=1}^m(-1)^{p(a_j)p(X)}. 
\end{align*} 
\end{enumerate}
\end{dfn}

The $\frS$-supermodule $\oPchW{V}$ carries the structure of a superoperad by letting
\begin{align*}
 X \otimes Y_1 \otimes \dotsb \otimes Y_m \lmto X \circ (Y_1 \odot \dotsb \odot Y_m)
\end{align*}
be the composition map and $\id_V\in \oPchW{V}(1)$ be the unit. 

For $X \in \oPchW{\Pi^{N+1}V}(2)_{\ol{1}}$, we define linear maps
$[\cdot_\Lambda \cdot]_X\colon V \otimes V \to V[\Lambda]$ and $\mu_X\colon V \otimes V \to V$ by
\begin{align}
\label{eq:1P:VALbraX}
&[a_\Lambda b]_X \ceq 
 (-1)^{p(a)(\ol{N}+\ol{1})}X_{\Lambda, -\Lambda-\nabla}(a\otimes b\otimes 1_\bbK), \\
\label{eq:1P:VAproX}
&\mu_X(a \otimes b) \ceq (-1)^{p(a)(\ol{N}+\ol{1})+\ol{1}} 
 \Res_\Lambda\bigl(\lambda^{-1}X_{\Lambda,-\Lambda-\nabla}(a \otimes b \otimes z_{1,2}^{-1})\bigr)
\end{align}
for each $a,b \in V$. Here $\Res_\Lambda(\lambda^{-1}-)\colon V[\Lambda]\to V$ is the residue map \cite[3.1.2]{HK}: 
\begin{align*}
 \Res_{\Lambda}(\lambda^{-1}\Lambda^{m|I}v) \ceq \delta_{m,0} \delta_{I,[N]} v \quad 
 (m \in \bbN, \, I \subset [N], \, v \in V). 
\end{align*}

Recall from \cite[Definition 1.3.1, Proposition 1.3.3]{NY} that for each superoperad $\oP$ 
we have a $\bbZ_{\ge -1}$-graded linear superspace
\begin{align}\label{eq:1P:L}
 L(\oP) \ceq \bigoplus_{n \ge -1}L^n(\oP), \quad 
 L^n(\oP) \ceq \{f \in \oP(n+1) \mid \forall \sigma\in\frS_{n+1}, \, f^\sigma=f\}
\end{align} 
equipped with a linear map $\square\colon L(\oQ)\otimes L(\oQ)\to L(\oQ)$,
and that there is a bijection
\begin{align}\label{eq:1P:MC}
 \MC\bigl(L(\oQ)\bigr) \ceq \{X \in L^1(\oQ) \mid X \square X = 0\} \lsto 
 \{\text{operad morphisms $\oLie \to \oQ$}\},
\end{align} 
where $[f,g] \ceq f \square g - (-1)^{p(f)p(g)} g \square f$ and $\oLie$ denotes the Lie operad. 
Now we have: 

\begin{thm}[{\cite[Theorem 2.3.15]{NY}}] \label{thm:1P:opNWVA}
Let $(V,\nabla)$ be an $\clH_W$-supermodule. 
\begin{enumerate}
\item 
For $X \in \MC\bigl(L\bigl(\oPchW{\Pi^{N+1}V}\bigr)\bigr)_{\ol{1}}$, the pair $([\cdot_\Lambda\cdot]_X, \mu_X)$ is a non-unital $N_W=N$ SUSY vertex algebra structure on $(V, \nabla)$. 

\item 
The map $X \mto ([\cdot_\Lambda\cdot]_X,\mu_X)$ gives a bijection
\begin{align*}
 \MC\bigl(L\bigl(\oPchW{\Pi^{N+1}V}\bigr)\bigr)_{\ol{1}} \lsto 
 \{\text{non-unital $N_W=N$ SUSY VA structures on $(V, \nabla)$}\}. 
\end{align*}
\end{enumerate}
\end{thm}

\section{The superoperad of SUSY Poisson vertex algebras}\label{s:2P}

In this section we introduce an algebraic operad of SUSY Poisson vertex algebras, which is a natural SUSY analogue of the operad $P^{\cl}$ of Poisson vertex algebras in \cite[\S10]{BDHK}. 

Hereafter we use the following symbols.
\begin{itemize}
\item
For a set $S$, we denote by $\frP(S)$ the power set of $S$.
\item
For a finite set $S$, we denote by $\# S$ the number of elements in $S$.
\end{itemize}

\subsection{Graphs and quivers}\label{ss:2P:graph}

We cite from \cite[\S9]{BDHK} some language of graphs and the associated cooperad which will be used to construct our operad of SUSY Poisson vertex algebras. 

A graph is a triple $G=(G_0,G_1,f)$, where $G_0,G_1$ are sets and $f\colon G_1 \to \frP(G_0)$ is a map such that $\#f(\alpha)=1,2$ for any $\alpha \in G_1$. Each element of $G_0$ and $G_1$ is called a vertex and an edge of $G$, respectively. Thus, in the standard language of graph theory, a graph in this note means an undirected graph with vertex set $G_0$ and edge set $G_1$ that allows multiple edges and edge loops.

\begin{dfn}\label{dfn:2P:graph}
For a graph $G=(G_0,G_1,f)$, we use the following notations: 
\begin{itemize}
\item 
For an edge $\alpha \in E$, we write
\begin{tikzcd}
 \underset{i}{\bullet} \arrow[r, "\alpha"', dash] & \underset{j}{\bullet}
\end{tikzcd}
to indicate $f(\alpha)=\{i,j\}$. 

\item
For $l \in \bbZ_{\ge 2}$, we set 
\begin{align*}
&G_l\ceq \{\alpha = (\alpha_1,\dotsc,\alpha_l)\in G_1^l \mid 
 f(\alpha_k) \cap f(\alpha_{k+1}) \ne \emptyset \ (k=1,\dotsc,l-1)\} .
\end{align*}
An element $\alpha \in G_l$ for $l \in \bbN$ is called a path of length $l$.

\item 
For vertices $i,j \in G_0$, we define
\begin{align*}
 G(i,j) \ceq \bigcup_{l \in \bbN}
 \bigl\{\alpha = (\alpha_1,\dotsc,\alpha_l) \in G_l \mid i \in f(\alpha_1), \, j \in f(\alpha_l)\bigr\}.
\end{align*}
Here we used the convention $f(i) \ceq i$ for each $i \in G_0$, by which one has $G(i,i) \ne \emptyset$. 
Each element of $G(i,j)$ is called a path from $i$ to $j$. 

\item
A path $\alpha \in G(i,i)$ of length $l \ge 1$ is called a cycle. 
A cycle of length $1$ is called a loop.  
If $G$ has no cycle, then $G$ is called acyclic. 

\item 
For vertices $i,j\in G_0$, we say that $i,j$ are connected if $G(i,j) \ne \emptyset$. 
This defines an equivalence relation on $G_0$. 
An equivalent class of $G_0$ is called a connected component of $G$. 
\end{itemize}
\end{dfn}

Next, we introduce the notion of $n$-graphs for $n \in \bbN$.

\begin{dfn}\label{dfn:2P:nGra}
Let $n \in \bbN$, and $G=(G_0,G_1,f_G)$ be a graph.
\begin{enumerate}
\item 
$G$ is called an $n$-graph if $G_0=[n]=\{1,\dotsc,n\}$. 

\item \label{dfn:2P:nGra:2}
Assume $G$ is an $n$-graph. Then, for $\sigma \in \frS_n$, let $\sigma G$ denote the $n$-graph $\sigma G=([n], G_1, f)$ with 
\begin{align*}
 f\colon G_1 \lto \frP([n]), \quad \alpha \lmto \sigma f_G(\alpha). 
\end{align*}
Here and hereafter, for $J=\{j_1,\dotsc,j_r\}\subset [n]$, 
we set $\sigma J \ceq \{\sigma(j_1),\dotsc,\sigma(j_r)\}$. 
By this action, the set of all $n$-graphs is a left $\frS_n$-set. 
\end{enumerate}
\end{dfn}

The following lemma will be used in \cref{ss:2P:NW}. The proof is straightforward. 

\begin{lem}\label{lem:2P:graph}
Let $n \in \bbN$, $\sigma \in \frS_n$, and $G=(G_0=[n],G_1,f)$ be an $n$-graph. 
\begin{enumerate}
\item \label{i:lem:2P:graph:1}
For $i,j \in [n]$ and $\alpha \in G(i,j)$, we have $\alpha \in (\sigma G)(\sigma(i),\sigma(j))$. 

\item \label{i:lem:2P:graph:2}
For $i,j \in [n]$, if $i,j$ are connected in $G$, then $\sigma(i),\sigma(j)$ are connected in $\sigma G$. 

\item \label{i:lem:2P:graph:3}
The $n$-graph $G$ is acyclic if and only if $\sigma G$ is acyclic.

\item \label{i:lem:2P:graph:4}
If $I^a \subset [n]=G_0$ is a connected component of $G$, 
then $\sigma I^a$ is a connected component of $\sigma G$.  
\end{enumerate}
\end{lem}

We also need the notion of a quiver.

\begin{dfn}\label{dfn:2P:quiver}
A quiver is a directed graph, i.e., a tuple $Q=(Q_0,Q_1,s,t)$ where $Q_0,Q_1$ are sets and $s,t\colon Q_1 \to Q_0$ are maps. Each element of $Q_0$ and $Q_1$ is called a vertex and an edge of $Q$ respectively. 
For a quiver $Q=(Q_0,Q_1,s,t)$, we use the following notations: 
\begin{itemize}
\item 
For an edge $\alpha \in Q$, we write
\begin{tikzcd}
 \underset{i}{\bullet} \arrow[r, "\alpha"'] & \underset{j}{\bullet}
\end{tikzcd}
to indicate $s(\alpha)=i$ and $t(\alpha)=j$. 

\item 
For $l \in \bbZ_{\ge 2}$, we set
\begin{align*}
&Q_l \ceq \{\alpha = (\alpha_1,\dotsc,\alpha_l) \in Q_1^l \mid 
            s(\alpha_{k+1})=t(\alpha_{k}) \, (k=1,\dotsc,l-1)\}. 
\end{align*}
An element $\alpha \in Q_l$ for $l \in \bbN$ is called a directed path of length $l$.

\item 
For vertices $i,j\in I$, we define
\begin{align*}
 Q (i,j) \ceq \bigcup_{l\in\bbN} 
 \bigl\{\alpha=(\alpha_1,\dotsc,\alpha_l) \in Q_l \mid s(\alpha_1)=i, \, t(\alpha_l)=j\bigr\}.
\end{align*}
Here we used the convention $s(i) = t(i) \ceq i$ for each $i \in Q_0$, by which one has $Q(i,i) \ne \emptyset$. An element of $Q(i,j)$ is called a directed path from $i$ to $j$. 

\item 
A directed path $\alpha\in Q(i,i)$ of length $l \ge 1$ is called a directed cycle. A directed cycle of length $1$ is called a loop. 


\item 
Let $f_Q\colon Q_1 \to \frP(Q_0)$ denote the map defined by $f(\alpha)\ceq \{s(\alpha), t(\alpha)\}$, then we have a graph $\ul{Q} \ceq (Q_0,Q_1,f_Q)$. The graph $\ul{Q}$ is called the underlying graph of $Q$.

\end{itemize}
\end{dfn}

For example, $Q=([5],Q_1,s,t)$ with $Q_1=\{\alpha_1,\alpha_2,\alpha_3\}$, $s(\alpha_1)=1$, $t(\alpha_1)=2$, $s(\alpha_2)=4$, $t(\alpha_2)=1$, $s(\alpha_3)=5$, $t(\alpha_1)=4$ is a quiver depicted as 
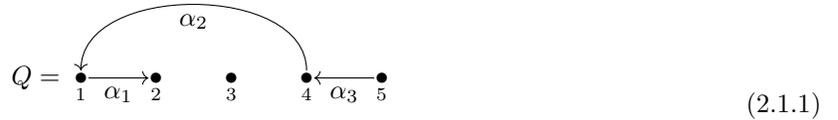
\begin{align}\label{exm:2P:5graph}
\begin{tikzpicture}
\draw (-0.6,0.13) node {$Q=$};
\draw ( 0.0,0.00) node {$\underset{1}{\bullet}$}; 
\draw ( 1.0,0.00) node {$\underset{2}{\bullet}$}; 
\draw ( 2.0,0.00) node {$\underset{3}{\bullet}$}; 
\draw ( 3.0,0.00) node {$\underset{4}{\bullet}$}; 
\draw ( 4.0,0.00) node {$\underset{5}{\bullet}$};
\draw[->] (0.1,0.13) to node[below]{$\alpha_1$} (0.9,0.13); 
\draw[->] (3.0,0.23) to [out=90, in=90] node[below]{$\alpha_2$} (0,0.23);
\draw[->] (3.9,0.13) to node[below]{$\alpha_3$} (3.1,0.13);
\end{tikzpicture}
\end{align}
The underlying graph $\ul{Q}$ is an acyclic graph depicted as
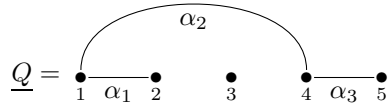
\begin{align*}
\begin{tikzpicture}
\draw (-0.6, 0.13) node {$\ul{Q}=$};
\draw (0, 0) node {$\underset{1}{\bullet}$}; \draw (1, 0) node {$\underset{2}{\bullet}$}; 
\draw (2, 0) node {$\underset{3}{\bullet}$}; \draw (3, 0) node {$\underset{4}{\bullet}$}; 
\draw (4, 0) node {$\underset{5}{\bullet}$};
\draw[-] (0.1,0.13) to node[below]{$\alpha_1$} (0.9,0.13); 
\draw[-] (3.0,0.23) to [out=90, in=90] node[below]{$\alpha_2$} (0,0.23);
\draw[-] (3.9,0.13) to node[below]{$\alpha_3$} (3.1,0.13);
\end{tikzpicture}
\end{align*}

Now we introduce a set-theoretic cooperad using quivers.
We need a left $\frS$-set, i.e. a sequence $\bigl(\oQ(n)\bigr)_{n \in \bbN}$ consisting of sets $\oQ(n)$ equipped with left $\frS_n$-action.

\begin{dfn}\label{dfn:2P:ngr}
Let $n \in \bbN$ and $Q=(Q_0,Q_1,s_Q,t_Q)$. 
\begin{enumerate}
\item 
$Q$ is called an $n$-quiver if $Q_0 = [n]$. 

\item \label{i:dfn:2P:ngr:2}
Assume $Q$ is an $n$-quiver. Then, for $\sigma\in\frS_n$, let $\sigma Q$ denote the $n$-quiver $\sigma Q=([n], Q_1, \sigma s_Q, \sigma t_Q)$ with $(\sigma s_Q)(\alpha) \ceq \sigma(s_Q(\alpha))$ and $(\sigma t_Q)(\alpha) \ceq \sigma(t_Q(\alpha))$ for each edge $\alpha \in Q_1$.
The set of all $n$-quivers is a left $\frS_n$-set by this action. 
\end{enumerate}
\end{dfn}

For the cyclic permutation $\sigma=(1,4,5)\in\frS_5$ and the $5$-quiver $Q$ in \cref{exm:2P:5graph}, we have
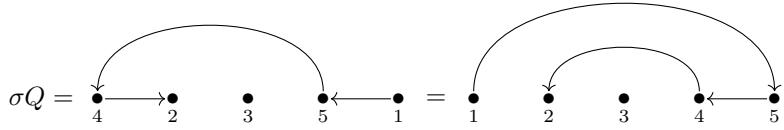
\begin{align*}
\begin{tikzpicture}
\draw (-0.75, 0.13) node {$\sigma Q=$}; 
\draw (0, 0) node {$\underset{4}{\bullet}$}; \draw (1, 0) node {$\underset{2}{\bullet}$}; 
\draw (2, 0) node {$\underset{3}{\bullet}$}; \draw (3, 0) node {$\underset{5}{\bullet}$}; 
\draw (4, 0) node {$\underset{1}{\bullet}$}; 
\draw[->] (0.1, 0.13) to (0.9,0.13); 
\draw[->] (3, 0.23) to[out=90, in=90] (0, 0.23);
\draw[->] (3.9, 0.13) to (3.1, 0.13);
\draw (4.5, 0.13) node {$=$}; 
\draw (5, 0) node {$\underset{1}{\bullet}$}; \draw (6, 0) node {$\underset{2}{\bullet}$}; 
\draw (7, 0) node {$\underset{3}{\bullet}$}; \draw (8, 0) node {$\underset{4}{\bullet}$}; 
\draw (9, 0) node {$\underset{5}{\bullet}$};
\draw[->] (8.9, 0.13) to (8.1, 0.13); 
\draw[->] (5, 0.23) to[out=90, in=90] (9, 0.23); 
\draw[->] (8, 0.23) to[out=90, in=90] (6, 0.23); 
\end{tikzpicture}
\end{align*}

The proof of the following lemma is straightforward. 

\begin{lem}\label{lem:2P:dgraph}
Let $n \in \bbN$, $\sigma \in \frS_n$, and $Q=([n],Q_1,s_Q,t_Q)$ be an $n$-quiver. 
\begin{enumerate}
\item \label{i:lem:2P:dgraph:1}
For $i,j \in [n]$ and $\alpha \in Q(i,j)$, we have $\alpha \in (\sigma Q)(\sigma(i),\sigma(j))$.


\item \label{i:lem:2P:dgraph:2}
The underlying graph of $\sigma Q$ is equal to $\sigma\ul{Q}$ in \cref{dfn:2P:nGra} \ref{dfn:2P:nGra:2}.

\item \label{i:lem:2P:dgraph:3}
For $\alpha \in Q_1$, let $Q \bs \alpha = ([n],E,s,t)$ be the quiver obtained from $Q$ by removing the edge $\alpha$, i.e., $E=Q_1 \bs \{\alpha\}$, $s=\rst{s_Q}{E}$ and $t=\rst{t_Q}{E}$.  Then we have $\sigma(Q \bs \alpha)=(\sigma Q) \bs \alpha$.
\end{enumerate}
\end{lem}

\begin{dfn}\label{dfn:2P:oQ}
For $n\in \bbN$: 
\begin{enumerate}
\item 
Let $\oQ(n)$ be the set of all $n$-quivers without loops.  

\item 
We denote by $\oQ_\ac(n)$ the set consisting of $Q \in \oQ(n)$ such that the underlying graph $\ul{Q}$ is acyclic in the sense of \cref{dfn:2P:graph}. 
\end{enumerate}
Note that $\oQ(0)=\oQ_\ac(0)=\{\emptyset\}$, where $\emptyset=(\emptyset,\emptyset,\emptyset,\emptyset)$ denotes the empty quiver.
\end{dfn}

Next we cite from \cite[9.1]{BDHK} the cocomposition maps on the left $\frS$-set $\oQ = \bigl(\oQ(n)\bigr)_{n \in \bbN}$, by which we have a set-theoretic cooperad $\oQ$.
Recall the notation $\bbN^m_n \ceq \{(n_1,\dotsc,n_m) \in \bbN^m \mid n_1+\dotsb+n_m=n \}$ from \eqref{eq:1P:Nmn}.

\begin{dfn}\label{dfn:2P:Delta}
Let $m \in \bbZ_{>0}$, $n \in \bbN$ and $\nu=(n_1,\dotsc,n_m) \in \bbN^m_n$. 
We set $N_k \ceq n_1+\dotsb+n_k$ for each $k \in [m]$ and $N_0 \ceq 0$. 
For $Q=([n],E,s,t) \in \oQ(n)$: 
\begin{enumerate}
\item 
We define an $n_k$-quiver $\Delta_k^\nu(Q) = ([n_k],E_k,s_k,t_k) \in \oQ(n_k)$ for each $k \in [m]$ by
\begin{align*}
&E_k\ceq \{\alpha \in E \mid s(\alpha),t(\alpha)\in \{N_{k-1}+1,\dotsc,N_k\}\}, \\
&s_k(\alpha) \ceq s(\alpha)-N_{k-1}, \quad t_k(\alpha) \ceq t(\alpha)-N_{k-1}. 
\end{align*}

\item 
We define an $m$-quiver $\Delta_0^\nu(Q)=([m],E_0,s_0,t_0) \in \oQ(m)$ by
\begin{align*}
&E_0 \ceq \{\alpha \in E \mid \{s(\alpha),t(\alpha)\} \not\subset \{N_{k-1}+1,\dotsc,N_k\} \ (k\in[m])\}, \\
&s_0(\alpha) \ceq i \quad \text{if} \ s(\alpha) \in \{N_{i-1}+1,\dotsc,N_i\}, \quad 
 t_0(\alpha) \ceq j \quad \text{if} \ t(\alpha) \in \{N_{j-1}+1,\dotsc,N_j\}. 
\end{align*}
\end{enumerate}
The map
\begin{align*}
 \Delta^\nu\colon \oQ(n) \lto \oQ(m) \times \oQ(n_1) \times \dotsb \times \oQ(n_m), \quad
 Q \lmto \bigl(\Delta_0^\nu(Q),\Delta_1^\nu(Q),\dotsc,\Delta_m^\nu(Q)\bigr)
\end{align*}
is called the cocomposition of $n$-quivers. 
\end{dfn}

\begin{eg}\label{eg:2P:gracomp}
Let $\nu=(3, 2, 2) \in \bbN^3_7$, and $Q$ be the $7$-quiver depicted as 
\begin{align*}
\begin{tikzpicture}
\draw (0.4,0.1) node{$Q=$};
\draw (1.0,0.0) node{$\underset{1}{\bullet}$}; 
\draw (2.0,0.0) node{$\underset{2}{\bullet}$}; 
\draw (3.0,0.0) node{$\underset{3}{\bullet}$}; 
\draw (4.0,0.0) node{$\underset{4}{\bullet}$}; 
\draw (5.0,0.0) node{$\underset{5}{\bullet}$}; 
\draw (6.0,0.0) node{$\underset{6}{\bullet}$}; 
\draw (7.0,0.0) node{$\underset{7}{\bullet}$}; 
\draw [->] (2.1,0.13) to node[below]{$\beta$}  (2.9,0.13); 
\draw [->] (3.1,0.13) to node[below]{$\gamma$} (3.9,0.13); 
\draw [->] (6.1,0.13) to node[below]{$\ve$}    (6.9,0.13); 
\draw [->] (1.05, 0.25) to [out=80, in=100] node[below]{$\alpha$} (2.95, 0.25);
\draw [->] (3.05, 0.25) to [out=30, in=150] node[below]{$\delta$} (6.95, 0.25);
\draw [->] (4.95,-0.20) to [out=210,in=330] node[above]{$\zeta$}  (1.05,-0.20);
\end{tikzpicture}
\end{align*}
Then we have 
\begin{align*}
\begin{tikzpicture}
\draw (0,0.1) node {$\Delta^\nu_1(Q)=$}; 
\draw (1,0.0) node {$\underset{1}{\bullet}$}; 
\draw (2,0.0) node {$\underset{2}{\bullet}$}; 
\draw (3,0.0) node {$\underset{3}{\bullet}$}; 
\draw [->] (2.1,0.13) to node[below]{$\beta$} (2.9,0.13); 
\draw [->] (1.0,0.23) to [out=90, in=90] node[below]{$\alpha$} (3,0.23);
\end{tikzpicture}
\qquad 
\begin{tikzpicture}
\draw (0,0.1) node {$\Delta^\nu_2(Q)=$}; 
\draw (1,0.0) node {$\underset{1}{\bullet}$}; 
\draw (2,0.0) node {$\underset{2}{\bullet}$}; 
\end{tikzpicture}
\qquad
\begin{tikzpicture}
\draw (0,0.1) node {$\Delta^\nu_3(Q)=$}; 
\draw (1,0.0) node {$\underset{1}{\bullet}$}; 
\draw (2,0.0) node {$\underset{2}{\bullet}$}; 
\draw [->] (1.1,0.13) to node[below]{$\ve$} (1.9,0.13); 
\end{tikzpicture}
\end{align*} 
\begin{align*}
\begin{tikzpicture}
\draw (0,0.1) node {$\Delta^\nu_0(Q)=$};
\draw (1,0.0) node {$\underset{1}{\bullet}$}; 
\draw (2,0.0) node {$\underset{2}{\bullet}$}; 
\draw (3,0.0) node {$\underset{3}{\bullet}$};
\draw [->] (1.1,0.18) to node[above]{$\gamma$} (1.9,0.18); 
\draw [->] (1.9,0.08) to node[below]{$\zeta$} (1.1,0.08); 
\draw [->] (1.0,0.23) to [out=90, in=90] node[below]{$\delta$} (3,0.23); 
\end{tikzpicture}
\end{align*}
\end{eg}

\begin{dfn}[{\cite[Definition 9.4]{BDHK}}]\label{dfn:2P:extconn}
Let $m \in \bbZ_{>0}$, $n \in \bbN$ and $\nu=(n_1,\dotsc,n_m) \in \bbN^m_n$. 
We set $N_k \ceq n_1+\dotsb+n_k$ for each $k\in[m]$ and $N_0\ceq 0$. 
\begin{enumerate}
\item 
For $Q=([n],E,s,t) \in \oQ(n)$, $i \in [m]$ and $j \in [n]$, we say that the vertex $i$ is externally connected to the vertex $j$ in $\Delta^\nu(Q)$ if there exists a path $(\alpha_1,\dotsc,\alpha_r) \in \ul{\Delta_0^\nu(Q)}(i,k)$ in the underlying graph $\ul{\Delta_0^\nu(Q)}$, where $k \in [m]$ is such that $j \in \{N_{k-1}+1,\dotsc,N_k\}$, satisfying the following conditions: 
\begin{clist}
\item There is no overlap between the edges $\alpha_1,\dotsc,\alpha_r$.

\item There is $p \in [r]$ such that $j \in \{s(\alpha_p),t(\alpha_p)\} \subset[n]=Q_0$.
\end{clist}

\item 
For $Q \in \oQ(n)$ and $j \in [n]$, let $\clE_Q^\nu(j)$ denote the set consisting of all $i \in [m]$ which are externally connected to $j$ in $\Delta^\nu(Q)$. 
\end{enumerate}
\end{dfn}

\begin{eg}
For $\nu=(3,2,2)$ and $Q$ in \cref{eg:2P:gracomp}, we have
\begin{align*}
&\clE^\nu_Q(1)=\{1,2,3\}, \quad \clE^\nu_Q(2)=\emptyset, \quad \clE^\nu_Q(3)=\{1,2,3\}, \\
&\clE^\nu_Q(4)=\{1,2,3\}, \quad \clE^\nu_Q(5)=\{1,2,3\}, \quad
 \clE^\nu_Q(6)=\emptyset, \quad \clE^\nu_Q(7)=\{1,2\}. 
\end{align*}
\end{eg}

\subsection{The superoperad of \texorpdfstring{$N_W=N$}{NW=N} SUSY Poisson vertex algebras}
\label{ss:2P:NW}

In this subsection we introduce the superoperad $\oPclW{}$ of $N_W=N$ SUSY Poisson vertex algebras (\cref{dfn:2P:Pcl}), and prove in \cref{thm:2P:NWPVA} that a structure of $N_W=N$ SUSY Poisson vertex algebra corresponds bijectively to an odd Lie algebra structure on the operad $\oPclW{}$. 
As before, we fix a non-negative integer $N$, and use the notation $\nabla=(T,S^1,\dotsc,S^N)$.

We first introduce the underlying $\frS$-supermodule of our superoperad, following the argument of the non-SUSY case in \cite[\S 10.2]{BDHK}. Let $\Lambda_k=(\lambda_k,\theta_k^1,\dotsc,\theta_k^N)$ be a $(1|N)_W$-supervariable for each $k \in \bbZ_{>0}$.

\begin{dfn}\label{dfn:2P:Pcl}
Let $V=(V,\nabla)$ be an $\clH_W$-supermodule and $n \in \bbN$. 
Recall from \eqref{eq:1P:WVnL} the linear superspace $V_{\nabla}[\Lambda_k]_{k=1}^n$, 
and the left $\frS_n$-set $\oQ(n)$ of $n$-quivers without loops from \cref{dfn:2P:oQ}.
We denote by $\bbK \oQ(n)$ the $\bbK$-linear space with basis $\oQ(n)$.
Now we define $\oPclW{V}(n)$ to be the linear superspace of linear maps
\[
 X\colon V^{\otimes n} \otimes \bbK\oQ(n) \lto V_\nabla[\Lambda_k]_{k=1}^n
\]
satisfying the following conditions for each $Q \in \oQ(n)$. 
Let us denote 
\begin{align}\label{eq:2P:Pcl-X}
 X^Q \ceq X(-\otimes Q)\colon V^{\otimes n} \lto V_\nabla[\Lambda_k]_{k=1}^n.
\end{align} 
\begin{clist}
\item (cycle relations) 
\begin{itemize}
\item
If $Q\in\oQ(n)\bs \oQ_{\ac}(n)$, i.e., the underlying graph $\ul{Q}$ (\cref{dfn:2P:quiver}) contains a cycle, then we have $X^Q=0$. 

\item 
For a directed cycle $(\alpha_1,\dotsc,\alpha_r)$ of $Q$, we have $\sum_{q=1}^r X^{Q \bs \alpha_q}=0$. 
See \cref{lem:2P:dgraph} for the explanation of the quiver $Q \bs \alpha_q$.
\end{itemize}

\item (sesquilinearity conditions) 
\begin{itemize}
\item 
For a connected component $I^a$ of $\ul{Q}$, we have
\begin{align}\label{eq:2P:Pclses1}
 (\pdd_{\lambda_k}-\pdd_{\lambda_l}) X^Q(v) = 0 \quad (k,l \in I^a, \, v \in V^{\otimes n}), 
\end{align}
and
\begin{align}\label{eq:2P:Pclses2}
   \sum_{k\in I^a}X^Q(T^{(k)}v) = -\sum_{k\in I^a}\lambda_{k}X^Q(v) \quad (v \in V^{\otimes n}).
\end{align}
Here and hereafter, for a linear transformation $\varphi\in \End V$, the symbol $\varphi^{(k)}$ denotes the linear transformation on $V^{\otimes n}$ defined by $\varphi^{(k)}\ceq \id_V\otimes \cdots \otimes \overset{k}{\varphi} \otimes \cdots \otimes \id_V$.

\item
For each $k \in [n]$, we have
\begin{align}\label{eq:2P:Pclses3}
  X^Q((S^i)^{(k)}v) = -(-1)^{p(X)}\theta_k^iX^Q(v) \quad (v\in V^{\otimes n}).
\end{align}
\end{itemize} 
\end{clist}
\end{dfn}

\begin{rmk}\label{rmk:2P:Pcl}
The cycle relations are the same as the non-SUSY case \cite[(10.4), (10.5)]{BDHK}, and the ``bosonic part'' of the sesquilinearity conditions \eqref{eq:2P:Pclses1} and \eqref{eq:2P:Pclses2} are the same as \cite[(10.6), (10.7)]{BDHK}. The ``fermionic part''\eqref{eq:2P:Pclses3} is the new point of our definition.
\end{rmk}

In the remaining of this \cref{ss:2P:NW}, we fix an $\clH_W$-supermodule $V=(V,\nabla)$. 
Let us define a right action of the symmetric group $\frS_n$ on $\oPclW{V}(n)$ for each $n \in \bbN$, following the non-SUSY case \cite[(10.10)]{BDHK}: 

\begin{dfn}
Let $n \in \bbN$. For $\sigma \in \frS_n$ and $X \in \oPclW{V}(n)$, we define a linear map 
\[
 X^\sigma\colon V^{\otimes n} \otimes \bbK\oQ(n) \lto V_\nabla[\Lambda_k]_{k=1}^n
\]
by
\begin{align*}
 X^\sigma(v_1 \otimes \dotsb \otimes v_n \otimes Q) \ceq
 X^{\sigma Q}_{\sigma(\Lambda_1,\dotsc,\Lambda_n)} \bigl(\sigma(v_1 \otimes \dotsb \otimes v_n)\bigr) 
 \quad (v_1 \otimes \dotsb \otimes v_n\in V^{\otimes n}, \, Q \in \oQ(n)), 
\end{align*}
where for $Q=(Q_0,Q_1,s_Q,t_Q)$ we set $\sigma Q \ceq (Q_0, Q_1, \sigma s_Q, \sigma t_Q)$ with $\sigma s_Q$ and $\sigma t_Q$ given in \cref{dfn:2P:ngr} and 
\begin{align*}
&\sigma(\Lambda_1,\dotsc,\Lambda_n) \ceq 
 \bigl(\Lambda_{\sigma^{-1}(1)},\dotsc,\Lambda_{\sigma^{-1}(n)}\bigr), \\
&\sigma(v_1 \otimes \dotsb \otimes v_n) \ceq
 \prod_{\substack{1\le k<l\le n \\ \sigma(k)>\sigma(l)}} (-1)^{p(v_k)p(v_l)} \cdot 
 v_{\sigma^{-1}(1)} \otimes \dotsb \otimes v_{\sigma^{-1}(n)}.
\end{align*} 
\end{dfn}

As shown in the next \cref{lem:2P:frSPcl}, this $\frS_n$-action is well-defined. 
The strategy of the proof is similar as the non-SUSY case \cite[Theorem 10.6, Proof, 1st paragraph]{BDHK}.

\begin{lem}\label{lem:2P:frSPcl}
Let $n \in \bbN$. For any $\sigma \in \frS_n$ and $X \in \oPclW{V}(n)$, we have $X^\sigma \in \oPclW{V}(n)$.
\end{lem}

\begin{proof}
We may assume $n>1$. We prove that the linear map $X^\sigma$ satisfies the cycle relations and sesquilinearity conditions. Let us fix $Q \in \oQ(n)$. 
\begin{itemize}
\item
$X^\sigma$ satisfies the first cycle relation: 
If $\ul{Q}$ contains a cycle, then $\ul{\sigma Q}=\sigma\ul{Q}$ also contains a cycle by \cref{lem:2P:graph} \ref{i:lem:2P:graph:1}, \ref{i:lem:2P:graph:3}. Thus, for each $v \in V^{\otimes n}$, we have
\begin{align*}
 (X^\sigma)^Q(v) = X^{\sigma Q}_{\sigma(\Lambda_1,\dotsc,\Lambda_n)}(\sigma v) = 0.
\end{align*}

\item 
$X^\sigma$ satisfies the second cycle relation: This can be checked by a direct calculation using \cref{lem:2P:dgraph} \ref{i:lem:2P:dgraph:1}, \ref{i:lem:2P:dgraph:3}. 

\item 
$X^\sigma$ satisfies the sesquilinearity condition \eqref{eq:2P:Pclses1}: 
Let $I^a$ be a connected component of $\ul{Q}$. For $k,l \in I^a$, we have
\begin{align*}
  (\pdd_{\lambda_k}-\pdd_{\lambda_l})(X^\sigma)^Q(v)
&=(\pdd_{\lambda_k}-\pdd_{\lambda_l})X^{\sigma Q}_{\sigma(\Lambda_1,\dotsc,\Lambda_n)}(\sigma v) \\
&=\rst{\bigl(\pdd_{\lambda'_{\sigma(k)}}-\pdd_{\lambda'_{\sigma(l)}}\bigr)
       X^{\sigma Q}_{\Lambda_1',\dotsc,\Lambda_n'}(\sigma v)}
      {(\Lambda_1',\dotsc,\Lambda_n') = \sigma(\Lambda_1,\dotsc,\Lambda_n)} \\
&=0,
\end{align*}
where $\Lambda_1', \ldots, \Lambda_n'$ are $(1|N)_W$-supervariables, and we used \cref{lem:2P:graph} \ref{i:lem:2P:graph:2} in the third equality. 

\item
$X^\sigma$ satisfies the sesquilinearity condition \eqref{eq:2P:Pclses2}: 
A direct calculation yields
\begin{align*}
 \sigma(T^{(k)}v)=T^{(\sigma(k))}(\sigma v) \quad (v \in V^{\otimes n}).
\end{align*}
for $k \in [n]$. This identity and \cref{lem:2P:graph} \ref{i:lem:2P:graph:4} gives the statement.

\item 
$X^\sigma$ satisfies the sesquilinearity condition \eqref{eq:2P:Pclses3}: 
For $k \in [n]$, one can obtain
\begin{align*}
 \sigma((S^i)^{(k)}v)=(S^i)^{(\sigma(k))}(\sigma v) \quad (v\in V^{\otimes n}), 
\end{align*}
which yields the statement. 
\end{itemize}
\end{proof}

By this action, the linear superspace $\oPchW{V}(n)$ is a right $\frS_n$-supermodule. 
Thus, we have an $\frS$-supermodule $\oPclW{V} \ceq \bigl(\oPclW{V}(n)\bigr)_{n\in\bbN}$. 

To describe the composition maps on $\oPclW{V}$, we introduce additional $(1|N)_W$-supervariables $\Xi_k=(x_k, \xi_k^1,\dotsc,\xi_k^N)$ for $k \in \bbZ_{>0}$, and for $m,n \in \bbN$, $\nu \in \bbN^m_n$, $Q \in \oQ(n)$ and $j\in[n]$, we set 
\begin{align*}
 \Xi_Q^\nu(j) \ceq \sum_{i \in \clE_Q^\nu(j)}\Xi_i. 
\end{align*}
Here $\clE_Q^\nu(j)$ is the set of $i \in [m]$ which are externally connected to $j$ defined in \cref{dfn:2P:extconn}. 
Also, recall the isomorphism \eqref{eq:2P:VNL=VL}. 
Now, following \cite[(10.11)--(10.13)]{BDHK}, we introduce the composition maps as:

\begin{dfn}\label{dfn:2P:clWcomp}
Let $m,n \in \bbN$ and $\nu=(n_1,\ldots,n_m) \in \bbN^m_n$. 
Set $N_j \ceq n_1+\dotsb+n_j$ for $j \in [m]$ and $N_0 \ceq 0$. 
\begin{enumerate}
\item 
For $Y_j \in \oPclW{V}(n_j)$, $j \in [m]$, we define a linear map
\begin{align*}
 Y_1 \odot \cdots \odot Y_m\colon V^{\otimes n} \otimes \bbK\oQ(n) \lto 
 \bigotimes_{j=1}^m V_\nabla[\Lambda_k]_{k=N_{j-1}+1}^{N_j} \otimes \bbK\oQ(m)
 \cong \bigotimes_{j=1}^m V[\Lambda_k]_{N_{j-1}+1}^{N_j-1}\otimes \bbK\oQ(m)
\end{align*}
by
\begin{align*}
&(Y_1 \odot \cdots \odot Y_m)(v_1 \otimes \dotsb \otimes v_n \otimes Q) \\
&\ceq \pm (Y_1)_{\Lambda_1,\ldots,\Lambda_{N_1}}^{\Delta_1^\nu(Q)}(w_1) \otimes \cdots \otimes 
          (Y_m)_{\Lambda_{N_{m-1}+1},\ldots,\Lambda_{N_m}}^{\Delta_m^\nu(Q)}(w_m) \otimes 
          \Delta_0^\nu(Q)
\end{align*}
for $v_1 \otimes \dotsb \otimes v_n \in V^{\otimes n}$ and $Q \in \oQ(n)$. Here we used
\begin{align*}
 \pm \ceq \prod_{1 \le i < j \le m} (-1)^{p(w_i)p(Y_j)}, \quad
 w_j \ceq v_{N_{j-1}+1} \otimes \dotsb \otimes v_{N_j} \quad (j \in [m]). 
\end{align*}

\item 
For $X \in \oPclW{V}(m)$ and $Y_j \in \oPclW{V}(n_j)$ with $j \in [m]$, let 
\begin{align*}
 X \circ (Y_1 \odot \cdots \odot Y_m)\colon V^{\otimes n} \otimes \bbK\oQ(n) \lto V_\nabla[\Lambda_k]_{k=1}^n
\end{align*}
be the linear map defined by composing the following linear maps: 
\begin{align*}
 V^{\otimes n}\otimes \bbK\oQ(n) 
&\xrr{Y_1\odot \cdots \odot Y_m}
  \bigotimes _{j=1}^m V[\Lambda_k]_{k=N_{j-1}+1}^{N_j-1} \otimes \bbK\oQ(m) \\ 
&\lto V^{\otimes m}[\Lambda_k]_{k=1}^n \otimes \bbK\oQ(m) 
 \xrr{X_{\Lambda_1',\dotsc,\Lambda_m'}} V_\nabla[\Lambda_k]_{k=1}^n.
\end{align*}
Here the second arrow 
denotes the linear map defined by 
\begin{align*}
 a_1 v_1 \otimes \cdots \otimes a_m v_m \otimes Q \lmto
 \rst{\pm a}{\Xi_j=\Lambda_j'+\nabla^{(j)} \, (j=1,\ldots,m)}(v_1 \otimes \cdots \otimes v_m ) \otimes Q
\end{align*}
for each $a_j \in \bbK[\Lambda_k]_{k=N_{j-1}+1}^{N_j-1}$, $v_j\in V$ and $Q \in \oQ(m)$ with 
\begin{align*}
&\pm \ceq \prod_{1\le i<j\le m} (-1)^{p(v_i)p(a_j)}, \quad 
 \Lambda_j' \ceq \Lambda_{N_{j-1}+1}+\dotsb+\Lambda_{N_j} \quad (j \in [m]), \\
&a \ceq \prod_{j=1}^ma_r(\Lambda_{N_{j-1}+1}+\Xi_Q^\nu(N_{j-1}+1),\ldots,\Lambda_{N_j-1}+\Xi_Q^\nu(N_j-1)),
\end{align*}
and the third arrow $X_{\Lambda_1',\ldots,\Lambda_m'}$ is the linear map defined by
\begin{align*}
 X_{\Lambda_1',\ldots,\Lambda_m'}\colon 
 b (v_1 \otimes \cdots \otimes v_m) \otimes Q \lmto 
 (-1)^{p(X)p(b)}b X_{\Lambda_1',\dotsc,\Lambda_m'}(v_1 \otimes \cdots \otimes v_m \otimes Q)
\end{align*}
for $b \in \bbK[\Lambda_k]_{k=1}^n$, $v_j \in V$ and $Q \in \oQ(n)$. 
\end{enumerate}
\end{dfn}

The definition looks quite complicated, and it is hard to see that composition maps preserve the cycle relations and the sesquilinearity conditions. To illustrate the sesquilinearity conditions, in particular the ``fermionic part'' \eqref{eq:2P:Pclses3}, let us give an example of composition.

\begin{eg}\label{eg:2P:comp}
Let us consider the case $m=3$, $n=5$, $\nu=(n_1,n_2,n_3)=(2,1,2)$ and $X \in \oPclW{V}(3)$, $Y_r \in \oPclW{V}(n_r)$ with $r \in [3]$. For the quiver
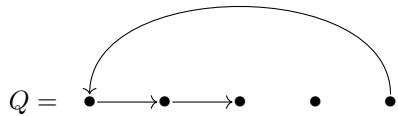
\begin{align*}
\begin{tikzpicture}
\draw (-0.75, 0) node {$Q=$}; 
\draw (0, 0) node {$\bullet$}; \draw (1, 0) node {$\bullet$}; 
\draw (2, 0) node {$\bullet$}; \draw (3, 0) node {$\bullet$}; 
\draw (4, 0) node {$\bullet$}; 
\draw[->] (0.1, 0) to (0.9,0); 
\draw[->] (1.1, 0) to (1.9, 0);
\draw[->] (4, 0.1) to[out=90, in=90] (0, 0.1);
\end{tikzpicture}
\end{align*}
let us calculate
\begin{align*}
 (X \circ (Y_1 \odot Y_2 \odot Y_3))(v_1 \otimes \cdots \otimes v_5 \otimes Q) \quad 
 (v_1,\ldots,v_5\in V), 
\end{align*}
and check the sesquilinearity conditions \eqref{eq:2P:Pclses2} and \eqref{eq:2P:Pclses3} for $X \circ (Y_1 \odot Y_2 \odot Y_3)$. 
First we compute the cocomposition $\Delta^\nu(Q)$: 
\begin{align*}
\begin{tikzpicture}
\draw (-1, 0) node {$\Delta_0^\nu(Q)=$}; 
\draw (0, 0) node {$\bullet$}; \draw (1, 0) node {$\bullet$}; 
\draw (2, 0) node {$\bullet$}; 
\draw[->] (0.1, 0) to (0.9, 0);
\draw[->] (2, 0.1) to[out=90, in=90] (0, 0.1);
\draw (4, 0) node {$\Delta_1^\nu(Q)=$};
\draw (5, 0) node {$\bullet$}; \draw (6, 0) node {$\bullet$};
\draw[->] (5.1, 0) to (5.9, 0); 
\draw (8, 0) node {$\Delta_2^\nu(Q)=$};
\draw (9, 0) node {$\bullet$}; 
\draw (11, 0) node {$\Delta_3^\nu(Q)=$};
\draw (12, 0) node {$\bullet$};
\draw (13, 0) node {$\bullet$};
\end{tikzpicture}
\end{align*}
Thus 
\begin{align*}
&(X\circ (Y_1\odot Y_2\odot Y_3))(v_1\otimes \cdots \otimes v_5\otimes Q)\\
&=X^{\Delta_0^\nu(Q)}_{\Lambda_1', \Lambda_2', \Lambda_3'}\Bigl(
	\rst{(Y_1)_{\Gamma_1}^{\Delta_1^\nu(Q)}(v_1 \otimes v_2) \otimes 
	(Y_2)^{\Delta_2^\nu(Q)}(v_3)\otimes
	(Y_3)_{\Gamma_4}^{\Delta_3^\nu(Q)}}{\Xi_r=\Lambda_r'+\nabla^{(r)}\, (r=1, 2, 3)}
\Bigr)
\end{align*}
with
\begin{align*}
\Lambda_1'\ceq \Lambda_1+\Lambda_2, \quad 
\Lambda_2'\ceq \Lambda_3 \quad 
\Lambda_3'\ceq \Lambda_4+\Lambda_5, \quad
\Gamma_1\ceq \Lambda_1+\Xi_Q^\nu(1), \quad 
\Gamma_4\ceq \Lambda_4+\Xi_Q^\nu(4). 
\end{align*}
Here we suppress $\Lambda_2$ in $(Y_1)_{\Lambda_1, \Lambda_2}^{\Delta_1^\nu(Q)}(v_1 \otimes v_2)$ when we use the isomorphism \eqref{eq:2P:VNL=VL} as well as for $Y_2$ and $Y_3$. 
By $\clE_Q^\nu(1)=\{3\}$ and $\clE_Q^\nu(4)=\emptyset$, we have $\Gamma_1=\Lambda_1+\Xi_3$ and $\Gamma_4=\Lambda_4$, so we get
\begin{align*}
&(X\circ (Y_1 \odot Y_2 \odot Y_3))(v_1 \otimes \cdots \otimes v_5 \otimes Q)\\
&=X^{\Delta_0^\nu(Q)}_{\Lambda_1', \Lambda_2', \Lambda_3'}\Bigl(
 (Y_1)_{\Lambda_1+\Lambda_3'+\nabla^{(3)}}^{\Delta_1^\nu(Q)}(v_1 \otimes v_2) \otimes 
 (Y_2)^{\Delta_2^\nu(Q)}(v_3) \otimes 
 (Y_3)_{\Lambda_4}^{\Delta_3^\nu(Q)}(v_4 \otimes v_5) \Bigr)\\
&=X_{\Lambda_1+\Lambda_2, \Lambda_3, \Lambda_4+\Lambda_5}^{\Delta_0^\nu(Q)}\Bigl(
 (Y_1)_{-\Lambda_2-\Lambda_3-\nabla^{(1)}-\nabla^{(2)}}(v_1 \otimes v_2) \otimes 
 (Y_2)^{\Delta_2^\nu(Q)}(v_3) \otimes 
 (Y_3)_{\Lambda_4}^{\Delta_3^\nu(Q)}(v_4 \otimes v_5) \Bigr), 
\end{align*}
where we used the sesquilinearity condition \eqref{eq:2P:Pclses2} of $X$ in the second equality. 
Now it is easy to see that $X\circ (Y_1\odot Y_2\odot Y_3)$ satisfies the sesquilinearity conditions \eqref{eq:2P:Pclses2} and \eqref{eq:2P:Pclses3}. 
\end{eg}

As shown in the next \cref{lem:2P:clWcomp}, the $\frS$-module $\oPclW{V}$ is closed under the composition maps in \cref{dfn:2P:clWcomp}. The strategy of the proof is the same as the non-SUSY case \cite[Lemma 10.5]{BDHK}, but let us write it down.

\begin{lem}\label{lem:2P:clWcomp}
Let $m,n \in \bbN$ and $\nu \in \bbN^m_n$. 
For any $X \in \oPclW{V}(m)$ and $Y_1 \otimes \cdots \otimes Y_m \in \oPclW{V}(\nu)$, we have
\begin{align*}
 X \circ (Y_1 \odot \cdots \odot Y_m) \in \oPclW{V}(n). 
\end{align*}
\end{lem}

\begin{proof}
Let $\nu=(n_1,\dotsc,n_m) \in \bbN^m_n$, and fix $v=v_1 \otimes \cdots \otimes v_n \in V^{\otimes n}$. 
For $j \in [m]$, we use the notations similar as in \cref{dfn:2P:clWcomp}:
\begin{align*}
 N_j \ceq n_1+\dotsb+n_j, \quad 
 w_j \ceq v_{N_{j-1}+1} \otimes \cdots \otimes v_{N_j}, \quad 
 \Lambda_j' \ceq \Lambda_{N_{j-1}+1}+\dotsb+\Lambda_{N_j}.
\end{align*}
For each $j \in [m]$, $Q \in \oQ(n_j)$ and $w \in V^{\otimes n_j}$, let us denote
\begin{align*}
&a_j^Q(w)(\Lambda_{N_{j-1}+1},\ldots,\Lambda_{N_j-1}) \ceq 
 (Y_j)^Q_{\Lambda_{N_{j-1}+1},\ldots,\Lambda_{N_j-1},-\Lambda_{N_{j-1}+1}-\cdots-\Lambda_{N_j-1}-\nabla}(w),\\
&\wt{a}_j^Q(w) \ceq 
 \rst{a_j^Q(w)\bigl(\Lambda_{N_{j-1}+1}+\Xi^\nu_Q(N_{j-1}+1), \ldots, \Lambda_{N_j-1}+\Xi^\nu_Q(N_j-1)\bigr)}
     {\Xi_k=\Lambda_k'+\nabla^{(k)} \, (k=1,\ldots,m)}.
\end{align*}
Then, for $Q \in \oQ(n)$ we have
\begin{align*}
 (X \circ (Y_1 \odot \cdots \odot Y_m))^Q(v)
=\pm X^{\Delta^\nu_0(Q)}_{\Lambda_1',\ldots,\Lambda_m'}
 \Bigl(\wt{a}_1^{\Delta^\nu_1(Q)}(w_1) \otimes \cdots \otimes \wt{a}_m^{\Delta^\nu_m(Q)}(w_m)\Bigr)
\end{align*}
with $\pm\ceq \prod_{1\le i<j \le m}(-1)^{p(w_i)p(Y_j)}$. 
Now we check the cycle relations and sesquilinearity conditions in \cref{dfn:2P:Pcl}.
\begin{itemize}
\item
$X \circ (Y_1 \odot \cdots \odot Y_m)$ satisfies the first cycle relation: 
If $Q \in \oQ(n) \bs \oQ_\ac(n)$, then it is clear that $\Delta^\nu_j(Q) \in \oQ(n_j) \bs \oQ_\ac(n_j)$ for some $j \in [m]$ or $\Delta^\nu_0(Q) \in \oQ(m) \bs \oQ_\ac(m)$. Thus, for any $Q \in \oQ(n) \bs \oQ_\ac(n)$, we have $(X \circ (Y_1 \odot \cdots \odot Y_m))^Q=0$ since $X$ and $Y_1,\dotsc,Y_m$ satisfy the first cycle relation. 

\item
$X\circ (Y_1 \odot \cdots \odot Y_m)$ satisfies the second cycle relation:
Let $Q \in \oQ(n)$ and $(\alpha_1,\dotsc,\alpha_k)$ be a cycle of $Q$. 
Let us denote the vertex set of a quiver $\Gamma$ by $E(\Gamma) \ceq \Gamma_1$.
If $\{\alpha_1,\dotsc,\alpha_k\} \subset E(\Delta_j^\nu(Q))$ for some $j\in [m]$, then $(\alpha_1,\ldots,\alpha_k)$ is a cycle of $\Delta_j^\nu(Q)$. Thus
\begin{align*}
&\sum_{q=1}^k (X \circ (Y_1 \odot \cdots \odot Y_m))^{Q \bs \alpha_q}(v) \\
&=\pm X^{\Delta^\nu_0(Q)}_{\Lambda_1',\ldots,\Lambda_m'} \Bigl(
   \wt{a}_1^{\Delta^\nu_1(Q)}(w_1) \otimes \cdots \otimes 
        \Bigl(\sum_{q=1}^k \wt{a}_j^{\Delta^\nu_j(Q) \bs \alpha_q}(w_j)\Bigr) \otimes 
        \cdots \otimes \wt{a}_m^{\Delta^\nu_m(Q)}(w_m)\Bigr) \\
&=0
\end{align*}
since $Y_j$ satisfies the second cycle relation.
If $\{\alpha_1,\dotsc,\alpha_k\} \not\subset E(\Delta_j^\nu(Q))$ for any $j \in [m]$, then there exists $\beta_1,\ldots,\beta_l \in \{\alpha_1,\ldots,\alpha_k\}$ such that $(\beta_1,\ldots,\beta_l)$ is a directed cycle of $\Delta^\nu_0(Q)$ and 
\begin{align*}
 A \ceq \{\alpha_1,\ldots,\alpha_k\}\bs\{\beta_1, \ldots, \beta_l\} \subset E(Q) \bs E(\Delta^\nu_0(Q)).
\end{align*}
Thus 
\begin{align*}
  \sum_{q=1}^k (X \circ (Y_1 \odot \cdots \odot Y_m))^{Q \bs \alpha_q}(v)
&=\sum_{\alpha\in A} X^{\Delta^\nu_0(Q)}_{\Lambda_1',\ldots,\Lambda_m'}\Bigl(
   \wt{a}_1^{\Delta^\nu_1(Q \bs \alpha)}(w_1) \otimes \cdots \otimes \wt{a}_m^{\Delta^\nu_m(Q\bs\alpha)}(w_m) \Bigr) \\
&\quad 
 +\sum_{q=1}^l X^{\Delta^\nu_0(Q) \bs \beta_q}_{\Lambda_1',\ldots,\Lambda_m'}\Bigl(
   \wt{a}_1^{\Delta^\nu_1(Q)}(w_1) \otimes \cdots \otimes \wt{a}_m^{\Delta^\nu_m(Q)}(w_m)\Bigr) \\
&=0
\end{align*}
since $X$ satisfies the first and second cycle relations. 

\item
$X\circ (Y_1 \odot \cdots \odot Y_m)$ satisfies the sesquilinearity condition \eqref{eq:2P:Pclses1}: For $k, l\in [n]$ such that $k$, $l$ are connected in $Q$, we need to show that
\begin{align}\label{eq:2P:compses1}
 (\pdd_{\lambda_k}-\pdd_{\lambda_l})(X\circ (Y_1 \odot \cdots \odot Y_m))^Q(v)=0. 
\end{align}
If $k,l \in \{N_{j-1}+1,\ldots,N_j\}$ for some $j \in [m]$, then $\Lambda_k$, $\Lambda_l$ appear in $\Lambda'_j$ in the form of $\Lambda_k+\Lambda_l$, and do not appear in $\Lambda_q'$ for $q \in [m] \bs \{j\}$. 
Also, the vertices $k-N_{j-1}$ and $l-N_{j-1}$ are connected in $\Delta^\nu_j(Q)$. Thus, the sesquilinearity condition \eqref{eq:2P:Pclses1} of $Y_j$ implies \eqref{eq:2P:compses1}. 
If $k\in\{N_{i-1},\ldots,N_i\}$ and $l \in \{N_{j-1},\ldots,N_j\}$ for different $i,j \in [m]$, then $i,j$ are connected in $\Delta^\nu_0(Q)$. 
Also, by the sesquilinearity condition of $Y_1,\ldots,Y_m$ and the definition of $\Xi^\nu_Q$, we have that $\Lambda_k$, $\Lambda_l$ appearing in $\wt{a}_q^{\Delta_q^\nu(Q)}(w_q)$ ($q\in[m]$) are in the form of $\Lambda_k+\Lambda_l$. Thus, the identity \eqref{eq:2P:compses1} holds. 

\item
$X\circ (Y_1 \odot \cdots \odot Y_m)$ satisfies the sesquilinearity condition \eqref{eq:2P:Pclses2}: Let $I^a$ be a connected component of $\ul{Q}$, and set 
\begin{align*}
 I^a_j \ceq \ul{\Delta_j^\nu(Q)}\cap \{k-N_{j-1}\mid k\in I_j\}
\end{align*}
for $j \in [m]$. Clearly, each $I^a_j$ is a connected component of $\Delta^\nu_j(Q)$. If $\ul{\Delta^\nu_0}(Q)$ has a cycle, then by the first cycle relation of $X$, the sesquilinearity condition \eqref{eq:2P:Pclses2} for $X\circ (Y_1 \odot \cdots \odot Y_m)$ is trivial. 
Thus, we can assume that $\ul{\Delta^\nu_0(Q)}$ is acyclic. Then, we have $\#I^a=\sum_{j=1}^m\#I^a_j$, which implies
\begin{align*}
&\sum_{k \in I^a} (X\circ(Y_1\odot\cdots \odot Y_m))^Q(T^{(k)}v)\\
&=\sum_{j=1}^m \sum_{k\in I_j^a}X^{\Delta^\nu_0(Q)}_{\Lambda_1', \ldots, \Lambda_m'}\Bigl(
 \wt{a}_1^{\Delta^\nu_1(Q)}(w_1) \otimes \cdots \otimes \wt{a}_j^{\Delta^\nu_j(Q)}(T^{(k)}w_j) 
 \otimes \cdots \otimes \wt{a}_m^{\Delta^\nu_m(Q)}(w_m)\Bigr). 
\end{align*}
By the sesquilinearity condition \eqref{eq:2P:Pclses2} of $Y_j$, we have
\begin{align*}
 \sum_{k\in I^a_j} \wt{a}_j^{\Delta^\nu_j(Q)}(T^{(k)}w_j) = 
-\sum_{k\in I_j^a} \rst{\bigl(\lambda_{N_{j-1}+k}+x^\nu_Q(N_{j-1}+k)\bigr)}
                       {\Xi_q=\Lambda_q'} \wt{a}_j^{\Delta_j^\nu(Q)}(w_j). 
\end{align*}
Hence, one can get by the sesquilinearity condition \eqref{eq:2P:Pclses2} of $X$, 
\begin{align*}
& \sum_{k \in I^a} (X \circ (Y_1 \odot \cdots \odot Y_m))^Q(T^{(k)}v)
=-\sum_{k \in I^a} \lambda_k(X \circ (Y_1 \odot \cdots \odot Y_m))^Q(v). 
\end{align*}

\item 
$X\circ (Y_1 \odot \cdots \odot Y_m)$ satisfies the sesquilinearity condition \eqref{eq:2P:Pclses3}: 
This can be checked similarly to the sesquilinearity condition \eqref{eq:2P:Pclses2}. 
\end{itemize}
\end{proof}

\begin{prp}
For an $\clH_W$-supermodule $V=(V,\nabla)$, the $\frS$-supermodule 
\begin{align*}
 \oPclW{V} = \bigl(\oPclW{V}(n)\bigr)_{n \in \bbN}
\end{align*}
is a superoperad by letting $X \otimes Y_1 \otimes \cdots \otimes Y_m \mto X \circ (Y_1 \odot \cdots \odot Y_m)$ be the composition map and $\id_V \in \oPclW{V}(1)$ be the unit. 
\end{prp}

\begin{proof}
The proof of the non-SUSY case \cite[Theorem 10.6]{BDHK} works with minor modification. We omit the detail.
\end{proof}

\begin{dfn}
We call the superoperad $\oPclW{V}$ \emph{the $N_W=N$ SUSY coisson operad of $V=(V,\nabla)$}.
\end{dfn}

Recall the space of Maurer-Cartan solutions from \eqref{eq:1P:MC}. The following is the main theorem of this \cref{ss:2P:NW}, which is a natural $N_W=N$ SUSY analogue of \cite[Theorem 10.7]{BDHK}. 

\begin{thm}\label{thm:2P:NWPVA}
Let $V=(V, \nabla)$ be an $\clH_W$-supermodule. 
\begin{enumerate}
\item 
For an odd Maurer-Cartan solution $X \in \MC\big(L\bigl(\oPclW{\Pi^{N+1}V}\bigr)\bigr)_{\ol{1}}$, define linear maps $\{\cdot_\Lambda\cdot\}_X\colon V \otimes V \to V[\Lambda]$ and $\mu_X\colon V \otimes V \to V$ by
\begin{align}
 \label{eq:2P:PVAlbra}
 \{a_\Lambda b\}_X &\ceq (-1)^{p(a)(\ol{N}+\ol{1})} X_{\Lambda, -\Lambda-\nabla}^{\tw}(a \otimes b) \\
 \label{eq:2P:PVApro}
 \mu_X(a\otimes b) &\ceq (-1)^{p(a)(\ol{N}+\ol{1})}
 \Res_\Lambda\bigl(\lambda^{-1}X_{\Lambda, -\Lambda-\nabla}^{\twab}(a \otimes b)\bigr)
\end{align}
for each $a, b\in V$. Then $(V, \nabla, \{\cdot_\Lambda\cdot\}_X, \mu_X)$ is a non-unital $N_W=N$ SUSY Poisson vertex algebra. 

\item 
The map $X \mto (\{\cdot_\Lambda\cdot\}_X, \mu_X)$ gives a bijection
\begin{align*}
 \MC\bigl(L\bigl(\oPclW{\Pi^{N+1}V}\bigr)\bigr)_{\ol{1}} \lsto
 \{\textup{non-unital $N_W=N$ SUSY PVA structure on $(V, \nabla)$}\}. 
\end{align*}
\end{enumerate}
\end{thm}

In the rest of  this \cref{ss:2P:NW}, we prove \cref{thm:2P:NWPVA}.
We denote $\wt{V} \ceq \Pi^{N+1}V=(\Pi^{N+1}V, \nabla)$ and $\oP \ceq \oPclW{\Pi^{N+1}V}$ for simplicity, and let $p,\wt{p}$ be the parity of $V,\wt{V}$ respectively.
Consider the linear maps $\{\cdot_\Lambda \cdot\}_X\colon V \otimes V \to V[\Lambda]$ and $\mu_X\colon V \otimes V \to V$ defined by \eqref{eq:2P:PVAlbra} and \eqref{eq:2P:PVApro} for an odd element $X \in \oP(2)_{\ol{1}}$. It is clear that $\{\cdot_\Lambda\cdot\}_X\colon V \otimes V \to V_\nabla[\Lambda]$ has parity $\ol{N}$, and $\mu_X\colon V\otimes V\to V$ is even. 

\begin{itemize}
\item
The linear map $\{\cdot_\Lambda\cdot\}_X$ satisfies (i) in \cref{dfn:1P:WLCA}: This can be checked by a direct calculation using \eqref{eq:2P:Pclses2} and \eqref{eq:2P:Pclses3}. For instance, we have
\begin{align*}
 \{S^ia_\Lambda b\}
&=(-1)^{(p(a)+\ol{1})(\ol{N}+\ol{1})}X^{\tw}_{\Lambda, -\Lambda-\nabla}(S^ia\otimes b)\\
&=(-1)^{(p(a)+\ol{1})(\ol{N}+\ol{1})}(-1)^{\wt{p}(X)+\ol{1}}
  \theta^i X^{\tw}_{\Lambda, -\Lambda-\nabla}(a\otimes b)\\
&=-(-1)^N\theta^i\{a_\Lambda b\}. 
\end{align*}

\item
The linear map $\mu_X$ satisfies (i) in \cref{dfn:1P:PVA}: By \eqref{eq:2P:Pclses2}, we have
\begin{align*}
   X^{\twab}_{\Lambda,-\Lambda-\nabla}(Ta \otimes b)
  +X^{\twab}_{\Lambda,-\Lambda-\nabla}(a \otimes Tb)
&= -\lambda X^{\twab}_{\Lambda,-\Lambda-\nabla}(a \otimes b)
+(\lambda+T)X^{\twab}_{\Lambda,-\Lambda-\nabla}(a \otimes b)\\
&=TX^{\twab}_{\Lambda,-\Lambda-\nabla}(a \otimes b),
\end{align*}
and by the definition \eqref{eq:2P:PVApro} of $\mu_X$ get $T(ab)=(Ta)b+a(Tb)$. 
We can also prove $S^i(ab)=(S^ia)b+(-1)^{p(a)}a(S^ib)$ using \eqref{eq:2P:Pclses3}. 
\end{itemize}

Note now that the map $X \mto (\{\cdot_\Lambda \cdot\}_X,\mu_X)$ gives a bijective correspondence between the set $\oP(2)_{\ol{1}}$ and the set of all pairs $(\{\cdot_\Lambda \cdot\},\mu)$ of a linear map $\{\cdot_\Lambda \cdot\}\colon V \otimes V \to V[\Lambda]$ of parity $\ol{N}$ satisfying (i) in \cref{dfn:1P:WLCA} and an even linear map $\mu\colon V \otimes V \to V$ satisfying (i) in \cref{dfn:1P:PVA}. 
The inverse map 
is given by 
\begin{align*}
X^{\tw}_{\Lambda, -\Lambda-\nabla}(a\otimes b)
&\ceq (-1)^{p(a)(\ol{N}+\ol{1})}\{a_\Lambda b\}, \\
X^{\twab}_{\Lambda, -\Lambda-\nabla}(a\otimes b)
&\ceq (-1)^{p(a)(\ol{N}+\ol{1})}\sum_{I\subset [N]}(-1)^{\#I(N+1)}\sigma(I)\theta^{[N]\bs I}\mu(S^Ia\otimes b). 
\end{align*}

\begin{itemize}
\item
The skew-symmetry \eqref{eq:1P:LCA:ssym} of $\{\cdot_\Lambda\cdot\}_X$ and the commutativity of $\mu_X$ are equivalent to $X^\sigma=X$ for any $\sigma \in \frS_2$: By the cycle relation $X^Q=0$ for $Q \in \oQ(2) \bs \oQ_{\ac}(2)$, the element $X$ satisfies $X^\sigma=X$ ($\sigma \in \frS_2$) if and only if $(X^{(1, 2)})^Q=X^Q$ for $Q \in \oQ_{\ac}(2)$, that is
\begin{align}
\label{eq:2P:Xske}
&(-1)^{\wt{p}(a)\wt{p}(b)}X^{\tw}_{\Lambda_2, \Lambda_1}(b\otimes a)
=X^{\tw}_{\Lambda_1, \Lambda_2}(a\otimes b), \\
\label{eq:2P:Xcom}
&(-1)^{\wt{p}(a)\wt{p}(b)}X^{\twba}_{\Lambda_2, \Lambda_1}(b\otimes a)
=X^{\twab}_{\Lambda_1, \Lambda_2}(a\otimes b).
\end{align}
It is clear that the identity \cref{eq:2P:Xske} is equivalent to the skew-symmetry of $\{\cdot_\Lambda\cdot\}_X$. 
Thus, it remains to prove that: 
\end{itemize}
       
\begin{lem}
The identity \eqref{eq:2P:Xcom} is equivalent to the commutativity of $\mu_X$. 
\end{lem}

\begin{proof}
By the cyclic relation $X^{\twba}=-X^{\twab}$, the identity \eqref{eq:2P:Xcom} is equivalent to 
\begin{align*}
 (-1)^{\wt{p}(a)\wt{p}(b)+\ol{1}}X^{\twab}_{-\Lambda-\nabla,\Lambda}(b \otimes a)
=X^{\twab}_{\Lambda,-\Lambda-\nabla}(a \otimes b). 
\end{align*}

If the identity \eqref{eq:2P:Xcom} holds, then we have
\begin{align*}
 (-1)^{\wt{p}(a)\wt{p}(b)+\ol{1}}
 \Res_\Lambda\bigl(\lambda^{-1}X^{\twab}_{-\Lambda-\nabla,\Lambda}(b \otimes a)\bigr)
=\Res_\Lambda\bigl(\lambda^{-1}X^{\twab}_{\Lambda,-\Lambda-\nabla}(a \otimes b)\bigr). 
\end{align*}
By the sesquilinearity \eqref{eq:2P:Pclses1}, one can show that
\begin{align*}
 \Res_\Lambda\bigl(\lambda^{-1}X^{\twab}_{-\Lambda-\nabla,\Lambda}(b \otimes a)\bigr)
=(-1)^N\Res_\Lambda\bigl(\lambda^{-1}X_{\Lambda,-\Lambda-\nabla}^{\twab}(b \otimes a)\bigr). 
\end{align*}
Thus, we get 
\begin{align*}
 (-1)^{\wt{p}(a)\wt{p}(b)+\ol{N}+\ol{1}}
 \Res_\Lambda\bigl(\lambda^{-1}X_{\Lambda,-\Lambda-\nabla}^{\twab}(b \otimes a)\bigr)
=\Res_\Lambda\bigl(\lambda^{-1}X^{\twab}_{\Lambda,-\Lambda-\nabla}(a \otimes b)\bigr), 
\end{align*}
which means the commutativity of $\mu_X$. 

Conversely, assume the commutativity of $\mu_X$. 
By the sesquilinearity condition \eqref{eq:2P:Pclses3}, we have
\begin{align*}
 X^{\twab}_{-\Lambda-\nabla, \Lambda}(b \otimes S^Ia)
=(-1)^{\wt{p}(b)\ol{\#I}} \theta^IX^{\twab}_{-\Lambda-\nabla,\Lambda}(b \otimes a). 
\end{align*}
Thus, using the sesquilinearity condition \eqref{eq:2P:Pclses3}, one can get
\begin{align*}
\Res_\Lambda\bigl(\lambda^{-1}\theta^IX_{-\Lambda-\nabla, \Lambda}(b \otimes a)\bigr)
&=(-1)^{\wt{p}(b)\ol{\#I}}
  \Res_\Lambda\bigl(\lambda^{-1}X^{\twab}_{-\Lambda-\nabla,\Lambda}(b \otimes S^Ia)\bigr) \\
&=(-1)^{\wt{p}(b)\ol{\#I}}(-1)^N
  \Res_\Lambda\bigl(\lambda^{-1}X_{\Lambda,-\Lambda-\nabla}^{\twab}(b \otimes S^Ia)\bigr) \\
&=(-1)^{\wt{p}(b)\ol{\#I}}(-1)^N(-1)^{p(b)(\ol{N}+\ol{1})}b(S^I a). 
\end{align*}
Hence we find that
\begin{align*}
X_{-\Lambda-\nabla, \Lambda}(b\otimes a)
&=(-1)^N(-1)^{p(b)(\ol{N}+\ol{1})}
  \sum_{I\subset[N]}(-1)^{\wt{p}(b)\ol{\#I}}\sigma(I)\theta^{[N]\bs I}b(S^Ia) \\
&=(-1)^N(-1)^{p(b)(\ol{N}+\ol{1})}
  \sum_{I\subset[N]}(-1)^{\wt{p}(b)\ol{\#I}}(-1)^{(p(a)+\ol{\#I})p(b)}\sigma(I)\theta^{[N]\bs I}(S^Ia)b \\
&=(-1)^{\wt{p}(a)\wt{p}(b)+\ol{1}}X_{\Lambda, -\Lambda-\nabla}^{\twab}(a\otimes b), 
\end{align*}
which is equivalent to the identity \eqref{eq:2P:Xcom}. 
\end{proof}

Recall from \eqref{eq:1P:MC} the linear superspace $L^1(\oP)$.
In what follows, we choose and fix $X \in L^1\bigl(\oP\bigr)_{\ol{1}}$, so that $\{\cdot_\Lambda \cdot\}_X\colon V \otimes V \to V[\Lambda]$ and $\mu_X\colon V \otimes V \to V$ are linear maps satisfying \cref{dfn:1P:WLCA} (i), (ii) and \cref{dfn:1P:PVA} (i), (ii). 
Then, it is enough to prove:

\begin{clm}\label{clm:2P:MC}
The Jacobi identity \eqref{eq:1P:LCA:Jac} of $\{\cdot_\Lambda\cdot\}_X$, the associativity of $\mu_X$,
and the Leibniz rule \eqref{eq:1P:Leib} are equivalent to the Maurer-Cartan condition $X \square X=0$. 
\end{clm}

Since $X \square X$ is invariant under the $\frS_3$-action, we have
\begin{align*}
 X \square X=0 \iff \text{$(X\square X)^Q=0$ for }
 Q = \thr, \quad \bullet \hspace{12.5pt} \bullet \to \bullet \ \text{ and } \ \thabc. 
\end{align*}
Hereafter until the end of this \cref{ss:2P:NW} we suppress $X$ in $\{\cdot_\Lambda\cdot\}_X$ and $\mu_X$. 

In the proof of the next \cref{lem:2P:XsqJac,lem:2P:XsqLeib,lem:2P:Xsqass}, we use the following notation: 
For a superoperad $\oP$, $m,n \in \bbN$ and $i \in [m]$, we denote 
\begin{align}\label{eq:1:circ_i}
 \circ_i\colon \oP(m)\otimes \oP(n) \lto \oP(m+n-1), \quad f \circ_i g \ceq 
 f \circ (1 \odot \cdots \odot 1 \odot \overset{i}{\check{g}} \odot 1\odot \cdots \odot 1), 
\end{align}
which is called the infinitesimal composition.

\begin{lem}\label{lem:2P:XsqJac}
For $X \in L^1\bigl(\oPclW{\Pi^{N+1}V}\bigr)_{\od}$, we have
\begin{align*}
&(X\square X)^{\thr}_{\Lambda_1,\Lambda_2,-\Lambda_1-\Lambda_2-\nabla}(a \otimes b \otimes c) \\
&=\pm \bigl(\{a_{\Lambda_1}\{b_{\Lambda_2}c\}\}
  -(-1)^{(p(a)+\ol{N})\ol{N}}\{\{a_{\Lambda_1}b\}_{\Lambda_1+\Lambda_2}c\}
  -(-1)^{(p(a)+\ol{N})(p(b)+\ol{N})}\{b_{\Lambda_2}\{a_{\Lambda_1}c\}\} \bigr)
\end{align*}
with $\pm\ceq (-1)^{N+1}(-1)^{p(a)\ol{N}}(-1)^{p(b)(\ol{N}+\ol{1})}$.  
Hence, the Jacobi identity \eqref{eq:1P:LCA:Jac} is equivalent to 
\[
 (X \square X)^{\thr} = 0.
\] 
\end{lem}

\begin{proof}
By a direct calculation, we have
\begin{align*}
  (X \circ_1 X)^{\thr}_{\Lambda_1,\Lambda_2,-\Lambda_1-\Lambda_2-\nabla}(a \otimes b \otimes c)
&=(-1)^{p(b)(\ol{N}+\ol{1})}\{\{a_{\Lambda_1}b\}_{\Lambda_1+\Lambda_2}c\}, 
\\
  (X \circ_2 X)^{\thr}_{\Lambda_1,\Lambda_2,-\Lambda_1-\Lambda_2-\nabla}(a \otimes b \otimes c)
&=(-1)^{\wt{p}(a)}(-1)^{p(a)(\ol{N}+\ol{1})} (-1)^{p(b)(\ol{N}+\ol{1})}\{a_{\Lambda_1}
  \{b_{\Lambda_2}c\}\}, 
\\
  \bigl((X \circ_2 X)^{(1, 2)}\bigr)^{\thr}_{\Lambda_1,\Lambda_2,-\Lambda_1-\Lambda_2-\nabla}
  (a \otimes b \otimes c)
&=(-1)^{\wt{p}(a)\wt{p}(b)} (-1)^{\wt{p}(b)} (-1)^{p(a)(\ol{N}+\ol{1})} (-1)^{p(b)(\ol{N}+\ol{1})}
  \{b_{\Lambda_2}\{a_{\Lambda_1}c\}\}, 
\end{align*}
which proves the first equality. The second equivalence is clear.
\end{proof}

\begin{lem}\label{lem:2P:XsqLeib}
For $X \in L^1\bigl(\oPclW{\Pi^{N+1}V}\bigr)_{\ol{1}}$, we have
\begin{align*}
  (X \square X)^{\thbc}_{\Lambda_1,\Lambda_2,-\Lambda_1-\Lambda_2-\nabla}(a \otimes b \otimes c)
&=\pm   \sum_{I\subset [N]} (-1)^{\wt{p}(a)\ol{\#I}} \sigma(I) \theta_2^{[N] \bs I} \\
&\qquad \times \bigl(\{a_{\Lambda_1}(S^Ib)c\} - \{a_{\Lambda_1}S^Ib\}c 
                     - (-1)^{(p(a)+\ol{N})(p(b)+\ol{\#I})}(S^Ib)\{a_{\Lambda_1}c\}\bigr).
\end{align*}
with $\pm \ceq (-1)^{p(b)(\ol{N}+\ol{1})+\ol{1}}$. 
Hence, the Leibniz rule \eqref{eq:1P:Leib} is equivalent to 
\[
 (X \square X)^{\thbc} = 0. 
\]
\end{lem}

\begin{proof}
A direct calculation shows that
\begin{align*}
 \Res_{\Lambda_2}\bigl(\lambda_2^{-1} (X \circ_1 X)^{\thbc}_{\Lambda_1,\Lambda_2,-\Lambda_1-\Lambda_2-\nabla}
  (a \otimes b \otimes c) \bigr) &=\epsilon_1 \{a_{\Lambda_1}b\}c, \\
 \Res_{\Lambda_2}\bigl(\lambda_2^{-1} (X \circ_2 X)^{\thbc}_{\Lambda_1,\Lambda_2,-\Lambda_1-\Lambda_2-\nabla}
  (a \otimes b \otimes c) \bigr) &=\epsilon_2 \{a_{\Lambda_1}bc\}, \\
 \Res_{\Lambda_2}\bigl(\lambda_2^{-1} 
  ((X \circ_2 X)^{(1,2)})^{\thbc}_{\Lambda_1,\Lambda_2,-\Lambda_1-\Lambda_2-\nabla}
  (a \otimes b \otimes c) \bigr) &=\epsilon_3 b\{a_{\Lambda_1}c\},
\end{align*}
where
\begin{align*}
 \epsilon_1 &\ceq (-1)^{p(a)(\ol{N}+\ol{1})} (-1)^{(p(a)+p(b)+\ol{N})(\ol{N}+\ol{1})}, \\
 \epsilon_2 &\ceq (-1)^{\wt{p}(a)}(-1)^{(\wt{p}(a)+\ol{1})\ol{N}}
                  (-1)^{p(a)(\ol{N}+\ol{1})} (-1)^{p(b)(\ol{N}+\ol{1})} = \pm, \\
 \epsilon_3 &\ceq (-1)^{\wt{p}(a)\wt{p}(b)}  (-1)^{\wt{p}(b)} (-1)^{p(a)(\ol{N}+\ol{1})}
                  (-1)^{p(b)(\ol{N}+\ol{1})}. 
\end{align*}
Thus we have
\begin{align*}
 \Res_{\Lambda_2}\bigl(\lambda_2^{-1}
  (X \square X)^{\thbc}_{\Lambda_1,\Lambda_2,-\Lambda_1-\Lambda_2-\nabla}(a \otimes b \otimes c)\bigr)
=\pm\bigl(\{a_{\Lambda_1}bc\}-\{a_{\Lambda_1}b\}c-(-1)^{(p(a)+\ol{N})p(b)}b\{a_{\Lambda_1}c\}\bigr),
\end{align*}
Using the sesquilinearity conditions \eqref{eq:2P:Pclses1} and \eqref{eq:2P:Pclses3}, 
we have the first statement. The second equivalence is clear.
\end{proof}

\begin{lem}\label{lem:2P:Xsqass}
For $X \in L^1\bigl(\oPclW{\Pi^{N+1}V}\bigr)_{\od}$, we have
\begin{align*}
&(X \square X)^{\thabc}_{\Lambda_1, \Lambda_2, -\Lambda_1-\Lambda_2-\nabla}(a\otimes b\otimes c)\\
&=\pm (-1)^N \sum_{I,J \subset [N]} (-1)^{p(a)\ol{\#J}}(-1)^{\#I\#J+\#J}\sigma(I)\sigma(J)
\times \theta_1^{[N] \bs I}\theta_2^{[N]\bs J}\bigl(((S^Ia)b)c-(S^Ia)(bc)\bigr),
\end{align*}
with $\pm\ceq (-1)^{p(b)(\ol{N}+\ol{1})}$. 
Hence, the associativity is equivalent to 
\[
 (X \square X)^{\thabc} = 0. 
\]
\end{lem}

\begin{proof}
A direct calculation yields that
\begin{align*}
 \Res_{\Lambda_1} \Res_{\Lambda_2} \bigl(\lambda_1^{-1}\lambda_2^{-1}
  (X \circ_1 X)^{\thabc}_{\Lambda_1,\Lambda_2,-\Lambda_1-\Lambda_2-\nabla}(a \otimes b \otimes c) \bigr)
&=\epsilon_1(ab)c, \\
 \Res_{\Lambda_1} \Res_{\Lambda_2} \bigl(\lambda_1^{-1}\lambda_2^{-1}
  (X \circ_2 X)^{\thabc}_{\Lambda_1, \Lambda_2, -\Lambda_1-\Lambda_2-\nabla}(a \otimes b \otimes c) \bigr)
&=\epsilon_2a(bc), 
\end{align*}
where 
\begin{align*}
\epsilon_1 &\ceq (-1)^{(p(a)+p(b)+\ol{N})(\ol{N}+\ol{1})} (-1)^{p(a)(\ol{N}+\ol{1})}=\pm, \\
\epsilon_2 &\ceq (-1)^{\wt{p}(a)}(-1)^{(\wt{p}(a)+\ol{1})\ol{N}} (-1)^{p(a)(\ol{N}+\ol{1})}
                 (-1)^{p(b)(\ol{N}+\ol{1})}. 
\end{align*}
Also, by the first cycle relation of $X$, we find that
\begin{align*}
 \bigl((X\circ_2X)^{(1,2)}\bigr)^{\thabc}_{\Lambda_1,\Lambda_2,-\Lambda_1-\Lambda_2-\nabla}
 (a \otimes b \otimes c) = 0. 
\end{align*}
Thus we have
\begin{align*}
 \Res_{\Lambda_1} \Res_{\Lambda_2}\bigl(\lambda_1^{-1}\lambda_2^{-1}
  (X \square X)^{\thabc}_{\Lambda_1,\Lambda_2,-\Lambda_1-\Lambda_2-\nabla}(a \otimes b \otimes c)\bigr)
 = \pm \bigl((ab)c-a(bc)\bigr).
\end{align*}
Using the sesquilinearity conditions, we obtain the first equality. The second equivalence is clear.
\end{proof}

By \cref{lem:2P:XsqJac,lem:2P:XsqLeib,lem:2P:Xsqass}, we have \cref{clm:2P:MC}.
Now the proof of \cref{thm:2P:NWPVA} is complete. 

\subsection{The superoperad of \texorpdfstring{$N_K=N$}{NK=N} SUSY Poisson vertex algebras}
\label{ss:2P:NK}

All the results and proofs given in \cref{ss:2P:NW} for the $N_W=N$ case are valid for the $N_K=N$ case with the following modifications: 
\begin{itemize}
\item 
Replace the superalgebra $\clH_W$  (\cref{dfn:1P:clHW}) with $\clH_K$ (\cref{dfn:1P:clHK}). 

\item 
Replace $(1|N)_W$-supervariables \eqref{eq:1P:polyW} with $(1|N)_K$-supervariables \eqref{eq:1P:polyK} . 
\end{itemize}

\begin{dfn}
Let $\Lambda_k=(\lambda_k,\theta_k^1,\dotsc,\theta_k^N)$ be a $(1|N)_K$-supervariable for each $k \in \bbZ_{>0}$. For an $\clH_K$-supermodule $V=(V,\nabla)$ and $n \in \bbN$, define $\oPclK{V}(n)$ to be the linear superspace of all linear maps 
\begin{align*}
 X\colon V^{\otimes n} \otimes \bbK\oQ(n) \lto V_\nabla[\Lambda_k]_{k=1}^n
\end{align*}
satisfying the conditions (i), (ii) in \cref{dfn:2P:Pcl}, replacing $(1|N)_W$-supervariables $\Lambda_1, \ldots, \Lambda_n$ by $(1|N)_K$-supervariables. 
\end{dfn}

Then we can define an $\frS_n$-action and the composition maps on $\oPclK{V}(n)$, and we have a superoperad $\oPclK{V} \ceq \bigl(\oPclK{V}(n)\bigr)_{n \in \bbN}$. 
It encodes the $N_K=N$ SUSY VPA structures in the following sense:

\begin{thm}\label{thm:2P:NKPVA}
Let $V=(V, \nabla)$ be an $\clH_K$-supermodule. 
\begin{enumerate}
\item 
For an odd Maurer-Cartan solution $X \in \MC\big(L\bigl(\oPclK{\Pi^{N+1}V}\bigr)\bigr)_{\ol{1}}$, define linear maps $\{\cdot_\Lambda\cdot\}_X\colon V \otimes V \to V[\Lambda]$ and $\mu_X\colon V \otimes V \to V$ by
\begin{align*}
 \{a_\Lambda b\}_X  &\ceq (-1)^{p(a)(\ol{N}+\ol{1})} X_{\Lambda, -\Lambda-\nabla}^{\tw}(a \otimes b), \\
 \mu_X(a \otimes b) &\ceq (-1)^{p(a)(\ol{N}+\ol{1})} 
 \Res_\Lambda\bigl(\lambda^{-1}X_{\Lambda,-\Lambda-\nabla}^{\twab}(a \otimes b)\bigr)
\end{align*}
for each $a,b\in V$. Then $(V, \nabla, \{\cdot_\Lambda\cdot\}_X, \mu_X)$ is a non-unital $N_K=N$ SUSY Poisson vertex algebra. 

\item 
The correspondence $X \mto (\{\cdot_\Lambda\cdot\}_X, \mu_X)$ gives a bijection
\begin{align*}
 \MC\bigl(L\bigl(\oPclK{\Pi^{N+1}V}\bigr)\bigr)_{\ol{1}} \lsto
 \{\textup{non-unital $N_K=N$ SUSY PVA structures on $(V,\nabla)$}\}. 
\end{align*}
\end{enumerate}
\end{thm}

\begin{dfn}
We call the superoperad $\oPclK{V}$ \emph{the $N_K=N$ SUSY coisson operad of $V=(V,\nabla)$}.
\end{dfn}

\section{Relation to the SUSY chiral operad}\label{s:3P}

One can construct graded Poisson vertex algebras from filtered vertex algebras.
There are several variations of the construction such as \cite{BDHK}, \cite[16.2.3]{FBZ} and \cite{Li1,Li2}. 
In this note, we consider a SUSY analogue of the construction using an arbitrary increasing filtration on the underlying linear superspace \cite[\S8.5]{BDHK}.

Let us first introduce a ``good'' filtration for a SUSY vertex algebra.

\begin{dfn}\label{dfn:3P:filter}
An $\clH_W$-supermodule $V=(V,\nabla)$ is called filtered if it is equipped with an increasing sequence of $\clH_W$-submodules
\begin{align*}
 V_0 \subset V_1 \subset \cdots \subset V_r \subset V_{r+1} \subset \cdots \subset V.
\end{align*}
We denote such a filtered $\clH_W$-module by $(V,(V_r)_{r \in \bbN})$.
The associated graded is denoted by 
\[
 \gr V \ceq \tboplus_{r \in \bbN} V_r/V_{r-1},
\]
which is naturally a graded $\clH_W$-supermodule. Here we use the convention $V_{-1}\ceq 0$. 

A non-unital $N_W=N$ SUSY vertex algebra $V$ is filtered if $V$ is filtered as an $\clH_W$-supermodule and the filtration $(V_r)_{r\in\bbN}$ satisfies
\begin{align*}
 [a_\Lambda b] \in V_{r+s-1}[\Lambda], \quad ab \in V_{r+s}. 
\end{align*}
for each $a \in V_r$ and $b \in V_s$. We denote a filtered $N_W=N$ SUSY VA by $(V,(V_r)_{r \in \bbN})$.
\end{dfn}

\begin{rmk}
Any $\clH_W$-supermodule $V$ can be seen as a filtered $\clH_W$-supermodule by letting $V_r\ceq V$ for each $r\in\bbN$. We call this the trivial filtration of $V$. 
\end{rmk}

The proof of the following statement is straightforward and we omit it.

\begin{prp}\label{prp:3P:grPVA}
Given a filtered SUSY vertex algebra $(V,(V_r)_{r \in \bbN})$, we have a natural SUSY PVA structure on the associated graded $\gr V \ceq \tboplus_{r \in \bbN} V_r/V_{r-1}$.
The operation is given by 
\begin{align*}
\bigl[\ol{a}^r_\Lambda\ol{b}^s\bigr]\ceq \ol{[a_\Lambda b]}^{r+s-1}, \quad 
\ol{a}^r\ol{b}^s\ceq \ol{ab}^{r+s}
\end{align*}
for each $a\in V_r$ and $b\in V_s$. Here $\ol{a}^r$ denotes the canonical projection $V_r\to V_r/V_{r-1}$. 
\end{prp}

In this subsection, we study this construction of graded SUSY PVAs from the viewpoint of the SUSY chiral/coisson operads along the non-SUSY argument in \cite[\S8, \S10.4]{BDHK}. We will show for a filtered $\clH_W$-supermodule $V=(V,(V_r)_{r \in \bbN})$ the following statements.
\begin{itemize}
\item 
The SUSY chiral operad $\oPclW{V}$ is filtered, and we have the associated graded superoperad $\gr\oPchW{V}$ (\cref{prp:3P:filtPch}).
\item
There is an injection $\alpha\colon \gr\oPchW{V} \inj \oPclW{\gr V}$ of superoperads (\cref{thm:3P:inj}).
\item
The SUSY PVA structures given by $\alpha$ and \cref{prp:3P:grPVA} coincide (\cref{rmk:3P:2PVAstr}). 
\end{itemize}
Although we give detailed arguments only for the $N_W=N$ case, similar statements hold for the $N_K=N$ case.

\subsection{Grading and filtration of superoperad}\label{ss:3P:gr}

We cite the notion of filtered and graded superoperads from \cite[(3.10)]{BDHK}. Recall the symbol $\bbN^m_n \ceq \{(n_1,\dotsc,n_m) \in \bbN^m \mid n_1+\dotsb+n_m=n \}$ in \eqref{eq:1P:Nmn}.
For an $\frS$-supermodule $\oP=(\oP(n))_{n \in \bbN}$, an $\frS$-submodule $\oQ \subset \oP$ means an $\frS$-supermodule $\oQ=(\oQ(n))_{n \in \bbN}$ such that $\oQ(n)$ is an $\frS_n$-submodule of $\oP(n)$ for each $n \in \bbN$.

\begin{dfn}
Let $\oP$ be a superoperad. A sequence $(\oP_r)_{r \in \bbN}$ consisting of $\frS$-submodules of $\oP$ is called a grading of $\oP$ if it satisfies the following conditions: 
\begin{clist}
\item 
We have $\oP=\bigoplus_{r\in\bbN}\oP_r$, i.e., $\oP(n)=\bigoplus_{r\in\bbN}\oP_r(n)$ for each $n\in\bbN$. 

\item 
For $m, n\in\bbN$, $(n_1, \ldots, n_m)\in\bbN^m_n$ and $r, s\in\bbN$, $(s_1, \ldots, s_m)\in\bbN^m_s$, if $X\in\oP_r(m)$ and $Y_i\in\oP_{s_i}(n_i)$ with $i\in[m]$, then $X\circ (Y_1\odot\cdots \odot Y_m)\in\oP_{r+s}(n)$. 
\end{clist}
We call a superoperad $\oP$ equipped with a grading $(\oP_r)_{r \in \bbN}$ a graded operad, and denote it by $\oP=\bigoplus_{r \in \bbN}\oP_r$.
\end{dfn}

\begin{dfn}\label{dfn:3P:opfilt}
Let $\oP$ be a superoperad. A sequence $(\oP_r)_{r \in \bbN}$ consisting of $\frS$-submodules of $\oP$ is called an operad filtration of the superoperad $\oP$ if it satisfies the following conditions: 
\begin{clist}
\item 
$1_{\oP} \in \oP_0(1)$, and $\oP_r \supset \oP_{r+1}$ for each $r \in \bbN$. 

\item 
For $m,n \in \bbN$, $(n_1,\ldots,n_m) \in \bbN^m_n$ and $r,s \in \bbN$, $(s_1,\ldots,s_m) \in \bbN^m_s$, if $X \in \oP_r(m)$ and $Y_i \in \oP_{s_i}(n_i)$ with $i \in [m]$, then we have $X \circ (Y_1 \odot \cdots \odot Y_m) \in \oP_{r+s}(n)$. 
\end{clist}
\end{dfn}

\begin{rmk}
Note that the operad filtration in \cref{dfn:3P:opfilt} is decreasing, as opposed to the increasing filtration of an $\clH_W$-supermodule in \cref{dfn:3P:filter}.
\end{rmk}

The following statement can be easily shown, and we omit the proof.

\begin{lem}\label{lem:3P:opfilt}
Let $\oP$ be a superoperad and $(\oP_r)_{r \in \bbN}$ be a sequence of $\frS$-submodules $\oP_r \subset \oP$ satisfying (i) in \cref{dfn:3P:opfilt}. Then $(\oP_r)_{r\in \bbN}$ is an operad filtration of the superoperad $\oP$ if and only if it satisfies the following condition: 
\begin{itemize}
\item
For $m,n,r,s \in \bbN$, if $X \in \oP_r(m)$ and $Y \in \oP_{s}(n)$, 
then we have $X \circ_1 Y \in \oP_{r+s}(m+n-1)$.
\end{itemize}
\end{lem}

For a superoperad $\oP$ equipped with an operad filtration $(\oP_r)_{r \in \bbN}$, 
we can endow the $\frS$-supermodule
\begin{align*}
 \gr\oP \ceq \bigoplus_{r \in \bbN} \oP_r/\oP_{r+1}, \quad 
 \oP_r/\oP_{r+1} \ceq \bigl(\oP_r(n)/\oP_{r+1}(n)\bigr)_{n \in \bbN}. 
\end{align*}
with a graded superoperad structure. To explain that, we denote by
\begin{align*}
 \oP_r(n) \lto \oP_r(n)/\oP_{r+1}(n), \quad X \lmto \ol{X}^r
\end{align*}
the canonical projection for each $n,r \in \bbN$. 
Then, for $m,n \in \bbN$ and $\nu=(n_1,\ldots,n_m) \in \bbN^m_n$, there exists a unique linear map
\begin{align*}
 \gamma_\nu\colon \gr\oP(m) \otimes \gr\oP(\nu) \lto \gr\oP(n)
\end{align*}
satisfying
\begin{align*}
 \gamma_\nu(\ol{X}^r \otimes \ol{Y_1}^{s_1} \otimes \cdots \otimes \ol{Y_k}^{s_k})
=\ol{X \circ (Y_1 \odot \cdots \odot Y_m)}^{r+s} \quad (X\in \oP_r(m), Y_i\in\oP_{s_i}(n_i))
\end{align*}
for $r,s \in \bbN$, $(s_1,\ldots, s_m) \in \bbN^m_s$. 
Then, $\gr\oP$ is a superoperad with the composition map $\gamma=(\gamma_\nu)_{\nu}$ and the unit $\ol{1_{\oP}}^0\in\gr\oP(1)$. 
Also, the sequence $(\oP_r/\oP_{r+1})_{r \in \bbN}$ is a grading of $\gr\oP$. 

\begin{dfn}\label{dfn:3P:groP}
We call the obtained graded superoperad $\gr\oP$ 
the associated graded operad of the filtered superoperad $(\oP, (\oP_r)_{r \in \bbN})$. 
\end{dfn}

\subsection{An operad filtration of $\oPchW{V}$}\label{ss:3P:filt}

In this subsection, for a filtered $\clH_W$-supermodule $V$, we define an operad filtration of the $N_W=N$ SUSY chiral operad $\oPchW{V}$. 

First, recall the superalgebra $\clO_n^{\star\rmT}$ in \cref{dfn:1P:O}:
\[
 \clO_n^{\star\rmT} = \bbK[z_{k,l}^{\pm1},\zeta_{k,l}^i \mid i \in [N], \ 1 \le k<l \le n], \quad
 z_{k,l} = z_k-z_l, \quad \zeta^i_{k,l} = \zeta^i_k - \zeta^i_l.
\]
For $n \in \bbZ_{>0}$, we set 
\[
 \clO^{\rmT}_n \ceq \bbK[z_{k,l},\zeta_{k,l}^i \mid i \in [N], \ 1 \le k<l \le n].
\]
We also set $\clO_0^{\rmT}\ceq \bbK$. 
We introduce an increasing filtration on $\clO_n^{\star\rmT}$, which is an $N_W=N$ SUSY analogue of \cite[\S8.1]{BDHK}.

\begin{dfn}\label{dfn:3P:FrOn}
For $n,r \in \bbN$, we define the linear subspace $F^r\clO_n^{\star\rmT} \subset \clO_n^{\star\rmT}$ by 
\begin{align*}
&F^r\clO_0^{\star\rmT} = F^r\clO_1^{\star\rmT} \ceq \bbK, \\
&F^r\clO_n^{\star\rmT} \ceq \Span_\bbK\{z_{k_1,l_1}^{-m_1} \cdots z_{k_r,l_r}^{-m_r}a \mid 
  m_s \in \bbN, \, I_s \subset [N], \, 1 \le k_s < l_s \le n, \, a \in \clO_n^{\rmT}\} \quad (n>1). 
\end{align*}
\end{dfn}

\begin{eg}
For $n=2$, we have 
\begin{align*}
 F^0\clO_2^{\star\rmT}=\bbK, \quad 
 F^r\clO_2^{\star\rmT}=\clO_2^{\star\rmT} \quad (r \ge 1). 
\end{align*}
\end{eg}

The next lemma follows immediately from \cref{dfn:3P:FrOn} of $F^r\clO_n^{\star\rmT}$.

\begin{lem}\label{lem:3P:FrOn}
Let $n,r,s \in \bbN$. 
\begin{enumerate}
\item
$F^r\clO_n^{\star\rmT} \subset F^{r+1}\clO_n^{\star\rmT}$.

\item 
$\bigl(F^r\clO_n^{\star\rmT}\bigr)\bigl(F^s\clO_n^{\star\rmT}\bigr) \subset F^{r+s}\clO^{\star\rmT}_n$. 

\item \label{i:lem:3P:FrOn:3}
The space $F^r\clO_n^{\star\rmT}$ is closed under the action of $\frS_n$, i.e., 
if $\sigma \in \frS_n$ and $f \in F^r\clO^{\star\rmT}_n$, then $\sigma f \in F^r\clO^{\star\rmT}_n$,
where $(\sigma f)(Z_1,\ldots,Z_n) \ceq f(Z_{\sigma(1)}, \ldots, Z_{\sigma(n)})$ as before. 
\end{enumerate}
\end{lem}

\begin{lem}\label{lem:3P:Frdecomp}
Let $m,n \in \bbZ_{>0}$ and $r \in \bbN$. 
For $f=z_{k_1,l_1}^{m_1} \cdots z_{k_r,l_r}^{m_r}a \in F^r\clO_{m+n}^{\star\rmT}$, 
there exists $f_0 \in F^{r_0}\clO_{m+n}^{\star\rmT}$ and $f_1 \in F^{r_1}\clO_n^{\star\rmT}$ 
with $(r_0,r_1) \in \bbN^2_r$ satisfying the following conditions.
\begin{clist}
\item 
$f=f_0 f_1$. 

\item 
$f_0$ has no pole at $z_k=z_l$ for $1 \le k < l \le n$. 

\item 
$f_1$ has the form $f_1=\prod_{1 \le  k < l \le n} z_{k,l}^{-m_{k,l}}$ with $m_{k,l} \in \bbN$. 
\end{clist}
\end{lem}

\begin{proof}
For $r=0$, the claim is trivial. For $r\ge 1$, we prove the claim by induction on $r$. 

Let $r=1$. Then $f$ is written in the form $f=z_{k,l}^{-m_{k,l}}a$. Thus, 
\begin{align*}
 (f_1,f_0) \ceq 
 \begin{cases}
 (z_{k,l}^{-m},a) & (1 \le k < l \le n) \\
 (1_\bbK,z_{k,l}^{-m}a) & (\text{otherwise})
 \end{cases}
\end{align*}
satisfies the conditions of the statement. 

Let $r>1$. We can write $f$ as 
\begin{align*}
 f = z_{k_1,l_1}^{-m_1}g, \quad g \ceq z_{k_2,l_2}^{-m_2} \cdots z_{k_r,l_r}^{-m_r}a, 
\end{align*}
and then $g \in F^{r-1}\clO_{m+n}$. 
Thus, by the induction hypothesis, there exists $g_0 \in F^{s_0}\clO_{m+n}^{\star\rmT}$ and $g_1 \in F^{s_1}\clO_n^{\star\rmT}$ with $(s_0,s_1) \in \bbN^2_{r-1}$ such that $f=z_{k_1,l_1}^{-m_1}g_1g_0$ and $g_0,g_1$ satisfies the condition (ii), (iii), respectively. Hence
\begin{align*}
 (f_1,f_0)\ceq 
 \begin{cases}
 (z_{k_1,l_1}^{-m_1}g_1,g_0) & (1\le k_1<l_1\le n) \\
 (g_1,z_{k_1,l_1}^{-m_1}g_0) & (\text{otherwise})
 \end{cases}
\end{align*}
satisfies the conditions of the statement.
\end{proof}

Next, we introduce an operad filtration of the $N_W=N$ SUSY chiral operad $\oPchW{V}$,
following \cite[\S8.5]{BDHK}.
Recall \cref{dfn:3P:filter} of a filtration on an $\clH_W$-supermodule $V=(V,\nabla)$.

\begin{dfn}
Let $(V, (V_r)_{r \in \bbN})$ be a filtered $\clH_W$-supermodule. For $n,r \in \bbN$, we denote
\begin{align*}
 F^rV^{\otimes n}\ceq \sum_{(r_1,\ldots,r_n) \in \bbN^n_r} V_{r_1} \otimes \cdots \otimes V_{r_n}. 
\end{align*}
For $n,r \in \bbN$, we define a linear subspace $F_r\oPchW{V}(n)\subset \oPchW{V}(n)$ by
\begin{align*}
 F_r\oPchW{V}(n) \ceq \bigl\{X \in \oPchW{V}(n) \mid 
  X(F^sV^{\otimes n} \otimes F^t\clO^{\star\rmT}_n)) \subset 
  V_{s+t-r, \nabla}[\Lambda_k]_{k=1}^n \ (s,t \in \bbN) \bigr\}
\end{align*}
with the convention $V_s \ceq 0$ for $s\in\bbZ_{<0}$. 
\end{dfn}

\begin{eg}
For $\clH_W$-supermodule $V$ with the trivial filtration, we have
\begin{align*}
 F_r\oPchW{V}(n)
 =\{X \in \oPchW{V}(n) \mid X(V^{\otimes n} \otimes F_{r-1}\clO^{\star\rmT}_n)=0\}
\quad (r\in\bbN)
\end{align*}
with the convention $F_{-1}\clO^{\star\rmT}_n \ceq 0$ for $r=0$. 
\end{eg}

By \cref{lem:3P:FrOn} \ref{i:lem:3P:FrOn:3}, $F_r\oPchW{V}(n)$ is a right $\frS_n$-submodule of $\oPchW{V}(n)$ for $n,r \in \bbN$. Thus, for each $r \in \bbN$, we have an $\frS$-submodule $F_r\oPchW{V} \ceq \bigl(F_r\oPchW{V}(n)\bigr)_{n \in \bbN}$ of $\oPchW{V}$. Moreover:

\begin{prp}\label{prp:3P:filtPch}
For a filtered $\clH_W$-supermodule $(V, (V_r)_{r \in \bbN})$, 
the sequence $\bigl(F_r\oPchW{V}\bigr)_{r \in \bbN}$ is an operad filtration of $\oPchW{V}$. 
\end{prp}

\begin{proof}
The condition (i) in \cref{dfn:3P:opfilt} is clear. 
The condition (ii) in \cref{dfn:3P:opfilt} follows from \cref{lem:3P:opfilt} using \cref{lem:3P:Frdecomp}. 
\end{proof}

Thus, by the argument before \cref{dfn:3P:groP}, we have the associated graded operad $\gr\oPchW{V}$ of the filtered superoperad $\bigl(\oPchW{V}, (F_r\oPchW{V})_{r\in\bbN}\bigr)$: 
\begin{align*}
 \gr  \oPchW{V} = \bigoplus_{r\in\bbN} F_r\oPchW{V}/F_{r+1}\oPchW{V}. 
\end{align*}

Now, recall \eqref{eq:1P:VALbraX}, \eqref{eq:1P:VAproX} and \cref{thm:1P:opNWVA}. 

\begin{prp}\label{prp:3P:filNWVA}
Let $(V, (V_r)_{r\in\bbN})$ be a filtered $\clH_W$-supermodule.
\begin{enumerate}
\item 
For $X \in \oPchW{\Pi^{N+1}V}(2)_{\ol{1}}$, the condition $X \in F^1\oPchW{\Pi^{N+1}V}(2)_{\ol{1}}$ is equivalent to the following condition: 
For $a \in V_r$ and $b \in V_s$ with $r,s \in \bbN$, we have
\begin{align}\label{eq:3P:opfilVA}
 [a_\Lambda b]_X \in V_{r+s-1}[\Lambda], \quad  \mu_X(a \otimes b) \in V_{r+s}. 
\end{align}

\item 
The correspondence $X \mto ([\cdot_\Lambda \cdot]_X, \mu_X)$ gives a bijection
\begin{align*}
 \MC&\bigl(L\bigl(\oPchW{\Pi^{N+1}V}\bigr)\bigr)_{\ol{1}}
     \cap F_1\oPchW{\Pi^{N+1}V}(2)_{\ol{1}}\\
    &\lsto \{\text{non-unital filtered $N_W=N$ SUSY VA structures on $(V, (V_r)_{r \in \bbN})$}\}. 
\end{align*}
\end{enumerate}
\end{prp}

\begin{proof}
Let $X \in F_1\oPchW{\Pi^{N+1}V}(2)_{\ol{1}}$, $a \in V_r$ and $b \in V_s$. 
Since $1_\bbK\in F^0\clO_2^{\star\rmT}$ and $z_{1, 2}^{-1}\in F^1\clO_2^{\star\rmT}$, we have
\begin{align*}
 X_{\Lambda_1,\Lambda_2}(a \otimes b \otimes 1_\bbK) \in V_{r+s-1,\nabla}[\Lambda_k]_{k=1,2}, \quad 
 X_{\Lambda_1,\Lambda_2}(a \otimes b \otimes z_{1,2}^{-1}) \in V_{r+s,\nabla}[\Lambda_k]_{k=1,2}, 
\end{align*}
which implies $[a_\Lambda b]_X \in V[\Lambda]_{r+s-1}$ and $\mu_X(a \otimes b)\in V_{r+s}$. 

Conversely, suppose that the condition \eqref{eq:3P:opfilVA} holds. 
Then, for $a \otimes b \in F^sV^{\otimes 2}$, we have
\begin{align}
\label{eq:3P:ab1s-1}
&X_{\Lambda, -\Lambda-\nabla}(a \otimes b \otimes 1_\bbK) \in V_{s-1}[\Lambda], \\
\label{eq:3P:abzs}
&X_{\Lambda, -\Lambda-\nabla}(a \otimes b \otimes z_{1,2}^{-1}) \in V_s[\Lambda]. 
\end{align}
By \eqref{eq:3P:ab1s-1}, we have 
$X(F^sV^{\otimes 2} \otimes F^0\clO_2^{\star\rmT}) \subset V_{s-1,\nabla}[\Lambda_k]_{k=1,2}$. 
Also, by \eqref{eq:3P:ab1s-1}, \eqref{eq:3P:abzs} and using the sesquilinearity of $X$, we get
\begin{align*}
&X_{\Lambda_1,\Lambda_2}\bigl(a \otimes b \otimes Z_{1,2}^{m   |I}\bigr) 
 \in V_{s-1,\nabla}[\Lambda_k]_{k=1,2} \subset V_{s,\nabla}[\Lambda_k]_{k=1,2}, \\
&X_{\Lambda_1,\Lambda_2}\bigl(a \otimes b \otimes Z_{1,2}^{-m-1|I}\bigr)
 \in V_{s,  \nabla}[\Lambda_k]_{k=1,2}
\end{align*}
for each $m \in \bbN$ and $I \subset [N]$. Thus, for $t \in \bbZ_{>0}$, 
\begin{align*}
 X(F^sV^{\otimes 2} \otimes F^t\clO_n^{\star\rmT}) \in V_{s,\nabla}[\Lambda_k]_{k=1,2}
 \subset V_{s+t-1,\nabla}[\Lambda_k]_{k=1,2}. 
\end{align*}
Hence we have $X\in F^1\oPchW{\Pi^{N+1}V}(2)_{\ol{1}}$. 
\end{proof}

\subsection{An embedding of $\gr\oPchW{V}$ in $\oPclW{\gr V}$}\label{ss:3P:emb}

In this subsection, we construct an injective morphism of the associated graded $\gr\oPchW{V}$ of the $N_W=N$ SUSY chiral operad into the $N_W=N$ SUSY coisson operad $\oPclW{\gr V}$, following the non-SUSY argument in \cite[\S10.4]{BDHK}. 

First, we introduce a grading on the SUSY coisson operad $\oPclW{V}$.
Recall the terminology on quivers (\cref{dfn:2P:quiver}), the set $\oQ(n)$ of $n$-quivers without loops (\cref{dfn:2P:oQ}), and the notation $X^Q\colon V^{\otimes n} \to V_{\nabla}[\Lambda_k]_{k=1}^n$ for an operation $X \in \oPclW{V}(n)$ and a quiver $Q \in \oQ(n)$ (see \eqref{eq:2P:Pcl-X}).

\begin{dfn}\label{dfn:3P:grPclW}
Let $V=\bigoplus_{r \in \bbN}V_r$ be a graded $\clH_W$-supermodule. 
For $n,r \in \bbN$, we denote by $\oPclW{V,r}(n)$ the linear subspace of $\oPclW{V}(n)$ consisting of $X \in \oPclW{V}(n)$ such that
\begin{align*}
 X^Q(v_1\otimes \cdots \otimes v_n)\in V_{s+t-r, \nabla}[\Lambda_k]_{k=1}^n \quad 
 (v_i \in V_{s_i}, \, Q \in \oQ(n) \text{ with $t$ edges})
\end{align*}
for $s,t \in \bbN$ and $(s_1,\ldots,s_n) \in \bbN^n_s$.
Then we have a graded superoperad $\oPclW{V}=\bigoplus_{r \in \bbN}\oPclW{V,r}$
\end{dfn}

\begin{dfn}
Let $n \in \bbN$. For a quiver $Q = (Q_0,Q_1,s,t) \in \oQ(n)$, we denote
\begin{align*}
 f_Q \ceq \prod_{k,l=1}^n z_{k,l}^{-m_{k,l}} \in \clO_n^{\star\rmT}, 
\end{align*}
where $m_{k,l}$ is the number of edges $\alpha \in Q_1$ with $s(\alpha)=k$ and $t(\alpha)=l$. 
\end{dfn}

The element $f_Q$ is denoted by $p_Q$ in \cite[\S8.3]{BDHK}.
The following statements are easy to verify, and we omit the proof. 

\begin{lem}\label{lem:3P:fQ}
Let $n,r \in \bbN$. 
\begin{enumerate}
\item \label{i:lem:3P:fQ:5}
For $Q \in \oQ(n)$ and $\sigma\in\frS_n$, we have $\sigma f_Q=f_{\sigma Q}$.

\item \label{i:lem:3P:fQ:1}
For $Q \in \oQ(n)$ with $r$ edges, we have $f_Q \in F_r\clO_n^{\star\rmT}$. 

\item \label{i:lem:3P:fQ:2}
For $Q \in \oQ(n) \bs \oQ_{\ac}(n)$ with $r$ edges, we have $f_Q \in F_{r-1}\clO_n^{\star\rmT}$. 

\item \label{i:lem:3P:fQ:3}
For $Q \in \oQ(n)$ and a directed cycle $(\alpha_1,\ldots,\alpha_l)$ of $Q$, 
we have $\sum_{k=1}^l f_{Q \bs \alpha_k}=0$ (see \cref{lem:2P:dgraph} for $Q \bs \alpha_k$). 

\item \label{i:lem:3P:fQ:4}
For $Q \in \oQ(n)$ and a connected component $I^a$ of the underlying graph $\ul{Q}$, 
we have $\sum_{k \in I^a}\pdd_{z_k}f_Q=0$ 
(recall that $I^a$ is regarded as a subset of the vertex set $Q_0$).  
\end{enumerate}
\end{lem}

The following claim is shown in the proof of \cite[Proposition 8.7]{BDHK} (see also \cite[Examples 9.1, 9.5]{BDHK}). 

\begin{lem}\label{lem:3P:f1Qf0}
Let $m,n \in \bbZ_{>0}$, and set $\nu \ceq (n,1,\ldots,1) \in \bbN^m_{m+n-1}$. 
For $j \in [n]$ and $Q \in \oQ(n)$, if $\Delta_0^\nu(Q) \in \oQ_{\ac}(m)$, then we have
\begin{align*}
  \rst{\pdd_{z_j}f_{\Delta_1^\nu(Q)}^{-1}f_Q}{z_1=\cdots=z_n}
= -\sum_{i\in\clE_Q^\nu(j)} \pdd_{z_i} f_{\Delta_0^\nu(Q)}
\end{align*}
under the identification $\bbK[z_n^{\pm1},\ldots,z_{m+n-1}^{\pm1}] \cong \bbK[z_1^{\pm1},\ldots,z_m^{\pm1}]$, $z_{i+n-1} \mto z_i$. 
\end{lem}

Hereafter until the end of this \cref{ss:3P:emb}, 
we fix a filtered $\clH_W$-supermodule $(V, (V_r)_{r\in\bbN})$. 

\begin{dfn}
Let $n,r \in \bbN$. For $X \in F_r\oPchW{V}(n)$, we define a linear map
\begin{align*}
 \wt{X}^r\colon (\gr V)^{\otimes n} \otimes \bbK\oQ(n) \lto (\gr V)_\nabla[\Lambda_k]_{k=1}^n
\end{align*}
by 
\begin{align*}
 \wt{X}^r(\ol{v}_1^{s_1} \otimes \cdots \otimes \ol{v}_n^{s_n} \otimes Q)
 \ceq \ol{X(v_1 \otimes \cdots \otimes v_n \otimes f_Q)}^{s+t-r}
 \quad (v_i \in V_{s_i}, \, Q \in \oQ(n) \text{ with $t$ edges}). 
\end{align*}
for each $s,t \in \bbN$ and $(s_1,\ldots,s_n) \in \bbN^n_s$. 
\end{dfn}

\begin{lem}\label{lem:3P:FrPchPcl}
Let $n,r \in \bbN$. 
\begin{enumerate}
\item 
For any $X \in F_r\oPchW{V}(n)$, we have $\wt{X}^r \in \oPclW{\gr V, r}(n)$. 

\item 
The map 
\begin{align*}
 F_r\oPchW{V}(n) \lto \oPclW{\gr V, r}(n), \quad X \lmto \wt{X}^r
\end{align*}
is a right $\frS_n$-supermodule homomorphism. 

\item 
For $X \in F_r\oPchW{V}(n)$, we have $\wt{X}^r=0$ if and only if $X \in F_{r+1}\oPclW{V}(n)$.
\end{enumerate}
\end{lem}

\begin{proof}
\begin{enumerate}
\item 
The cycle relations for $\wt{X}^r$ follow \cref{lem:3P:fQ} \ref{i:lem:3P:fQ:2}, \ref{i:lem:3P:fQ:3}. 
To show the sesquilinearity conditions, let $s,t \in \bbN$, $(s_1,\ldots,s_n) \in \bbN^n_s$, $v_i \in V_{s_i}$ ($i \in [n]$), and $Q \in \oQ(n)$ with $t$ edges. 
Also, let $I^a$ be a connected component of $\ul{Q}$.
\begin{itemize}
\item 
$\wt{X}^r$ satisfies the sesquilinearity condition \eqref{eq:2P:Pclses1}: 
For $k,l \in I^a$, by definition of $\wt{X}^r$, 
\begin{align*}
 (\pdd_{\lambda_k}-\pdd_{\lambda_l})
  \wt{X}^r(\ol{v}_1^{s_1} \otimes \cdots \otimes \ol{v}_n^{s_n} \otimes Q)
=(\pdd_{\lambda_k}-\pdd_{\lambda_l})\ol{X(v_1 \otimes \cdots \otimes v_n \otimes f_Q)}^{s+t-r}. 
\end{align*}
Since $X\colon V^{\otimes n} \otimes \clO_n^{\star\rmT} \to V_\nabla[\Lambda_k]_{k=1}^n$ is a right $\clD_n^{\rmT}$-supermodule homomorphism, we have 
\begin{align*}
    (\pdd_{\lambda_k}-\pdd_{\lambda_l})X(v_1 \otimes \cdots \otimes v_n \otimes f_Q)
&=-X(v_1 \otimes \cdots \otimes v_n \otimes f_Q) \cdot z_{k,l}  \\
&=-X(v_1 \otimes \cdots \otimes v_n \otimes f_Q z_{k,l}). 
\end{align*}
It is clear that $f_Q z_{k,l} \in F_{t-1}\clO_n^{\star\rmT}$ if $Q \in \oQ_{\ac}(n)$, thus
\begin{align*}
 (\pdd_{\lambda_k}-\pdd_{\lambda_l})
  \wt{X}^r(\ol{v}_1^{s_1} \otimes \cdots \otimes \ol{v}_n^{s_n} \otimes Q) = 0.
\end{align*}
	
\item 
$\wt{X}^r$ satisfies the sesquilinearity condition \eqref{eq:2P:Pclses2}: 
Since $X\colon V^{\otimes n} \otimes \clO_n^{\star\rmT} \to V_\nabla[\Lambda_k]_{k=1}^n$ is a right $\clD_n^{\rmT}$-supermodule homomorphism, we have 
\begin{align*}
&  \sum_{k \in I^a}
    \wt{X}^r\bigl(T^{(k)}(\ol{v}_1^{s_1} \otimes \cdots \otimes \ol{v}_n^{s_n}) \otimes Q\bigr) 
 = \sum_{k \in I^a}\ol{X\bigl(T^{(k)}(v_1 \otimes \cdots \otimes v_n) \otimes f_Q\bigr)}^{s+t-r} \\
&=-\sum_{k \in I^a}\lambda_k \ol{X(v_1 \otimes \cdots \otimes v_n \otimes f_Q)}^{s+t-r}
  +\sum_{k \in I^a}\ol{X(v_1 \otimes \cdots \otimes v_n \otimes \pdd_{z_k}f_Q)}^{s+t-r} \\
&=-\sum_{k \in I^a}\lambda_k \ol{X(v_1 \otimes \cdots \otimes v_n \otimes f_Q)}^{s+t-r} 
 =-\sum_{k \in I^a}\lambda_k \wt{X}^r(\ol{v}_1^{s_1} \otimes \cdots \otimes \ol{v}_n^{s_n} \otimes Q).
\end{align*}
In the third equality, we used \cref{lem:3P:fQ} \ref{i:lem:3P:fQ:4}. 

\item 
The linear map $\wt{X}^r$ satisfies the sesquilinearity condition \eqref{eq:2P:Pclses3}: 
This can be checked by a direct calculation using the fact that 
$X\colon V^{\otimes n}\otimes \clO_n^{\star\rmT} \to V_\nabla[\Lambda_k]_{k=1}^n$ 
is a right $\clD_n^{\rmT}$-supermodule homomorphism. 
\end{itemize}
	
Thus, we have $\wt{X}^r \in \oPclW{\gr V}(n)$, and now $\wt{X}^r \in \oPclW{\gr V,r}(n)$ is clear by \cref{dfn:3P:grPclW} of the grading of $\oPclW{\gr V}$. 

\item 
This can be checked by a direct calculation using \cref{lem:3P:fQ} \ref{i:lem:3P:fQ:5}. 

\item 
It is easy to check that $X \in F_{r+1}\oPchW{V}(n)$ implies $\wt{X}^r=0$. 
Conversely, suppose $\wt{X}^r=0$. By the sesquilinearity of $X$, it is enough to show that
\begin{align*}
 X(v_1 \otimes \cdots \otimes v_n \otimes f) \in V_{s+t-r-1,\nabla}[\Lambda_k]_{k=1}^n, \quad 
 (v_i \in V_{s_i}, \, f \ceq z_{k_1,l_1}^{-1} \cdots z_{k_t, l_t}^{-1}) 
\end{align*}
for $s,t \in V$ and $(s_1,\ldots,s_n) \in \bbN^n_s$ (see also \cite[Lemma 8.3]{BDHK}). 
Using $Q \in \oQ(n)$ such that $f_Q=f$, we have
\begin{align*}
 \ol{X(v_1 \otimes \cdots \otimes v_n \otimes f)}^{s+t-r}
=\wt{X}^r(\ol{v}_1^{s_1} \otimes \cdots \otimes \ol{v}_n^{s_n} \otimes Q)
=0,
\end{align*}
which means $X(v_1 \otimes \cdots \otimes v_n \otimes f) \in V_{s+t-r-1, \nabla}[\Lambda_k]_{k=1}^n$. 
\end{enumerate}
\end{proof}

\begin{lem}\label{lem:3P:c1t=tc1}
Let $m,n \in \bbZ_{>0}$ and $r,s \in \bbN$. 
For $X \in F_r\oPchW{V}(m)$ and $Y \in F_s\oPchW{V}(n)$, we have 
\begin{align}\label{eq:3P:c1t=tc1}
 \wt{X \circ_1 Y}^{r+s} = \wt{X}^r \circ_1 \wt{Y}^s.
\end{align} 
\end{lem}

\begin{proof}
Let $\nu \ceq (n,1,\ldots,1) \in \bbN^m_{m+n-1}$ and $Q \in \oQ(n)$. 
If $\Delta_0^\nu(Q) \in \oQ(m) \bs \oQ_{\ac}(m)$, it is clear that the right hand side of \eqref{eq:3P:c1t=tc1} (evaluating on $Q$) equals to $0$, because $X$ satisfies the first cycle relation. Also, since 
\begin{align*}
 \rst{f_{\Delta_1^\nu(Q)}^{-1}f_Q}{z_1=\cdots =z_n} \in F_{t-1}\clO^{\star\rmT}_n, \quad
 t \ceq \#E(Q) - \#E(\Delta_1^\nu(Q)), 
\end{align*}
the left hand side of \eqref{eq:3P:c1t=tc1} is also 0. 
If $\Delta_1^\nu(Q)\in\oQ_{\ac}(m)$, then one can show \eqref{eq:3P:c1t=tc1} by a direct calculation using \cref{lem:3P:f1Qf0}. 
\end{proof}

By \cref{lem:3P:FrPchPcl}, we have an injective right $\frS_n$-supermodule homomorphism
\begin{align*}
 \alpha_{n,r}\colon F_r\oPchW{V}(n) / F_{r+1}\oPchW{V}(n) \lto \oPclW{\gr V,r}(n), \quad 
 \ol{X}^r \lmto \wt{X}^r. 
\end{align*}
Here $\ol{X}^r$ denotes the equivalent class of $X\in F_r\oPchW{V}(n)$. We define
\begin{align*}
 \alpha_n \ceq \bigoplus_{r \in \bbN} \alpha_{n,r}\colon \gr\oPchW{V}(n) \lto \oPclW{\gr V}(n),
\end{align*}
then by \cref{lem:3P:c1t=tc1}, the sequence $\alpha \ceq (\alpha_n)_{n\in\bbN}$ is a morphism of operads $\gr\oPchW{V}$ to $\oPclW{\gr V}$. 

Summarizing the discussion so far, we have the following main statement of \cref{s:3P}, which is a natural SUSY analogue of \cite[Theorem 10.12]{BDHK}. 

\begin{thm}\label{thm:3P:inj}
For a filtered $\clH_W$-supermodule $(V, (V_r)_{r\in\bbN})$, the sequence $\alpha =(\alpha_n)_{n\in\bbN}$ is an injective morphism of graded superoperads from $\gr\oPchW{V}$ to $\oPclW{\gr V}$. 
\end{thm}

\begin{rmk}\label{rmk:3P:2PVAstr}
For a non-unital filtered SUSY VA $V$, we have a structure of a non-unital SUSY PVA on $\gr V$ by \cref{prp:3P:grPVA}. On the other hand, if we denote by 
\begin{align*}
 X \in \MC\bigl(L(\oPchW{\Pi^{N+1}V})\bigr)_{\ol{1}} \cap F_1\oPchW{\Pi^{N+1}V}(2)_{\ol{1}}
\end{align*}
the element corresponding to the filtered SUSY VA structure in \cref{prp:3P:filNWVA}, then we have another SUSY PVA structure $\alpha_2(\ol{X}^1)=\wt{X}^1$ on $\gr V$. 
By the definition of $\wt{X}^1$, \cref{thm:1P:opNWVA} and \cref{thm:2P:NWPVA}, these two SUSY PVA structures coincide. 
\end{rmk}

\appendix
\section{Poisson cohomology bicomplex and finite cohomology complex}\label{s:AP}

The purpose of this \Cref{s:AP} is to give a remark on the finite operad $P^{\cl}$ \cite[\S10.5]{BDHK}, and the content is logically independent of the main text. Hereafter we use the calligraphy symbol $\oPcl{}$ for the finite operad for consistency with the main text.

By the theory of algebraic operads (see \cite{LV} for example), a Poisson algebra structure on a given linear space $V$ corresponds bijectively to an operad morphism $\alpha\colon \oPoi \to \oHom_{V}$ from the Poisson operad $\oPoi$ \cite[\S 13.3]{LV} to the endomorphism operad $\oHom_{V}$.
The Poisson operad $\oPoi$ is a binary homogeneous quadratic operad, 
and can be constructed out of the Lie operad $\oLie$ and the commutative operad $\oCom$ by means of a distributive law: $\oPoi = \oCom \circ \oLie$.
As an $\frS_n$-module, we have $\oPoi(n) = (\oLie \circ \oCom)(n) \simeq \oAss(n) = \bbK[\frS_n]$, where $\oAss$ denotes the associative operad.

By the operadic cohomology theory \cite[Chap.\ 6]{LV}, given a Poisson algebra $A$ corresponding to an operad morphism $\alpha\colon \oPoi \to \oHom_{V}$, we have the Andre-Quillen type cohomology complex with differential induced by $\alpha$ (see \cref{sss:AP:pc} for a brief review).
This complex has a complicated nature, but by the work of Fresse \cite{Fr}, it has a bicomplex structure $C^{\dbl}_{\tPoi}(A,A)$ whose vertical and horizontal complexes are essentially the same as the Chevalley-Eilenberg (Lie algebra) cohomology complex and the Harrison (commutative algebra) cohomology complex, respectively. See \cref{sss:AP:pcb} for an explanation.

In \cite[\S10.5]{BDHK}, it is argued that a Poisson algebra structure $A$ on a linear space $V$ corresponds bijectively to an operad morphism $X\colon \oLie \to \oPfn{V}$ from the Lie operad to what is called the finite operad. Then, we have the Chevalley-Eilenberg-type cohomology complex $\frg(\oPfn{A})$ with differential induced by $X$. 

So we may ask how the two cohomology complexes $C^{\dbl}_{\tPoi}(A,A)$ and $\frg(\oPfn{A})$ relate.
In \cref{thm:AP:fcc=pcb}, we will show that $\frg(\oPfn{A})$ has a bicomplex structure, and the two bicomplexes coincide up to vertical shift.

We will work over a field $\bbK$ of characteristic $0$, and all the objects (linear spaces, algebras and so on) are defined over $\bbK$ unless otherwise stated. 
For $\bbK$-modules $V$ and $W$ (i.e., linear spaces over $\bbK$), we simplify the notation as $\Hom(V,W) \ceq \Hom_{\bbK}(V,W)$, $V \otimes W \ceq V \otimes_{\bbK}W$ and so on.

\subsection{Poisson cohomology bicomplex}\label{ss:AP:pcb}

\subsubsection{Poisson cohomology}\label{sss:AP:pc}

Let $A=(A,-\cdot-)$ be a commutative algebra with multiplication $-\cdot-\colon A \otimes A \to A$.
A \emph{Poisson bracket} on $A$ is an antisymmetric biderivation $\{-,-\}\colon A \otimes A \to A$ which satisfies the Jacobi relation. A \emph{Poisson algebra} is a commutative algebra equipped with a Poisson bracket. Thus, a Poisson algebra $(A,-\cdot-,\{-,-\})$ has both a structure $(A,-\cdot-)$ of a commutative algebra and a structure $(A,\{-,-\})$ of a Lie algebra.

Let $A=(A,-\cdot-,[-,-])$ be a Poisson algebra.
A \emph{Poisson $A$-module} is a module $M$ over the commutative algebra $(A,-\cdot-)$ equipped with an antisymmetric bilinear map $[-,-]\colon M \otimes A \oplus A \otimes M \to M$ which is a biderivation with respect to the algebra action $A \to \End(M)$, making $M$ a representation of the Lie algebra $(A,[-,-])$. 
We also have the natural notion of \emph{morphisms of Poisson $A$-modules}, and the corresponding category $\cMod^{\tPoi}_A$ of Poisson $A$-modules. 
There exists an associative algebra $U_{\tPoi}(A)$, called the \emph{enveloping algebra of the Poisson algebra $A$}, such that $\cMod^{\tPoi}_A$ is equivalent to the category of left $U_{\tPoi}(A)$-modules. 
See \cite[1.1.2--1.1.4]{Fr} for the detail, and also \cite[12.3.1--12.3.4]{LV} for the general argument applicable to algebras over any operad.

Next, we briefly recall the Poisson cohomology, following \cite[1.2]{Fr}, which gives a Poisson analogue of the Andr\'e-Quillen (co)homology theory. 
See also \cite[12.3]{LV} which can be applied to an arbitrary operad.
Let $A$ be a Poisson algebra, and $M$ be a Poisson $A$-module. 
We say that a linear map $d\colon A \to M$ is a \emph{Poisson derivation} if it is a derivation with respect to both the multiplication and the Poisson bracket. 
We denote by $\Der_{\tPoi}(A,M)$ the module of Poisson derivations $A \to M$. 
Roughly speaking, the Poisson cohomology is defined to be the ``derived functor of $A \mto \Der_{\tPoi}(A,M)$''.
More precisely, there exists a Poisson $A$-module $\Omega_{\tPoi}^1(A)$ such that $\Der_{\tPoi}(A,M)=\Hom_{U_{\tPoi}(A)}(\Omega_{\tPoi}^1(A),M)$ for any Poisson $A$-module $M$.
We then replace $A$ by its quasi-free resolution $R$, 
and define the $p$-th Poisson cohomology of $A$ with coefficients in $M$ is defined to be 
\begin{align}\label{eq:1.1:Hp}
 H^p_{\tPoi}(A,M) \ceq H^p\bigl(\Hom_{U_{\tPoi}(R)}(\Omega^1_{\tPoi}(R),M)^{\bl}\bigr).
\end{align}
Here, a quasi-free resolution of $A$ is a dg Poisson algebra $R=(R_{\bl},\pdd,-\cdot-,[-,-])$ equipped with a surjective quasi-isomorphism $R \to A$ whose underlying graded Poisson algebra $(R_{\bl},-\cdot-,[-,-])$ is the reduced free graded Poisson algebra $\ol{P}(V)$ of some submodule $V \subset R$ stable under the derivation $\pdd\colon R_{\bl} \to R_{\bl-1}$. 
The reduced free Poisson algebra $\ol{\oP}(V)$ is the augmentation ideal of the free Poisson algebra $\oP(V)$, which is given by $\oP(V)=\clS\bigl(\clL(V)\bigr)$, the symmetric algebra of the free Lie algebra $\clL(V)$ of the module $V$. 
Back to \eqref{eq:1.1:Hp}, the module $\Hom_{U_{\tPoi}(R)}(\Omega^1_{\tPoi}(R),M)$ has a grading induced by the homological grading $R_{\bl}$, and is equipped with a differential induced by $\pdd$. Thus, it is a cohomological complex, and we can take the $p$-th cohomology.

\subsubsection{Eulerian decomposition}\label{sss:AP:eul}

As a preliminary of the construction explained in the next \Cref{sss:AP:pcb}, we recall the \emph{Eulerian} (also called \emph{Hodge-type}) \emph{decomposition} of the Hochschild complex of a commutative algebra \cite{GS,L1}. 
We follow the explanation in \cite[\S4.5]{L} and \cite[\S1.3.5]{LV}

For a linear space $V$, we denote by $\clT^c=\clT^c(V)$ the \emph{tensor coalgebra}.
As a module, it is given by $\bigoplus_{n \in \bbN} V^{\otimes n}$.
We denote the naturally $\bbN$-grading as $\clT^c_n \ceq V^{\otimes n}$
We denote its element as 
$(v_1, \dotsc, v_n) \ceq v_1 \otimes \dotsb \otimes v_n \in V^{\otimes n}$.
The comultiplication is given by the \emph{deconcatenation}:
\[
 \Delta(v_1,\dotsc,v_n) \ceq \sum_{i=0}^n (v_1,\dotsc,v_i) \otimes (v_{i-1},\dotsc,v_n)
 \in \clT^c \otimes \clT^c.
\]

The coalgebra $\clT^c$ has a commutative Hopf algebra structure whose multiplication is given by the \emph{shuffle product} $\mu$:
\[ 
 \mu\bigl((v_1 \dotsm v_m) \otimes (v_{m+1},\dotsc,v_{m+n})\bigr) \ceq 
 \sum_{\sigma \in \Sh_{m,n}} (-1)^\sigma v_{\sigma^{-1}(1)} \dotsm v_{\sigma^{-1}(m+n)},
\] 
where $\Sh_{m,n}$ is the subgroup of the $(m+n)$-th symmetric group $\frS_{m+n}$ consisting of \emph{$(m,n)$-shuffles}, i.e., 
\begin{align}\label{eq:AP:sh}
 \Sh_{m,n} \ceq \{\sigma \in \frS_{m+n} \mid  
 \sigma(1)<\dotsb<\sigma(m), \, \sigma(m+1)<\dotsb<\sigma(m+n)\}.
\end{align}
We denote the unit and counit by $u$ and $\ve$, respectively.
See \cite[\S 1.2]{LV} for their precise definitions.

Using the Hopf algebra $(\clT^c,\mu,\Delta,u,\ve)$, we define the \emph{convolution} of two linear endomorphisms $f,g\in \End(\clT^c)$ to be 
\[
 f*g \ceq \mu \circ (f \otimes g) \circ \Delta.
\]
If $f$ and $g$ are algebra endomorphism, then $f*g$ is also an algebra endomorphism.
We also note that if $f$ is an algebra endomorphism with $f(1)=0$ for $1 \in \bbK$, then $f^{*k} \ceq f * \dotsb * f$ ($k$ times) is $0$ when restricted to $V^{\otimes n}$ for $n<k$.

Let us consider the iterative convolution $\id^{* k}$.
It is an algebra endomorphism of $\clT^c$, and we have $\id^{* k} = \mu^k \circ \Delta^k$, where we denoted by $\mu^k\colon V^{\otimes k} \to V$ and $\Delta^k\colon V \to V^{\otimes k}$ the iterative composition of $\mu$'s and $\Delta$'s, respectively. 
Note that $\mu^2=\mu$ and $\Delta^2=\Delta$ in our convention.
Now we define
\[
 e^{(1)} \ceq \sum_{k=1}^\infty (-1)^k \frac{\id^{* k}}{k}.
\]
Due to the equality $(\id-u \ve)(1)=0$ for $1 \in \bbK$, the above remark claims that $e^{(1)}$ is a well-defined algebra endomorphism on $\clT^c$.
We further define 
\[
 e^{(p)} \ceq \frac{1}{p!} (e^{(1)})^{* p}, \quad 
 e^{(p)}_n \ceq \rst{e^{(p)}}{\clT^c_n}.
\]
We have $e^{(0)}_0=1$ and $e^{(p)}_n=0$ for $p>n$.

By a general argument \cite[4.5.3 Proposition]{L}, the elements $e^{(p)}_n \in \End\bigl(\clT^c_n\bigr)$ enjoys the following properties.
\begin{clist}
\item $e^{(1)}_n+\dotsb+e^{(n)}_n=\id$.
\item $e^{(p)}_n e^{(j)}_n = \delta_{i,j} e_n^{(p)}$, i.e, they are orthogonal idempotents.
\end{clist}
These elements are called the \emph{Eulerian idempotents}.

Since each Eulerian idempotent $a = e^{(p)}_n$ ($p=1,\dotsc,n$) acts on the homogeneous component $\clT^c_n=V^{\otimes n}$, the image of $(v_1,\dotsc,v_n) \in V^{\otimes n}$ is of the form $\sum_{\sigma \in \frS_n} a(\sigma) \sigma.(v_1,\dotsc,v_n)$, and we can identify $a$ with $\sum_{\sigma \in \frS_n} a(\sigma) \sigma \in \bbQ[\frS_n]$. 
Under this identification, we have $e^{(p)}_n \in \bbQ[\frS_n]$ ($i=1,\dotsc,n$), for which there is known an explicit formula using the descents of a permutation.
We also have $e^{(p)} = \sum_{n \in \bbN} e_n^{(p)} \in \prod_{n \in \bbN} \bbQ[\frS_n]$, 
regarded as a completion of $\bigoplus_{n \in \bbN}\bbQ[\frS_n]$, 
We refer to \cite[\S 4.5.5]{L} for the detail.

Now we can explain the decomposition of Hochschild cochain complex of a commutative algebra.
Let $A$ be a commutative algebra, and $M$ be a symmetric $A$-bimodule, i.e., an $A$-bimodule satisfying the condition $a.m=m.a$ for any $a \in A$ and $m \in M$. 
We also denote by $\Sigma$ the suspension of a homologically graded module.
Thus $(\Sigma V)_p = V_{p-1}$ for $p \in \bbZ$.
Let 
\begin{align}\label{eq:AP:CA}
 C_{\tAss}^{\bl}(A,M) \ceq \Hom(\clT^c(\Sigma A),M\bigr)
\end{align}
be the Hochschild cochain complex. Each graded component is 
$C_{\tAss}^n(A,M) \simeq \Hom(A^{\otimes n},M)$ ($n \in \bbN$), 
and the coboundary map $b\colon C_{\tAss}^n(A,M) \to C_{\tAss}^{n+1}(A,M)$ is given by 
\begin{align}\label{eq:AP:Ad}
\begin{split}
 (b f)(a_1,\dotsc,a_{n+1}) \ceq a_1 f(a_2,\dotsc,a_{n+1}) 
 &+\sum_{i=1}^n (-1)^i f(a_1,\dotsc, a_i a_{i+1}, \dotsc, a_{n+1}) \\
 &+(-1)^{n+1} f(a_1,\dotsc,a_n)a_{n+1}.
\end{split}
\end{align}
Then the Eulerian idempotents $e_n^{(p)}$ induces the following decomposition of $C^{\bl}(A,M)$, established by \cite{GS,L1}. See also \cite[4.2]{Fr} for an explanation.

\begin{fct}\label{fct:AP:eul}
The Hochschild cochain complex of a commutative algebra $A$ with coefficients in a symmetric $A$-bimodule $M$ has a decomposition into subcomplexes 
\[
 C_{\tAss}^{\bl}(A,M) = \bigoplus_{s \in \bbN}C_{(p)}^{\bl}(A,M), \quad 
 C_{(p)}^{\bl}(A,M) \ceq \Hom(e^{(p)} \clT^c(\Sigma A),M).
\]
Furthermore, we have
\begin{align}\label{eq:AP:Sp}
  C_{(p)}^{\bl}(A,M) \simeq \Hom\bigl(\clS_p\bigl(\clL^c(\Sigma A)\bigr),M\bigr)
\end{align}
for each $p \in \bbN$, where $\clL^c(\Sigma A)$ is the cofree Lie coalgebra of the graded module $\Sigma A$, and $\clS = \bigoplus_{p \in \bbN} \clS_p$ denotes the free commutative algebra functor, i.e., the symmetric tensor product. 
In particular, the $p=1$ part is isomorphic to the Harrison cochain complex:
\begin{align}\label{eq:AP:CC}
 C_{(1)}^{\bl}(A,M) \simeq \Hom(\clL^c(\Sigma A),M) = C^{\bl}_{\tCom}(A,M).
\end{align}
\end{fct}

For later use, let us describe the Harrison (co)chain complex explicitly. Let $n \in \bbZ_{>0}$, and for $r=1,\dotsc,n$, we set $s_{r,n-r} \ceq \sum_{\sigma \in \Sh(r,n-r)} \sgn(\sigma) \sigma \in \bbZ[\frS_n]$, which acts on $A^{\otimes n}$ by permuting tensor factors. Then we define 
\[
 \ch_n(A) \ceq A^{\otimes n}/\langle s_{r,n-r}(a_1 \otimes \dotsb \otimes a_n) \mid 
 r=1,\dotsc,n, \, a_i \in A\rangle_{\text{$A$-mod}}.
\]
We will denote by $[a_1 \otimes \dotsb \otimes a_n] \in \ch_n(A)$ the element associated to $a_1 \otimes \dotsb \otimes a_n \in A^{\otimes n}$. Now, for an $A$-module $M$, we set 
\begin{align}\label{eq:AP:harh}
 \ch_n(A;M) \ceq \ch_n(A) \otimes M. 
\end{align}
The graded module $\bigoplus_{n>0} \ch_n(A;M)$ is equipped with the boundary map $\pdd\colon \ch_n(A;M) \to \ch_{n-1}(A;M)$ defined by 
\begin{align}\label{eq:AP:harb}
\begin{split}
 \pdd_n([a_1 \otimes \dotsb \otimes a_n] \otimes m) \ceq 
 [a_1 \otimes \dotsb \otimes a_{n-1}] \otimes a_n.m 
&+\sum_{i=1}^{n-1} [a_1 \otimes \dotsb \otimes a_i a_{i+1} \otimes \dotsb \otimes a_n] \otimes m \\
&+(-1)^n [a_2 \otimes \dotsb \otimes a_n] \otimes a_1.m.
\end{split}
\end{align}
The obtained homology complex $(\ch_{\bl}(A;M),\pdd)$ is the Harrison chain complex. 
On the other hand, the Harrison cochain complex $C^{\bl}_{\tCom}(A,M)$ is given by 
\begin{align}\label{eq:AP:harcoh}
  C^n_{\tCom}(A,M) \ceq \Hom(\ch_n(A),M)
\end{align}
with the coboundary map $d\colon C^n_{\tCom}(A,M) \to C^{n+1}_{\tCom}(A,M)$ being the same one with the Hochschild coboundary map \eqref{eq:AP:Ad}.

\begin{rmk}\label{rmk:AP:eul=Har}
Some comments on references are in order.
The Harrison complex is originally defined \cite{Ha} using monotone permutations, which are also used in \cite{BDHK}. On the other hand, the cofree Lie coalgebra is constructed as the indecomposables of the tensor coalgebra with respect to the shuffle product (see \cite{SS} for example). These two objects are shown to be equivalent in \cite{GS}. The above description \eqref{eq:AP:harh} and \eqref{eq:AP:harb} is indeed the one for the cofree Lie coalgebra. Later, the relationship between Lie and commutative algebraic objects is enhanced to the operadic Koszul duality theory \cite{GK}. 
\end{rmk}

\subsubsection{Poisson cohomology bicomplex}\label{sss:AP:pcb}

The Poisson cohomology complex is introduced in \cite{Fr} to calculate the Poisson cohomology \eqref{eq:1.1:Hp}. Here we give a brief recollection, following \cite[\S1.3]{Fr} and \cite[\S 2]{N}. As a doubly-graded module, it is nothing but the Eulerian-decomposed Hochschild complex reviewed in the previous \Cref{sss:AP:eul}. 

Let $A=(A,-\cdot-,[-,-])$ be a Poisson algebra, and $M$ be a Poisson $A$-module. 
Then $M$ can be regarded as a symmetric module over the underlying commutative algebra $A_{\tCom} \ceq (A,-\cdot-)$, and we have the Hochschild cochain complex $C_{\tAss}^{\bl}(A_{\tCom},M)$.
Recall the Eulerian decomposition of the Hochschild cochain complex:
\[
 C^{\bl}_{\tAss}(A_{\tCom},M) = \bigoplus_{p \in \bbN} C^{\bl}_{(p)}(A_{\tCom},M), \quad 
 C^{\bl}_{(p)}(A,M) \simeq \Hom(\clS_p\bigl(\clL^c(\Sigma A)\bigr),M)^{\bl}.
\] 
For $p,q \in \bbN$, we set $C^{p,q}_{\tPoi}(A,M) \ceq 0$ for $p=0$ and 
\[
 C^{p,q}_{\tPoi}(A,M) \ceq C^{p+q}_{(p)}(A,M) \quad  (p>0).
\]
We denote the restriction of the Hochschild coboundary map \eqref{eq:AP:Ad} by the same symbols as
\[
 d\colon C^{p+q}_{(p)}(A,M) \lto C^{p+q+1}_{(p)}(A,M).
\]
By \cref{fct:AP:eul}, it gives the vertical differential of the double complex.

In order to introduce the horizontal differential, we need to recall the Chevalley-Eilenberg cochain complex of a Lie algebra. Let $L=(L,[-,-])$ be a Lie algebra, and $M$ be an $L$-module with $[-,-]\colon L \otimes M \to M$ its Lie module structure. Then the Chevalley-Eilenberg cochain complex $(C_{\tLie}^{\bl}(L,M),\delta_{\tCE})$ is given by 
\begin{align}\label{eq:AP:CL}
 C_{\tLie}^n(L,M) \ceq \Hom(\clS_n(\Sigma L),M), 
\end{align}
and for $f \in C_{\tLie}^n(L,M)$ and $u_0 \dotsm u_n \in \clS_{n+1}(\Sigma L)$, 
\begin{align}\label{eq:AP:Ld}
 \delta_{\tCE}(f)(u_0 \dotsm u_n) = 
 \sum_{i=0}^n (-1)^i [u_i,f(u_0 \dotsm \wh{u_i} \dotsm u_n)] +
 \sum_{0 \le i<j \le n} (-1)^{i+j} f([u_i,u_j] u_0 \dotsm \wh{u_i} \dotsm \wh{u_j} \dotsm u_n).
\end{align}

By \cite[1.3.5]{Fr}, the Poisson bracket on $A$ induced a $(-1)$-shifted Lie bracket $[-,-]$ on $\clL^c(\Sigma A)$, i.e., $L \ceq \Sigma^{-1} \clL^c(\Sigma A)$ is a graded Lie algebra.
Also, by \cite[1.3.7]{Fr}, the Poisson $A$-module structure on $M$ induces a module structure over the $(-1)$-shifted Lie algebra $(\clL^c(\Sigma A),[-,-])$. 
Thus we have the Chevalley-Eilenberg cochain complex $(C_{\tLie}^{\bl}(L,M),\delta_{\tCE})$.
Unfolding the definitions, we have a differential 
\[
 \delta_{\tCE}\colon C^{p+q}_{(p)} \lto C^{p+q+1}_{(p+1)}.
\]
Since the Hochschild coboundary map $d$ induces a dg Lie algebra structure on $(\clL^c(\Sigma A),[-,-])$ by \cite[1.3.5]{Fr}, we finally obtain:

\begin{dfn}[{\cite[1.3.13]{Fr}}]\label{dfn:AP:pcb}
For a Poisson algebra $(A,-\cdot-,[-,-])$ and a Poisson $A$-module $M$, we have a bicomplex
\[
 \bigl(C_{\tuPoi}^{\dbl}(A,M),d,\delta_{\tCE}\bigr),
\]
where
\[
 C_{\tPoi}^{p,q}(A,M) \ceq \begin{cases} 
  C^{p+q}_{(p)}(A,M) \simeq \Hom\bigl(\clS_p\bigl(\clL^c(\Sigma A)\bigr),M\bigr)^{p+q} 
  & (p \in \bbZ_{>0}, q \in \bbN) \\ 0 & (\text{otherwise}) \end{cases},
\]
the vertical differential $d\colon C_{\tPoi}^{p,q} = C^{p+q}_{(p)} \to C^{p+q+1}_{(p)} = C_{\tPoi}^{p,q+1}$ is induced by the Hochschild coboundary map \eqref{eq:AP:Ad}, and the horizontal differential $\delta\colon C_{\tPoi}^{p,q} = C^{p+q}_{(p)} \to C^{p+q+1}_{(p+1)} = C_{\tPoi}^{p+1,q}$ is induced by the Chevalley-Eilenberg coboundary map \eqref{eq:AP:Ld}.
We call this double complex the \emph{Poisson cohomology bicomplex}.
\end{dfn}

\begin{fct}[{\cite[1.3.14]{Fr}}]\label{fct:AP:pcb->pc}
The cohomology of the total complex of the Poisson cohomology bicomplex in \cref{dfn:AP:pcb} is isomorphic to the Poisson cohomology \eqref{eq:1.1:Hp}.
\end{fct}

\cref{fig:AP:pcb} shows the diagrams concerning the Poisson cohomology bicomplex. We only display the non-zero parts. The top left diagram shows the configuration of $C^{p+q}_{(p)}$'s, and the top right one shows $C^{p,q}$'s. The bottom diagram is a citation from \cite[p.268]{N}, showing the content of each module with the Harrison chain complex $\ch_{\bl} \simeq \clL^c(\Sigma A)_{\bl}$ (see \eqref{eq:AP:harh}, \eqref{eq:AP:harb} and \cref{rmk:AP:eul=Har}).
\begin{figure}[htbp]
\begin{tikzcd}[row sep=scriptsize, column sep=small]
 \vdots & & & & \vdots \\
 C_{(1)}^3 \ar[r,"\delta_{\tCE}"] \ar[u,"d"] & \cdots & & & 
 C^{1,2} \ar[r] \ar[u] & \cdots \\
 C_{(1)}^2 \ar[r,"\delta_{\tCE}"] \ar[u,"d"] & 
 C_{(2)}^3 \ar[r,"\delta_{\tCE}"] \ar[u,"d"] & \cdots & &
 C^{1,1} \ar[r] \ar[u] & C^{2,1} \ar[r] \ar[u] & \cdots \\
 C_{(1)}^1 \ar[r,"\delta_{\tCE}"] \ar[u,"d"] & C_{(2)}^2 \ar[r,"\delta_{\tCE}"] \ar[u,"d"] & 
 C_{(3)}^3 \ar[r,"\delta_{\tCE}"] \ar[u,"d"] & \cdots & 
 C^{1,0} \ar[r] \ar[u] & C^{2,0} \ar[r] \ar[u] & C^{3,0} \ar[r] \ar[u] & \cdots
\end{tikzcd}

\begin{tikzcd}[row sep=scriptsize, column sep=scriptsize]
 \vdots \\
 \Hom(\ch_3,M) \ar[r] \ar[u] & \cdots \\
 \Hom(\ch_2,M) \ar[r] \ar[u] & \Hom(\ch_2 \otimes \ch_1,M) \ar[r] \ar[u] & \cdots \\
 \Hom(\ch_1,M) \ar[r] \ar[u] & \Hom(\wedge^2 \ch_1,M) \ar[r] \ar[u] & 
 \Hom(\wedge^3 \ch_1,M) \ar[r] \ar[u] & \cdots
\end{tikzcd}
\caption{Poisson cohomology bicomplex}
\label{fig:AP:pcb}
\end{figure}

As a corollary, we have:

\begin{fct}\label{fct:AP:pcb-cl}
Let $A=(A,-\cdot-,\{-,-\})$ be a Poisson algebra, and $M$ be a Poisson $A$-module. Denote by $A_{\tCom} \ceq (A,-\cdot-)$ and $A_{\tLie} \ceq (A,\{-,\cdot-\})$ the underlying commutative and Lie algebra, respectively. Then the Harrison cochain complex $C_{\tCom}^{\bl}(A_{\tCom},M)$ and the Chevalley-Eilenberg cochain complex $C_{\tLie}^{\bl}(A_{\tLie},M)$ are embedded in the Poisson cohomology bicomplex $C_{\tuPoi}^{\dbl}(A,M)$ as 
\[
 C_{\tCom}^{\bl}(A_{\tCom},M) = C_{\tPoi}^{\bl,1}(A,M), \quad
 C_{\tLie}^{\bl}(A_{\tLie},M) = C_{\tPoi}^{1,\bl}(A,M).
\]
\end{fct}

\subsection{Finite operad and Poisson algebra structure}\label{ss:AP:Pfn}

In this \Cref{ss:AP:Pfn}, we give a brief review of the \emph{finite operad}, which was introduced in\cite[10.5]{BDHK} as a byproduct of the study of the classical operad $\oP^{\cl}$.

\subsubsection{$n$-graphs}\label{sss:AP:graph}

We recollect the terminology on graphs and quivers from \cref{ss:2P:graph}.
See also \cite[8.2]{BDHK}, \cite[4.1]{BDHK2} and \cite[4.2]{BDKV21}.


Recall from \cref{dfn:2P:oQ} the set $\oQ(n)$ of $n$-quivers without loops and the subset $\oQ_{\ac}(n) \subset \oQ(n)$ of acyclic $n$-quivers. Let $\bbK \oQ(n)$ be the linear space with basis $\oQ(n)$. 
An element of $\bbK \oQ(n)$ is called a \emph{cycle relation} (\cite[4.1]{BDHK2}) if it has either of the following forms.
\begin{clist}
\item
An element $\Gamma \in \oQ(n) \sm \oQ_{\ac}(n)$, i.e., an $n$-quiver whose underlying graph $\ul{\Gamma}$ contains a cycle.
\item
A linear combination $\sum_{e \in C} \Gamma \sm e$, where $\Gamma \in \oQ(n)$ and $C \subset E(\Gamma)$ is a directed cycle. Here $\Gamma \sm e$ denotes the ($n-1$)-quiver obtained from $\Gamma$ by removing the edge $e$ contained in the directed cycle $C$ (c.f.\ \cref{lem:2P:dgraph} \ref{i:lem:2P:dgraph:3}).
\end{clist}
We denote by $R(n) \subset \bbK \oQ(n)$ the linear span of all the cycle relations.
The $\frS_n$-action on the set $\oQ(n)$ extends to a linear action on $\bbK \oQ(n)$, which preserves the subspace $R(n)$. 
Thus we have a linear $\frS_n$-action on the quotient space $\bbK \oQ(n)/R(n)$.

Let us consider a partition 
\[
 \{1,\dotsc,n\} = \{i^1_1,\dotsc,i^1_{m_1}\} \sqcup \{i^2_1,\dotsc,i^2_{m_2}\} \sqcup 
                  \dotsb \sqcup \{i^p_1,\dotsc,i^p_{m_p}\}
\]
with $p \in \bbZ_{\ge 1}$, $m_l \ge 1$ ($l=1,2,\dotsc,p$) and $m_1+m_2+\dotsb+m_p=n$ such that 
\[
 i^1_1=1 < i^2_1 < \dotsb < i^p_1, \quad 
 i^l_1 = \min\{i^l_1,\dotsc,i^l_{m_l}\} \quad (l=1,2,\dotsc,p).
\]
For such a partition $I$, we have an $n$-quiver $\Lambda_I$ of the following form.
\begin{align}\label{eq:AP:GammaI}
\begin{tikzpicture}
 \draw (-1.0,0.2) node {$\Lambda_I \ceq$};
 \draw ( 0.0,0.0) node {$\underset{i^1_1}{\bullet}$}; 
 \draw ( 0.6,0.0) node {$\underset{i^1_2}{\bullet}$}; 
 \draw ( 1.4,0.2) node {$\cdots$}; 
 \draw ( 2.2,0.0) node {$\underset{i^1_{m_1}}{\bullet}$}; 
 \draw[->] (0.1,0.2)--(0.5,0.2);
 \draw[->] (0.7,0.2)--(1.1,0.2);
 \draw[->] (1.7,0.2)--(2.1,0.2);
 \draw ( 3.0,0.0) node {$\underset{i^2_1}{\bullet}$}; 
 \draw ( 3.6,0.0) node {$\underset{i^2_2}{\bullet}$}; 
 \draw ( 4.4,0.2) node {$\cdots$}; 
 \draw ( 5.2,0.0) node {$\underset{i^2_{m_2}}{\bullet}$}; 
 \draw[->] (3.1,0.2)--(3.5,0.2);
 \draw[->] (3.7,0.2)--(4.1,0.2);
 \draw[->] (4.7,0.2)--(5.1,0.2);
 \draw (6.1,0.2) node {$\cdots$};
 \draw ( 7.0,0.0) node {$\underset{i^p_1}{\bullet}$}; 
 \draw ( 7.6,0.0) node {$\underset{i^p_2}{\bullet}$}; 
 \draw ( 8.4,0.2) node {$\cdots$}; 
 \draw ( 9.2,0.0) node {$\underset{i^p_{m_p}}{\bullet}$}; 
 \draw[->] (7.1,0.2)--(7.5,0.2);
 \draw[->] (7.7,0.2)--(8.1,0.2);
 \draw[->] (8.7,0.2)--(9.1,0.2);
\end{tikzpicture}
\end{align}
We call $\Lambda_I$ a \emph{disjoint union of lines}, and denote by $\clL(n) \subset \oQ(n)$ the set of $n$-quivers that are disjoint unions of lines. Clearly we have $\clL(n) \subset \oQ_{\ac}(n)$.
Now let us cite:

\begin{fct}[{\cite[Lemma 4.1]{BDHK2}}]\label{fct:AP:clL}
The set $\clL(n)$ is a basis of the quotient space $\bbK \oQ(n)/R(n)$.
\end{fct}

We close this subsection by recalling the cocomposition map of $n$-quivers (\cref{dfn:2P:Delta}). 
See also \cite[9.1]{BDHK}.
For an $n$-tuple of positive integers $M=(m_1,m_2,\dotsc,m_n) \in \bbZ_{\ge 1}^n$, we set 
\begin{align}\label{eq:AP:coc}
 M_0 \ceq 0, \quad M_i \ceq m_1+m_2+\dotsb+m_i \quad (i=1,2,\dotsc,n).
\end{align}
Note that $M_n=m_1+\dotsb+m_n$. Then the map 
\[
 \Delta^M=(\Delta^M_0,\Delta^M_1,\dotsc,\Delta^M_n)\colon 
 \oQ(M_n) \lto \oQ(n) \times \oQ(m_1) \times \dotsb \times \oQ(m_n)
\]
is defined by the following description. Let $\Gamma \in \oQ(M_n)$.
\begin{itemize}
\item 
For $i=1,\dotsc,n$, $\Delta^M_i(\Gamma) \in \oQ(m_i)$ is the directed subgraph of $\Gamma$ associated to the vertices $M_{i-1}+1,\dotsc,M_i$.
\item
$\Delta^M_0(\Gamma) \in \oQ(n)$ is the directed graph obtained by collapsing the vertices $M_{i-1}+1,\dotsc,M_i$ into the single vertex $i$ for each $i=1,\dotsc,n$.
\end{itemize}

\subsubsection{Finite operad}\label{sss:AP:Pfn}

Now we cite from \cite[10.5]{BDHK} the definition of the finite operad $\oPfn{V}$.

Let $V$ be a linear space. For $n \in \bbN$, we define
\begin{align}\label{eq:AP:Pfn}
 \oPfn{V}(n) \ceq \Hom\bigl((\bbK\oQ(n)/R(n))\otimes V^{\otimes n}, V\bigr).
\end{align}
We denote its element as a map 
\[ 
 f\colon \oQ(n) \times V^{\otimes n} \lto V, \quad 
 (\Gamma,v_1 \otimes \dotsb \otimes v_n) \lmto f^\Gamma(v_1 \otimes \dotsb \otimes v_n)
\] 
which is linear in the second factor $V^{\otimes n}$, and we extend it by linearity for $\Phi = \sum_{\Gamma} c_\Gamma \Gamma \in \bbK \oQ(n)$, i.e., $f^\Phi \ceq \sum_{\Gamma}c_\Gamma f^{\Gamma}\colon V^{\otimes n} \to V$. Then $f^{\Phi}=0$ for every cycle relation $\Phi \in R(n)$.

Recall that in \Cref{sss:AP:graph} we defined the $\frS_n$-actions on the set $\oQ(n)$, the space $\bbK \oQ(n)$ and the quotient space $\bbK \oQ(n)/R(n)$. 
Then the space $\oPfn{V}(n)$ is a right $\frS_n$-module by 
\[
 (f^\sigma)^\Gamma(v_1 \otimes \dotsb \otimes v_n) \ceq 
 f^{\sigma(\Gamma)}(v_{\sigma^{-1}(1)} \otimes \dotsb \otimes v_{\sigma^{-1}(n)}).
\]

Finally, we define the composition maps. 
For each $M=(m_1,\dotsc,m_n) \in \bbN^n$, we should define $\gamma\colon \oPfn(n) \otimes \oPfn(m_1) \otimes \dotsb \otimes \oPfn(m_n) \to \oPfn(M_n)$ with $M_n \ceq m_1+\dotsb+m_n$. 
For $f \in \oPfn(n)$, $g_1 \in \oPfn(m_1)$, $\dotsc$, $g_n \in \oPfn(m_n)$ and $\Gamma \in \oQ(m_1+\dotsb+m_n)$, we define $f(g_1,\dotsc,g_n) = \gamma(f;g_1,\dotsc,g_n)$ by 
\begin{align}\label{eq:2P:Pfn:gamma}
 f(g_1,\dotsc,g_n)^\Gamma \ceq f^{\Delta_0^M(\Gamma)}\bigl(g_1^{\Delta_1^M(\Gamma)} \otimes \dotsb \otimes g_n^{\Delta_n^M(\Gamma)}\bigr).
\end{align}
Here we used the cocomposition map $\Delta^M_i$ in \Cref{sss:AP:graph}.

\begin{fct}[{\cite[10.5]{BDHK}}]\label{fct:2P:Pfn}
The $\frS$-module $\oPfn{V} = \bigl(\oPfn{V}(n)\bigr)_{n \in \bbN}$ has a structure of an operad whose composition map $\gamma\colon \oPfn{V} \circ \oPfn{V} \to \oPfn{V}$ given by \eqref{eq:2P:Pfn:gamma}. We call it \emph{the finite operad}.
\end{fct}

For later use, let us describe the linear space $\oPfn{V}(2)$ explicitly. By \cref{fct:AP:clL}, the linear space $\bbK \oQ(2)/R(2)$ has a basis $\Lambda_{\{1\} \sqcup \{2\}}, \Lambda_{\{1\,2\}}$ with
\[
 \begin{tikzpicture}
 \draw (-1.1,0.11) node {$\Lambda_{\{1\} \sqcup \{2\}} \ceq$};
 \draw ( 0.0,0.00) node {$\underset{1}{\bullet}$}; 
 \draw ( 0.6,0.00) node {$\underset{2}{\bullet}$}; 
 \draw ( 2.1,0.11) node {$\Lambda_{\{1,2\}} \ceq$};
 \draw ( 3.0,0.00) node {$\underset{1}{\bullet}$}; 
 \draw ( 3.6,0.00) node {$\underset{2}{\bullet}$}; 
 \draw[->] (3.1,0.15)--(3.5,0.15);
 \end{tikzpicture}
\]
For these $2$-quivers, we denote 
\begin{align}\label{eq:AP:fG2}
 f^{\{1\} \sqcup \{2\}} \ceq f^{\Lambda_{\{1\} \sqcup \{2\}}}, \quad 
 f^{\{1,2\}} \ceq f^{\Lambda_{\{1,2\}}}.
\end{align} 
Then the linear space $\oPfn{V}(2)$ is spanned by the elements $f^{\{1\} \sqcup \{2\}}, f^{\{1 2\}}$ with $f \in \Hom(V^{\otimes 2},V)$.

\subsubsection{Finite operad and Poisson algebra structure}

What we want to explain next is the fact established in \cite[Theorem 10.16]{BDHK} that $\oPfn{V}$ is related to Poisson algebra structures on $V$. For that, recall the operadic deformation and cohomology theory of algebraic structures from \cite[\S1.3]{NY}.

%
Let $\oQ$ be an arbitrary operad, and consider the cohomology complex of Lie algebra structures on $\oQ$:
\begin{align}\label{eq:AP:frg}
 \frg(\oLie,\oQ)^{\bl} = \bigl(\Hom_{\frS}(\oLie^!,\oQ)^{\bl},[-,-]\bigr), \quad 
 [f,g] \ceq f \square g - (-1)^{\abs{f} \abs{g}}g \square f.
\end{align}
The graded component $\frg(\oLie,\oQ)^n$ is given by 
\begin{align}\label{eq:AP:frgn}
 \frg(\oLie,\oQ)^n \simeq \bigl(\pc \oQ(n+1)\bigr)^{\frS_{n+1}} \ceq 
 \{f \in \oQ(n+1) \mid \forall \sigma \in \frS_{n+1}, \ f^{\sigma} = \sgn(\sigma) f\}
\end{align}
as a linear space. Here $f^\sigma$ denotes the right action of $\sigma \in \frS_{n+1}$ on $f \in \oQ(n+1)$. 
The pre-Lie product $\square$ in $\frg(\oLie,\oQ)$. For $f \in \frg(\oLie,\oQ)^n$ and $g \in \frg(\oLie,\oQ)^m$, we have $f \square g \in \frg(\oLie,\oQ)^{n+m}$ with 
\begin{align}\label{eq:AP:square}
 f \square g \ceq \sum_{\sigma \in \Sh(m+1,n)} (f \circ_1 g)^{\sigma^{-1}},
\end{align}
where $\Sh(m+1,n) \subset \frS_{n+m+1}$ denotes the subset of $(m+1,n)$-shuffles (see \eqref{eq:AP:sh}),
and $\circ_1$ denotes the infinitesimal composition \eqref{eq:1:circ_i}.
The graded Lie algebra $\frg(\oLie,\oQ)$ is essentially the same as the \emph{universal Lie superalgebra associated to $\oQ$} in \cite[3.2]{BDHK}.


Given an element $X \in \MC\bigl(\frg(\oLie,\oQ)\bigr)$, i.e., a solution of the Maurer-Cartan equation with trivial differential, we have a differential $d_X \ceq [X,-]$ on $\frg(\oLie,\oQ)$. Let us summarize the argument as:

\begin{fct}[{\cite[\S3]{BDHK}}]\label{fct:AP:gX}
Let $\oQ$ be an operad, and $\frg(\oLie,\oQ)^{\bl}$ be the graded Lie algebra \eqref{eq:AP:frg}, \eqref{eq:AP:frgn}, \eqref{eq:AP:square}.
\begin{enumerate}
\item 
An operad morphism $\varphi\colon \oLie \to \oQ$ is called a \emph{Lie algebra structure in $\oQ$}, which is in one-to-one correspondence with an element $X \in \MC\bigl(\frg(\oLie,\oQ)\bigr)$, i.e., $X \in \frg(\oLie,\oQ)^1$ satisfying $[X,X]=0$.
\item
For $X \in \MC\bigl(\frg(\oLie,\oQ)\bigr)$, we have the dg Lie algebra 
\begin{align}\label{eq:gLQX}
 \frg(\oLie,\oQ)^X \ceq \bigl(\frg(\oLie,\oQ), d_X\bigr), \quad 
 d_X \ceq [X,-]\colon \frg(\oLie,\oQ)^{\bl} \to \frg(\oLie,\oQ)^{\bl+1}.
\end{align}
\end{enumerate}
\end{fct}


Now we resume the discussion. By \cref{fct:2P:Pfn}, we have the finite operad $\oPfn{V}$, to which we can apply \cref{fct:AP:gX}. The result is summarized as:

\begin{fct}[{\cite[Theorem 10.16]{BDHK}}]
For a linear space $V$, we denote the graded Lie algebra in \cref{fct:AP:gX} with $\oQ=\oPfn{V}$ as $\frg \ceq \frg(\oLie,\oPfn{V})$. 
By \eqref{eq:AP:frgn} and \eqref{eq:AP:Pfn}, the underlying graded linear space is 
\begin{align}\label{eq:AP:gfnVn}
 \frg^{n-1} = \bigl(\pc \oPfn{V}(n)\bigr)^{\frS_n} = 
 \left\{f \in \Hom\bigl((\bbK \oQ(n)/R(n)) \otimes V^{\otimes n},V\bigr) \mid 
        \forall \sigma \in \frS_n, \ f^\sigma = \sgn(\sigma) f\right\}.
\end{align}
Then there is a bijection between the set 
\[
 \MC(\frg) = \{X \in \frg^1 \mid X \square X = 0\} = 
 \{X \in \bigl(\pc \oPfn{V}(2)\bigr)^{\frS_2} \mid X \square X=0\}
\]
and the set of Poisson algebra structures $(-\cdot-,\{-,-\})$ on $V$. 
Under the notation \eqref{eq:AP:fG2}, the bijection is given by 
\begin{align}\label{eq:AP:X=P}
 X^{\{1\} \sqcup \{2\}}(v \otimes w) = \{v,w\}, \quad 
 X^{\{1,2\}}(v \otimes w) = v \cdot w.
\end{align}
\end{fct}

Hereafter we identify $X \in \MC\bigl(\frg(\oLie,\oPfn{V})\bigr)$ and the Poisson algebra structure on $V$ corresponding to $X$ under \eqref{eq:AP:X=P}, and denote the Poisson algebra as $A=(V,X)=(V,-\cdot-,\{-,-\})$. Given such a Poisson algebra $A$, we have the cohomology complex \eqref{eq:gLQX}, which is denoted as 
\begin{align}\label{eq:AP:fcc}
 \frg^{\tfn}_A \ceq 
 \bigl(\frg(\oLie,\oPfn{V}), d_X\bigr), \quad d_X(f) \ceq [X,f].
\end{align}
We call it \emph{the finite cohomology complex} of the Poisson algebra $A=(V,X)$.

\subsection{Identification}\label{ss:AP:fcc=pcb}

Let us continue to use the symbols in the previous \Cref{ss:AP:Pfn}.
Given a Poisson algebra $A=(V,X)$, we now have two cohomology complexes.
\begin{itemize}
\item 
The cohomology complex $\frg^{\tfn}_A$ in \eqref{eq:AP:fcc}
\item
The Poisson cohomology bicomplex $C_{\tuPoi}^{\dbl}(A,A)$ in \cref{dfn:AP:pcb}. 
\end{itemize}
A relation between these two complexes is given by \cite{BDKV21}.

\begin{fct}[{c.f.\ \cite[Theorem 4.1]{BDKV21}}]\label{fct:AP:dHar}
Let $A=(V,X)=(V,-\cdot-,\{-,-\})$ be a Poisson algebra with underlying linear space $V$. Then there is a surjective morphism of cochain complexes from the cohomology complex $\frg^{\tfn}_A$ to the Harrison cochain complex of the commutative algebra $A_{\tCom}=(V,-\cdot-)$, 
\begin{align}\label{eq:AP:gfnX}
 \frg^{\tfn}_A \lsrj C_{\tCom}(A_{\tCom},A_{\tCom}), 
\end{align}
mapping $Y \in (\frg^{\tfn}_A)^{n-1} = \bigl(\pc \oPfn{V}(n)\bigr)^{\frS_n}$ to $Y^{\Lambda_{\{1,\dotsc,n\}}}$, where $\Lambda_{\{1,\dotsc,n\}}$ denotes the $n$-quiver $\Lambda_I$ in \eqref{eq:AP:GammaI} with partition $I=\{1,\dotsc,n\}$, i.e., 
\begin{center}
\begin{tikzpicture}
 \draw (-1.1,0.11) node {$\Lambda_{\{1,\dotsc,n\}} =$};
 \draw ( 0.0,0.00) node {$\underset{1}{\bullet}$}; 
 \draw ( 0.6,0.00) node {$\underset{2}{\bullet}$}; 
 \draw ( 1.4,0.15) node {$\cdots$}; 
 \draw ( 2.2,0.00) node {$\underset{n}{\bullet}$}; 
 \draw[->] (0.1,0.15)--(0.5,0.15);
 \draw[->] (0.7,0.15)--(1.1,0.15);
 \draw[->] (1.7,0.15)--(2.1,0.15);
\end{tikzpicture}
\end{center}
\end{fct}

\begin{rmk*}
Strictly speaking, \cite{BDKV21} studies the cohomology complex associated to the classical operad $\oP^{\cl}_{V,\pdd}$ and the differential Harrison complex for a given differential commutative algebra $(V,\pdd, -\cdot-)$. The above \cref{fct:AP:dHar} follows from loc.\ cit.\ by putting $\pdd=0$.
\end{rmk*}

Now recall \cref{fct:AP:pcb-cl} which claims that the Harrison cochain complex is embedded in the Poisson cohomology bicomplex $C_{\tuPoi}^{\dbl}(A,A)$ as 
\[
 C_{\tCom}^{\bl}(A_{\tCom},A_{\tCom}) = C_{\tPoi}^{\bl,1}(A,A).
\]
So it is natural to ask whether the surjection in \cref{fct:AP:dHar} extends to a map $\frg^{\tfn}_A \to C_{\tPoi}(A,A)$. We give an affirmative answer.

\begin{thm}\label{thm:AP:fcc=pcb}
For a Poisson algebra $A=(V,-\cdot-,[-,\cdot,-])=(V,X)$ with underlying linear space $V$, the finite cohomology complex $\frg^{\tfn}_A = \frg(\oLie,\oPfn{V}V)^X$ has a bicomplex structure $(\bigoplus_{p,q \in \bbZ} (\frg^{\tfn}_A)^{p,q},d,\delta_{\tCE})$ which is isomorphic to the Poisson cohomology bicomplex up to shift:
\[
 (\frg^{\tfn}_A)^{\dbl} \simeq  C^{\dbl+1}_{\tPoi}(A,A).
\]
The total complex of this bicomplex is equal to the finite cohomology complex $\frg^{\tfn}_A$.
Moreover, restricting this isomorphism to the subspace $(\frg^{\tfn}_A)^{1,\bl} \subset (\frg^{\tfn}_A)^{\dbl}$, we recover the surjection \eqref{eq:AP:gfnX}.
\end{thm}

The rest part of this \Cref{ss:AP:fcc=pcb} is devoted to the proof of \cref{thm:AP:fcc=pcb}.

We begin with another citation from \cite[(3.36)]{BDKV21}. Recall that for an $n$-quiver $\Gamma$, we denote by $E(\Gamma)$ the edge set of $\Gamma$. Then we can define a natural $\bbN$-grading
\[
 \oPfn{V}V(n) = \bigoplus_{r \in \bbN} \gr^r \oPfn{V}(n) 
\]
on the finite operad $\oPfn{V}$ by 
\begin{align}\label{eq:2P:groP}
 \gr^r \oPfn{V}(n) \ceq \{f \in \oPfn{V}(n) \mid f^\Gamma=0 
 \text{ for any $n$-quiver $\Gamma$ with $\# E(\Gamma) \ne r$}\}.
\end{align}
Note that we have $\oPfn{V}(n) = \bigoplus_{r=0}^{n-1} \gr^r \oPfn{V}(n)$ by the cycle conditions. Now $\oPfn{V}$ is a graded operad, i.e., the $\frS$-module structure and the composition map $\gamma$ respect the grading.

Next, using the notation in \Cref{ss:AP:Pfn}, let us choose and fix a Poisson algebra $A=(V,-\cdot-,\{-,-\})=(V,X)$ with underlying linear space $V$ and $X \in \MC(\frg^{\tfn}_V)$. For simplicity, let us denote by
\[
 \frg \ceq \frg^{\tfn,X}_A
\]
the cohomology complex \eqref{eq:AP:gfnX}. By \eqref{eq:AP:gfnVn}, we have
\[
 \frg^{n-1} = 
 \left\{f \in \Hom\bigl((\bbK \oQ(n)/R(n)) \otimes V^{\otimes n},V\bigr) \mid 
        \forall \sigma \in \frS_n, \ f^\sigma = \sgn(\sigma) f\right\},
\]
and by \cref{fct:AP:clL}, the space $\bbK \oQ(n)/R(n)$ has a basis $\clL(n)$ consisting of disconnected unions of lines $\Lambda_I$ where $I$ runs over partitions of $\{1,\dotsc,n\}$. Then the grading \eqref{eq:2P:groP} on $\oPfn{V}$ induces the following double grading on $\frg$.

\begin{lem}
For a partition $I = \{i^1_1,\dotsc,i^1_{m_1}\} \sqcup \{i^2_1,\dotsc,i^2_{m_2}\} \sqcup \dotsb \sqcup \{i^p_1,\dotsc,i^p_{m_p}\}$ of $\{1,\dotsc,n\}$, we denote $\pi_0(I) \ceq p$. Then \eqref{eq:2P:groP} induces the $\bbZ^2$-grading $\frg = \bigoplus_{p,q \in \bbZ^2}\frg^{p,q}$ by 
\begin{align}\label{eq:AP:gpq}
 \frg^{p,q} \ceq \{f \in \frg^{p+q} \mid f^{\Lambda_I}=0 \text{ for any partition $I$ of $\{1,\dotsc,n\}$ with $\pi_0(I) \neq p$}\}.
\end{align}
We have $\frg^{n-1}=\bigoplus_{p=1}^n \frg^{p,n-1-p}=\frg^{1,n-2} \oplus \dotsb \oplus \frg^{n,-1}$ for $n \in \bbZ_{\ge 1}$.
\end{lem}

\begin{proof}
The claim follows from \cref{fct:AP:clL} and the fact that $\pi_0(I) = p$ if and only if $\# E(\Lambda_I)=n-p$ for a partition $I$ of $\{1,\dotsc,n\}$. 
\end{proof}

We can identify the doubly-graded linear spaces $\frg^{\dbl}$ and $C_{\tuPoi}^{\dbl}$ in the following way.

\begin{lem}
We have 
\begin{align}\label{eq:AP:gpq=Cpq+1}
 \frg^{p,q} \simeq C^{p,q+1}_{\tPoi}(A,A)
\end{align}
as linear spaces for any $p,q \in \bbZ$.
\end{lem}

\begin{proof}
We may assume $p \in \bbZ_{>0}$ and $q+1 \in \bbN$. Let us denote by $\ch_{\bl} \ceq \clL^c(\Sigma A_{\tCom})_{\bl}$ the Harrison chain complex of the commutative algebra $A_{\tCom}=(V,-\cdot-)$.
Then, \cref{dfn:AP:pcb} claims
\begin{align*}
&C^{p,q+1}_{\tPoi}(A,A) 
=\Hom\bigl(\clS_p(\ch_{\bl}),V\bigr)^{p+q+1} = \bigoplus_J \Hom(\ch_J,V),
\end{align*}
where the summation index $J=\{j_1,\dotsc,j_l\}$ runs over $l \in \bbZ_{>0}$, $j_k \in \bbZ_{>0}$ ($k=1,\dotsc,l$), $\sum_{k=1}^l j_k=p$ and $\sum_{k=1}^l k j_k=p+q+1$, and 
\[
 \ch_J \ceq 
 \wedge^{j_1} \ch_1 \otimes \wedge^{j_2} \ch_2 \otimes \dotsb \otimes \wedge^{j_l}\ch_l.
\]
Such $J$ is in one-to-one correspondence with the equivalent class of a partition $I$ of $\{1,\dotsc,p+q+1\}$ with $\pi_0(I)=p$ under the natural $\frS_n$-action. Recalling that we have $\ch_\bl \simeq \clL^c(\Sigma A_{\tCom}) \simeq e^{(1)}T^c(\Sigma A_{\tCom})$ by \eqref{eq:AP:CC}, we find that this bijection induces 
\[
 \Hom(\ch_J,V) \lsto \left\{f=f^I \in \Hom(\bbK \Lambda_I \otimes V^{\otimes n},V) \mid 
 \forall \sigma \in \frS_n, \ f^\sigma = \sgn(\sigma) f\right\}.
\]
Since $\clL(n)$ is a basis of $\bbK \oQ(n)/R(n)$, we see that $\bigoplus_J \Hom(\ch_J,V)$ is isomorphic to the right hand side of \eqref{eq:AP:gpq}. Thus we have the conclusion.
\end{proof}

As for the differential $d_X=[X,-]$ of $\frg$, we have the following \cref{lem:AP:d1d2},
which will finish the proof of \cref{thm:AP:fcc=pcb}.

\begin{lem}\label{lem:AP:d1d2}
Using the basis \eqref{eq:AP:fG2}, we decompose $X \in \frg^1$ as 
\[
 X = X^{\{1\} \sqcup \{2\}} + X^{\{1,2\}}, \quad 
 X^{\{1\} \sqcup \{2\}} \in \frg^{2,-1}, \quad X^{\{1,2\}} \in \frg^{1,0}.
\]
Then the induced decomposition 
\[
 d_X=d_h+d_v \quad \text{with } 
 d_h \ceq [X^{\{1\} \sqcup \{2\}},-], \quad d_v \ceq [X^{\{1,2\}},-]
\]
satisfies 
\[
 d_h(\frg^{p,q}) \subset \frg^{p+1,q}, \quad d_v(\frg^{p,q}) \subset \frg^{p,q+1}.
\]
Moreover, under the identification $\frg^{p,q} \simeq C^{p,q+1}_{\tPoi}(A,A)$ in \eqref{eq:AP:gpq=Cpq+1} and $X_1=\{-,-\}$, $X_2=-\cdot-$ in \eqref{eq:AP:X=P}, we have
\[
 (\frg^{\dbl},d_h,d_v) \simeq (C_{\tuPoi}^{\dbl+1}(A,A),\delta_{\tCE}, d)
\]
as bicomplexes.
\end{lem}

Before starting the proof, we note that the part $d_v=d$ is due to \cite[Lemma 4.10]{BDKV21} (c.f.\ \cref{fct:AP:dHar}).

\begin{proof}
For simplicity, let us denote $X_h \ceq X^{\{1\} \sqcup \{2\}}$ and $X_v \ceq X^{\{1,2\}}$.
We want to calculate $d_i(Y)=[X_i,Y]=X \square Y - (-1)^{n-1} Y \square X_i$ for $i=h,v$, $Y \in \frg^{p,q} \subset \frg^{n-1}$ with $n \ceq p+q+1$. By \cref{fct:AP:clL}, we may assume $Y^\Gamma=0$ unless $\Gamma=\Lambda_I$ where $I$ is a partition of $\{1,\dotsc,n\}$ with $\pi_0(I)=p$. Let us set 
\[
 I = I_1 \sqcup \dotsb \sqcup I_p, \quad I_a \ceq \{i^a_1,\dotsc,i^a_{m_a}\} \quad (a=1,\dotsc,p).
\]

Recalling the description \eqref{eq:AP:square} of the pre-Lie product $\square$, 
we find that $X_i \square Y \in \frg^n$ is expressed as 
\[
 X_i \square Y = \sum_{\sigma \in \Sh(n,1)} (X_i \circ_1 Y)^{\sigma^{-1}},
\]
where each $\sigma \in \Sh(n,1)$ can be written as 
\[
 \sigma = \begin{pmatrix}
  1 & \cdots &        & \cdots & n   & n+1 \\ 
  1 & \cdots & \wh{j} & \cdots & n+1 & j   \end{pmatrix} \quad (1 \le j \le n+1).
\]
Recalling the parity change $\pc$ in $\frg$ (c.f.\ \eqref{eq:AP:gfnVn}) and unfolding the infinitesimal composition $\circ_1$, we have
\begin{align}\label{eq:AP:XiY}
\begin{split}
 (X_i \circ_1 Y)^{\sigma^{-1}}(v_1 \otimes \dotsb \otimes v_{n+1}) 
&=\sgn(\sigma) \cdot \bigl(X_i(Y \otimes \id)\bigr)
   (v_{\sigma(1)} \otimes \dotsb \otimes v_{\sigma(n+1)}) \\
&=\sgn(\sigma) \cdot 
   X_i\bigl(Y(v_{\sigma(1)} \otimes v_{\sigma(2)} \otimes \dotsb \otimes v_{\sigma(n)})
           \otimes v_{\sigma(n+1)} \bigr)
\end{split}
\end{align}
for $v_1,\dotsc,v_{n+1} \in V$. Similarly, we have $Y \square X_i \in \frg^n$ with
\[
 Y \square X_i  = \sum_{\tau \in \Sh(2,n-1)} (Y \circ_1 X_i)^{\tau^{-1}}, 
\]
and each $\tau \in \Sh(2,n-1)$ can be written as 
\[
 \tau = \begin{pmatrix}
  1 & 2 & \cdots &        & \cdots &        & \cdots & n+1 \\ 
  j & k & \cdots & \wh{j} & \cdots & \wh{k} & \cdots & n+1 \end{pmatrix} 
 \quad (1 \le j < k \le n+1).
\]
The infinitesimal composition $(Y \circ_1 X_i)^{\tau^{-1}}$ is given by 
\begin{align}\label{eq:AP:YXi}
\begin{split}
&(Y \circ_1 X_i)^{\tau^{-1}}(v_1 \otimes \dotsb \otimes v_{n+1}) \\
&=\sgn(\tau) \cdot \bigl(Y(X_i \otimes \underbrace{\id \otimes \dotsb \otimes \id}_{n-1})\bigr)
  (v_{\tau(1)} \otimes \dotsb \otimes v_{\tau(n+1)}\bigr) \\
&=\sgn(\tau) \cdot Y\bigl(X_i(v_{\tau(1)} \otimes v_{\tau(2)}) \otimes
                              v_{\tau(3)} \otimes \dotsb \otimes v_{\tau(n+1)}\bigr).
\end{split}
\end{align}

Now we focus on $X_v=X^{\{1,2\}}$. By the assumption $Y^\Gamma=0$ unless $\Gamma=\Lambda_I$, and by the description of the cocomposition map \eqref{eq:AP:coc} and \eqref{eq:2P:Pfn:gamma}, we have $(X_v \circ_1 Y)^\Gamma=0$ unless $\Gamma=\Lambda_{J^v(a)}$, $a=1,\dotsc,n$, where 
$J^v(a)$ is a partition of $\{1,\dotsc,n+1\}$ obtained from $I$ by replacing the entry $a$ by $a$ and $n+1$. In the case $\Gamma=\Lambda_{J^v(a)}$, we have 
\[
 (X_v \circ_1 Y)^{\Lambda_{J^v(a)}} = \pm X_v(Y \otimes \id).
\]
Similarly, we have $(Y \circ_1 X_v)^\Gamma=0$ unless $\Gamma=\Lambda_{K^v}$, where $K^v$ is a partition of $\{1,\dotsc,n+1\}$ obtained from $I$ by replacing the entry $1$ by $1$ and $2$, and the entries $2,\dotsc,n$ by $3,\dotsc,n+1$, respectively. If $\Gamma=\Lambda_{K^v}$, then 
\[
 (Y \circ_1 X_v)^{\Lambda_{K^v}} = Y(X_v \otimes \id^{\otimes (n-1)}).
\]
Thus we find $(X_v \square Y)^\Gamma=0$ and $(Y \square X_v)^\Gamma=0$ unless $\Gamma \in \frg^{p,q+1}$, which yields $d_v(\frg^{p,q}) \subset \frg^{p,q+1}$.
Moreover, by \eqref{eq:AP:XiY} and \eqref{eq:AP:YXi}, and after the careful calculation on the sign shown in \cite[Lemma 4.10]{BDKV21}, we have 
\begin{align*}
& d_v(Y)(v_1 \otimes \dotsc \otimes v_{n+1}) 
 =(-1)^{n+1}\Bigl( X_v\bigl(v_1 \otimes Y(v_2 \otimes \dotsb \otimes v_{n+1})\bigr) \\
&  + \sum_{j=1}^n (-1)^j Y\bigl(v_1 \otimes \dotsb \otimes X_v(v_j \otimes v_{j+1}) 
                                    \otimes \dotsb \otimes v_{n+1}\bigr)
   + (-1)^{n+1}X_v\bigl(Y(v_1 \otimes \dotsb \otimes v_n) \otimes v_{n+1}\bigr)\Bigr).
\end{align*}
Under the identification $\frg^{\dbl} \simeq C^{\dbl+1}_{\tPoi}(A,A)$ and $X_v=-\cdot-$, 
this is nothing but the Harrison (or Hochschild) coboundary map, i.e.,
the vertical differential $d$ in $C_{\tuPoi}^{\dbl}(A,A)$.

Next, for $X_h=X^{\{1\} \sqcup \{2\}}$, we have $(X_h \circ_1 Y)^\Gamma=0$ unless $\Gamma=\Lambda_{J^h}$ with $J^h \ceq J \sqcup \{n+1\}$, and 
\[
 (X_h \circ_1 Y)^{\Lambda_{J^h}} = X_h(Y \otimes \id).
\]
We also have $(Y \circ_1 X_h)^\Gamma=0$ unless $\Gamma=\Lambda_{K^h}$, where $K^h$ is the partition of $\{1,\dotsc,n+1\}$ given by 
\[
 K^h \ceq \{1\} \sqcup I_1'' \sqcup \dotsb \sqcup I_p'', \quad 
 I_a'' \ceq \{i^a_1+1,\dotsc,i^a_{m_a}+1\} \quad (a=1,\dotsc,p),
\]
and in the case $\Gamma=\Lambda_{K^h}$, we have 
\[
 (Y \circ_1 X_h)^{\Lambda_{K^h}} = 
 Y(X_h \otimes \id^{\otimes (n-1)}).
\]
Hence we have $(X_h \square Y)^\Gamma=0$ and $(Y \square X_h)^\Gamma=0$ unless $\Gamma \in \frg^{p+1,q}$, which yields $d_h(\frg^{p,q}) \subset \frg^{p+1,q}$.
By \eqref{eq:AP:XiY} and \eqref{eq:AP:YXi}, we have 
\begin{align*}
  d_h(Y)(v_1 \otimes \dotsc \otimes v_{n+1}) 
&=\sum_{j=1}^{n+1} (-1)^{j-1} 
   X_h\bigl(Y(v_1 \otimes \dotsb \wh{v}_j \dotsb \otimes v_{n+1}) \otimes v_j\bigr) \\
&-(-1)^{n-1} \sum_{1 \le j<k \le n+1} (-1)^{j+k} Y\bigl(X_h(v_j \otimes v_k) \otimes 
             v_1 \otimes \dotsb \wh{v}_j \dotsb \wh{v}_k \dotsb \otimes v_{n+1}\bigr) \\
&=(-1)^n\biggl(\sum_{j=1}^{n+1} (-1)^{j-1} X_h\bigl(v_j \otimes 
              Y(v_1 \otimes \dotsb \wh{v}_j \dotsb \otimes v_{n+1})\bigr) \\
&            +\sum_{1 \le j<k \le n+1} (-1)^{j+k} Y\bigl(X_h(v_j \otimes v_k) \otimes 
              v_1 \otimes \dotsb \wh{v}_j \dotsb \wh{v}_k \dotsb \otimes v_{n+1}\bigr)\biggr),
\end{align*}
and under the identification $\frg^{\dbl} \simeq C_{\tuPoi}^{\dbl+1}(A,A)$ and $X_h=\{-,-\}$, this is nothing but the Chevalley-Eilenberg coboundary, i.e., the horizontal differential $\delta_{\tCE}$ in $C_{\tuPoi}^{\dbl}(A,A)$.
\end{proof}



\begin{thebibliography}{MMMMN}

\bibitem[BDHK19]{BDHK}
B.~Bakalov, A.~De Sole, R.~Heluani, V.~G.~Kac, 
\emph{An operadic approach to vertex algebra and Poisson vertex algebra cohomology}, 
Jpn.\ J.\ Math., \textbf{14}, 249--342 (2019).

\bibitem[BDHK20]{BDHK2}
B.~Bakalov, A.~De Sole, R.~Heluani, V.~G.~Kac, 
\emph{Chiral versus classical operad}, 
Int.\ Math.\ Res.\ Not.\ IMRN, \textbf{2020}, no.\ 19, 6463--6488 (2020).

\bibitem[BDK20]{BDK20}
B.~Bakalov, A.~De Sole, V.~G.~Kac, 
\emph{Computation of cohomology of Lie conformal and Poisson vertex algebras},
Sel.\ Math.\ New Series, 26:50, 51 pp.\ (2020).

\bibitem[BDK21]{BDK21}
B.~Bakalov, A.~De Sole, V.~G.~Kac, 
\emph{Computation of cohomology of vertex algebras},
Jpn.\ J.\ Math., \textbf{16}, 81--154 (2021).

\bibitem[BDKV21]{BDKV21}
B.~Bakalov, A.~De Sole, V.~G.~Kac, V.~Vignoli,
\emph{Poisson Vertex Algebra Cohomology and Differential Harrison Cohomology},
in \emph{Representation Theory, Mathematical Physics, and Integrable Systems}, 
Progress in Math., vol.\ \textbf{340}, Birkh\"{a}user, pp.\ 39--69 (2021).


\bibitem[BD04]{BD}
A.~Beilinson, V.~Drinfeld, 
\emph{Chiral Algebras}, 
AMS Colloq.\ Publ., \textbf{51}, Amer.\ Math.\ Soc., Province RI (2004).





\bibitem[FBZ04]{FBZ}
E.~Frenkel, D.~Ben-Zvi,
\emph{Vertex algebras and algebraic curves}, 2nd ed.,
Math.\ Surv.\ Monog., \textbf{88}, Amer.\ Math.\ Soc., Province RI (2004).

\bibitem[F06]{Fr}
B.~Fresse, 
\emph{Th\'{e}orie des op\'{e}rades de Koszul et homologie des alg\`{e}bres de Poisson},
Ann.\ Math.\ Blaise Pascal, \textbf{13} (2006), no.\ 2, 237--312.

\bibitem[GK94]{GK}
V.~Ginzburg, M.~Kapranov,
\emph{Koszul duality for operads},
Duke Math.\ J., \textbf{76}, no.\ 1, 203--272 (1994).

\bibitem[GS87]{GS}
M.~Gerstenhaber, S.~D.~Schack, 
\emph{A Hodge-type decomposition for commutative algebra cohomology}, 
J.\ Pure Appl.\ Alg., \textbf{48} (1987), 229--247.

\bibitem[Ha62]{Ha}
D.~K.~Harrison, \emph{Commutative algebras and cohomology},
Trans.\ Amer.\ Math.\ Soc., \textbf{104} (1962), 191--204.


\bibitem[HK07]{HK}
R.~Heluani, V.~G.~Kac,
\emph{Supersymmetric Vertex Algebras}, 
Commun.\  Math.\ Phys., \textbf{271}, 103--178 (2007).

\bibitem[K17]{K}
V.~Kac, \emph{Introduction to Vertex Algebras, Poisson Vertex Algebras, and Integrable Hamiltonian PDE}, 
in: \emph{Perspectives in Lie Theory}, Springer INdAM Series, vol \textbf{19}. Springer, 3--72 (2017);
arXiv:1512.00821. 



\bibitem[Li04]{Li1} 
H.~Li, \emph{Vertex algebras and vertex Poisson algebras}. 
Commun.\ Contemp.\ Math., 6 (2004) no.\ 1, 61--110.

\bibitem[Li05]{Li2} 
H.~Li, \emph{Abelianizing vertex algebras}. 
Commun.\ Math.\ Phys., \textbf{259} (2005), 391--411.

\bibitem[Lo89]{L1}
J.-L.~Loday, \emph{Op\'erations sur l'homologie cyclique des alg\`ebres commutatives}, 
Invent.\ Math., \textbf{96} (1989), 205--230.

\bibitem[Lo98]{L}
J.-L.~Loday, \emph{Cyclic homology}, 2nd ed., 
Grund.\ math.\ Wiss., \textbf{301}, Springer-Verlag, 1998.

\bibitem[LV12]{LV}
J-L.~Loday, B.~Vallette, \emph{Algebraic Operads}, 
Grundlehren Math.\ Wiss., \textbf{346}, Springer (2012).

\bibitem[N08]{N}
Y.~Namikawa, \emph{Flops and Poisson Deformations of Symplectic Varieties},
Publ.\ RIMS, Kyoto Univ., \textbf{44} (2008), 259--314.

\bibitem[NY]{NY}
Y.~Nishinaka, S.~Yanagida, 
\emph{Algebraic operad of SUSY vertex algebra}, preprint (2022),
arXiv:2209.14617. 

\bibitem[SS85]{SS}
M.~Schlessinger, J.~Stasheff, 
\emph{The Lie algebra structure of tangent cohomology and deformation theory}, 
J.\ Pure Appl.\ Alg., \textbf{38} (1985), no.\ 2--3, 313--322.

\bibitem[Y22]{Y}
S.~Yanagida, \emph{Li filtrations of SUSY vertex algebras}, 
Lett.\ Math.\ Phys., \textbf{112} (2022), ID:103, 77pp.
\end{thebibliography}
\end{document}